\documentclass[10pt,a4paper]{amsart}
\usepackage{amssymb,amsmath,amsthm,amscd}
\usepackage{leftidx,soul}
\usepackage{graphicx}
\usepackage{color,mathabx}
\usepackage{hyperref,url}
\usepackage{tikz,enumitem}
\usepackage{comment}
\usepackage{caption}
\usepackage[dvipsnames]{xcolor}
\usepackage{subcaption}
\usepackage{xfrac,faktor}
\usepackage{epstopdf} 
\usepackage{float}
\usepackage{tikz-cd}
\usepackage[all]{xy}
\usepackage{pgfplots}
\pgfplotsset{compat=1.16}

\DeclareMathAlphabet\mathbfcal{OMS}{cmsy}{b}{n}
 
\newtheorem{lemma}{Lemma}[section]
\newtheorem{theorem}[lemma]{Theorem}
\newtheorem{corollary}[lemma]{Corollary}
\newtheorem{proposition}[lemma]{Proposition}

\newtheorem{question}[lemma]{Question}
\theoremstyle{definition}
\newtheorem{definition}[lemma]{Definition}
\theoremstyle{remark}
\newtheorem{remark}{Remark}

\theoremstyle{definition}
\newtheorem{thmx}{Theorem}

\newcommand{\C}{\mathbb{C}}
\newcommand{\D}{\mathbb{D}}
\newcommand{\N}{\mathbb{N}}

\newcommand{\R}{\mathbb{R}}
\newcommand{\Z}{\mathbb{Z}}

\newcommand{\cC}{\mathcal{C}}
\newcommand{\cH}{\mathcal{H}}
\newcommand{\cF}{\mathcal{F}}
\newcommand{\cK}{\mathcal{K}}
\newcommand{\cM}{\mathcal{M}}
\newcommand{\cR}{\mathcal{R}}
\newcommand{\cU}{\mathcal{U}}
\newcommand{\cS}{\mathcal{S}}
\newcommand{\cT}{\mathcal{T}}

\DeclareMathOperator{\Int}{int}

\renewcommand{\epsilon}{\varepsilon}
\renewcommand{\phi}{\varphi}
\renewcommand{\theta}{\vartheta}

\DeclareMathOperator{\Deg}{deg}

\date{\today}

\begin{document}

\title[Antiholomorphic correspondences and mating II]{Antiholomorphic correspondences and mating II: Shabat polynomial slices}

\author[M. Lyubich]{Mikhail Lyubich}
\address{Institute for Mathematical Sciences, Stony Brook University, NY, 11794, USA}
\email{mlyubich@math.stonybrook.edu}
\thanks{M.L. was partially supported by NSF grants DMS-1901357 and 2247613.}

\author[J. Mazor]{Jacob Mazor}
\email{n.jacobmazor@gmail.com}
\thanks{}

\author[S. Mukherjee]{Sabyasachi Mukherjee}
\address{School of Mathematics, Tata Institute of Fundamental Research, 1 Homi Bhabha Road, Mumbai 400005, India}
\email{sabya@math.tifr.res.in, mukherjee.sabya86@gmail.com}
\thanks{S.M. was supported by the Department of Atomic Energy, Government of India, under Project Identification No. RTI 4014, an endowment of the Infosys Foundation, and ANRF research project grants MTR/2022/000248 and ANRF/ARGM/2025/000089/MTR}

\begin{abstract}
We study natural one-parameter families of antiholomorphic correspondences arising from Shabat polynomials, indexed by rooted dessins d'enfants. The actions of our correspondences are governed by the dynamics of certain Schwarz reflection maps associated with univalent restrictions of the Shabat polynomials. We prove that the parameter spaces are topological quadrilaterals, giving a partial description of the univalency loci for the Shabat polynomials. We show that the escape loci of our parameter spaces admit natural real-analytic parametrizations by disks. We proceed with designing a puzzle structure (dual to the indexing dessin) for non-renormalizable maps, yielding combinatorial rigidity in these classes. Then we develop a renormalization theory for pinched (anti-)polynomial-like maps in order to describe all combinatorial Multibrot and Multicorn copies contained in our connectedness loci (a curious feature of these parameter spaces is the presence of multiple period-one copies). Finally, we construct locally connected combinatorial models for the connectedness loci into which the indexing dessins naturally embed.
\end{abstract}

\maketitle

\setcounter{tocdepth}{1}
\tableofcontents

\section{Introduction}\label{intro}

\subsection{Preamble}
Over the past decade, a new area of research, dynamics generated by Schwarz reflections in
quadrature domains, has emerged and flourished (see survey \cite{LM23}).
It gave a natural framework for an explicit  unification
of the two parts  of the Fatou-Sullivan  Dictionary
between dynamics of (anti-)rational maps and actions of Kleinian (reflection) groups.
Namely, under certain circumstances, an (anti-)rational map can be {\em mated} with a
Kleinian reflection group to produce a Schwarz dynamics.
On some part of the phase space, called the {\em tiling set}, this dynamics looks like a group action,
while on the other, called the {\em non-escaping set}, it looks like dynamics of an anti-rational map (i.e., complex conjugate of a rational map).
Moreover, it can often be lifted to an ({\em anti-)algebraic  correspondence}
with similar features, realizing in this way Fatou's vision of a unified theory
of holomorphic dynamics (see \cite{Fat29}).

In our previous paper \cite{PartI}, this scenario was described in the case of anti-Hecke groups and  Bers-like parabolic anti-rational maps; these are anti-rational maps with a simply connected and fully invariant marked immediate parabolic basin, named in analogy with Kleinian groups on boundaries of Bers slices. 
The resulting Schwarz reflections are organized in varieties split into the bifurcation and stability loci in the spirit of the
familiar pictures for families of (anti-)rational maps.
And in fact, it was shown that there is a close explicit connection between these pictures.

In (anti-)holomorphic dynamics a prominent role has been played by natural
one-parameter slices of maps, like the complex quadratic or anti-quadratic families
producing the Mandelbrot and Tricorn sets.
They demonstrate various deep features in a visual way,
with many additional topological and analytic tools at one's disposal.
In particular, the description of the bifurcation loci become much more explicit for such families.

Previously, such a program was carried out for two families of quadratic Schwarz reflections,
the so called  {\em $C\& C$ family} and the {\em cubic Chebyshev family} \cite{LLMM1,LLMM2,LLMM3}.
In this paper we introduce  an infinite collection of natural one-parameter families,
associated with maps referred to as {\em Shabat polynomials},
and describe in quite a detail the structure of their bifurcation loci. 
The simplest example of such a one-parameter family is furnished by the cubic Chebyshev polynomial.

\subsection{Bijection between three types of conformal dynamical systems}\label{bijection_subsec}
The main result of \cite{PartI} provides a bijection between three different families of dynamical systems, which we now describe.

\subsubsection{Parabolic anti-rational maps $\cF_d$} The first of them is a family $\cF_d$ of degree $d\geq 2$ anti-rational maps $R$ with a connected, simply connected parabolic basin where $R$ is conformally conjugate to a canonical unicritical parabolic anti-Blaschke product (see \S~\ref{para_anti_rat_subsec} for a precise definition).

\subsubsection{Schwarz reflection maps $\cS_{\cR_d}$} The second object is a space of Schwarz reflection maps with connected non-escaping set generated by degree $d+1$ polynomials that are injective on $\overline{\D}$ and have a unique critical point on $\mathbb{S}^1$. It turns out that such Schwarz reflections present pinched anti-polynomial-like restrictions whose escaping dynamics is modeled by a piecewise-analytic parabolic circle covering $\mathcal{R}_d$ (in other words, $\mathcal{R}_d$ is the external map of the pinched anti-polynomial-like restrictions). The circle map $\mathcal{R}_d$ arises from the anti-Hecke group $\pmb{\Gamma}_d$, which is a discrete subgroup of the group of all
conformal and anti-conformal automorphisms of $\D$ generated by the rigid rotation by angle $2\pi/(d+1)$ and the reflections in the sides of a regular ideal hyperbolic $(d+1)-$gon in the disk. This space of Schwarz reflections is denoted by $\cS_{\cR_d}$ (see \S~\ref{antiFarey_subsec} for details).

\subsubsection{Antiholomorphic correspondences} The third family in question is that of antiholomorphic correspondences. Specifically, it is the
connectedness locus of the space of antiholomorphic correspondences

\begin{equation}
z\mapsto w \iff \frac{f(w)-f(1/\overline{z})}{w-1/\overline{z}}=0,
\label{corr_eqn}
\end{equation}
generated by the circular reflection $\eta(z)=1/\overline{z}$ and the local deck transformations of degree $d+1$ polynomials $f$ that are injective on $\overline{\D}$ and have a unique critical point on $\mathbb{S}^1$. Such correspondences are obtained by lifting the Schwarz reflections in $\cS_{\cR_d}$ by the polynomials $f$. We denote this space of correspondences by $\widehat{\cS_{\cR_d}}$.

According to \cite[Theorems~A,~B]{PartI}, there is a natural bijection between the spaces $\cF_d$, $\cS_{\cR_d}$, and $\widehat{\cS_{\cR_d}}$ such that for each $R\in\cF_d$, the corresponding Schwarz reflection in $\cS_{\cR_d}$ (respectively, the antiholomorphic correspondence in $\widehat{\cS_{\cR_d}}$) is a quasiconformal mating between $R$ and the external map $\cR_d$ (respectively, between $R$ and the anti-Hecke group $\pmb{\Gamma}_d$). In particular, the anti-rational map $R\in\cF_d$ and the corresponding Schwarz reflection in $\cS_{\cR_d}$ are hybrid conjugate in neighborhoods of their non-escaping sets.

\subsection{Natural one-parameter families: Shabat polynomial slices}
 In this paper, we study one-parameter slices of $\cS_{\cR_d}$ such that the polynomials whose univalent restrictions to $\D$ generate the Schwarz reflections are \emph{Shabat}; i.e., they have two critical values in the plane (cf. \cite{BZ96,LZ04}). Affine equivalence classes of Shabat polynomials are completely determined by certain combinatorial bicolored plane trees called \emph{dessins d'enfants}. It turns out that the subspace of $\cS_{\cR_d}$ generated by Shabat polynomials having a fixed dessin is real two-dimensional. We extend this slice of $\cS_{\cR_d}$ in a natural way so as to include Schwarz reflections with disconnected non-escaping set. It turns out that the Schwarz reflections in this extended slice also have three critical values: a passive critical value in the tiling set, another passive critical value at the unique cusp of the quadrature domain, and a free/active critical value. The branched covering structure of such Schwarz reflections is also encoded in a combinatorial bicolored plane tree $\cT$, which is closely related to the dessin of the generating Shabat polynomial. We denote this space of Schwarz reflections by $S_\cT$ (see \S~\ref{belyi_schwarz_subsec}). In other words, the real two-dimensional parameter space $S_\cT$ of Schwarz reflections intersects $\cS_{\cR_d}$ in its connectedness locus $\cC(S_\cT)$.

The slice of $\cF_d$ that corresponds to the connectedness locus of $S_\cT$ under the bijection of \S~\ref{bijection_subsec} is given by the collection of \emph{Belyi} anti-rational maps (anti-rational maps with three critical values in $\widehat{\C}$) in $\cF_d$ such that the persistently parabolic fixed point is a passive/confined critical value and the associated dessin d'enfants is $\cT$ (see \S~\ref{belyi_anti_rat_subsec}). Note that Shabat polynomials are particular examples of Belyi rational maps. In this paper, we reserve the term \emph{Shabat} for the non-dynamical polynomials that uniformize the quadrature domains defining the Schwarz reflection maps and the term \emph{Belyi} for dynamical anti-rational maps or Schwarz reflections.

\subsection{Main results}
Our first main theorem gives a complete description of the interior and the boundary of the parameter space $S_{\mathcal{T}}$. We use a combination of local and global properties of Schwarz reflection dynamics to study the space of Shabat polynomials that are injective on the closed unit disk. This leads to the following result.

\begin{thmx}\label{para_space_jordan_thm_intro}
The parameter space $S_\cT$ has a bounded Jordan domain interior and a closure which is a topological quadrilateral. Furthermore, the boundary arcs are all real analytic and can be described as
\begin{enumerate}[leftmargin=8mm]
    \item $\Gamma^\perp$, a straight line segment of parameters for which a free critical point escapes the quadrature domain in one step,
    \item $\Gamma^{\mathrm{hoc}}$, an arc for which the quadrature domain has a higher order cusp, and
    \item $\Gamma^{\mathrm{dp}}$, a pair of arcs for which the quadrature domain has a double point on the boundary.
\end{enumerate}
\end{thmx}

We refer the reader to \S~\ref{conn_para_space_sec} for the precise definitions of the sets $\Gamma^{\mathrm{hoc}}$, $\Gamma^{\perp}$, and $\Gamma^{\mathrm{dp}}$. The proof of Theorem~\ref{para_space_jordan_thm_intro} is given in Theorem~\ref{bdry_thm} and Theorem~\ref{para_space_conn_thm}.

The real two-dimensionality of $S_\cT$ is related to the fact that each Schwarz reflection in this family has a unique free critical value. We refer to the complement of the connectedness locus $\cC(S_\cT)$ in $S_\cT$ as the \emph{escape locus}, and note that for maps in the escape locus, the free critical value eventually escapes the quadrature domain and never lands on the unique cusp on its boundary. The conformal model $\cR_d$ of the escaping dynamics of Schwarz reflections in $S_\cT$ allows us to record the conformal position of the escaping free critical value. We leverage this to give an explicit uniformization of the escape locus.

\begin{thmx}\label{escape_unif_thm_intro}
The conformal position of the critical value defines a homeomorphism from the escape locus $S_{\mathcal{T}}\setminus\cC(S_\cT)$ onto a simply connected domain.
In particular, the connectedness locus is connected.
\end{thmx}
\noindent (See Theorem~\ref{escape_unif_thm} for a precise statement.)

We move on to develop puzzle pieces and a renormalization theory for Schwarz reflections in $S_\cT$. Since the maps in $S_\cT$ have one free critical value, it comes as no surprise that the connectedness locus contains many little `copies' of appropriate Multibrot and Multicorn sets.

\begin{thmx}\label{renorm_locus_thm_intro}
The connectedness locus $\cC(S_\cT)$ contains infinitely many homeomorphic copies of Multibrot sets and combinatorial copies of Multicorn sets (which are the connectedness loci of unicritical holomorphic and antiholomorphic polynomials). 
Any parameter not contained in these homeomorphic/combinatorial copies of Multibrot/Multicorn sets is combinatorially rigid.
\end{thmx}
\noindent (See Theorem~\ref{multicorn_thm} and Theorem~\ref{PinchedMulticorn} for precise statements.)

Interestingly there are multiple period $1$ renormalization loci that are combinatorially equivalent to Multicorn sets, a departure from the unicritical case.

In light of discontinuity of straightening maps in antiholomorphic dynamics (cf. \cite{IM21}), the combinatorial copies of Multicorns in $\cC(S_\cT)$ may not be upgraded to genuine homeomorphic copies, at least via dynamically natural homeomorphisms.

Finally, Theorem~\ref{escape_unif_thm_intro} gives a tessellation of the escape locus by dynamically natural tiles, which comes from the tiling of the hyperbolic plane under the action of the anti-Hecke group. This, in turn, permits one to construct external parameter rays as dual to the dynamically natural tessellation; equivalently, the external parameter rays of $S_\cT$ arise from a Cayley graph of the anti-Hecke group.  
The above structure of the escape locus can be regarded as an analog of the foliation of the exterior of the Mandelbrot set or the Tricorn by external rays and equipotentials. This facilitates the implementation of Douady's philosophy `plough in the dynamical plane and harvest in the parameter plane' in our family. Specifically, we study accumulation properties of external parameter rays of $S_\cT$ (which exhibit subtle differences from the classical unicritical setting), puzzle pieces and combinatorial rigidity arguments to provide a topological description of the connectedness locus $\cC(S_\cT)$.

\begin{thmx}\label{conn_locus_model_thm_intro}
There exists a model $\widetilde{\cC_\cT}$ of $\cC(S_\cT)$ given by collapsing appropriate combinatorial classes, and this model is homeomorphic to the quotient of a disk by a closed geodesic lamination. This lamination is given by taking appropriate pullbacks of the dual lamination to $\cT$, taking the closure, and then ``filling in the gaps" by appropriately tuned laminations for Multibrot and Multicorn sets.
Furthermore, there is a natural embedding of the dessin $\cT$ in this abstract connectedness locus.
\end{thmx}
\noindent (See Theorem~\ref{comb_model_thm} and Theorem~\ref{dessin_in_conn_locus_thm} for precise statements.)

\subsection{Special features of the family $S_\cT$ and comments on the proofs}\label{special_features_subsec}
Although Schwarz reflections in the family $S_\cT$ have a unique free critical value, they exhibit various fine differences from the standard unicritical (anti-)polynomial families. The fact that Schwarz reflections are semi-global objects is also a source of key differences between the dynamics/parameter spaces of (anti-)polynomial and Schwarz reflections.

\subsubsection{Non-dynamical Shabat polynomials and dynamics of Belyi maps}
The primary objects of this paper, the Schwarz reflection maps, are generated by univalent restrictions of Shabat polynomials $f$ to round disks. Consequently, the Schwarz reflections $\sigma\in\cS_{\cR_d}$ themselves have three critical values, and hence are examples of Belyi maps. Naturally, there is an explicit relation between the dessin of $f$ and that of $\sigma$. Specifically, the unique cusp on the quadrature domain boundary acts as a root for the dessin $\cT$ of $\sigma$, and the dessin of $f$ is obtained by adding an edge to (the mirror image of) $\cT$ at the root point. 

\Large\begin{equation*}
\bigg\{\substack{\mathrm{Shabat}\\ \mathrm{polynomials}}\bigg\}\ \longleftrightarrow \bigg\{\substack{\mathrm{Belyi\ Schwarz}\\ \mathrm{reflections}}\bigg\}\ \longleftrightarrow \bigg\{\substack{\mathrm{Bers-Belyi\ anti-rational}\\ \mathrm{maps}}\bigg\} 
\end{equation*}
\normalsize On the other hand, the parabolic anti-rational map $R\in\cF_d$ that $\sigma$ is hybrid conjugate to has the same dessin $\cT$. These combinatorial relations between the non-dynamical Shabat polynomials and the dynamical systems given by Belyi Schwarz reflections and Belyi anti-rational maps are established and exploited in \S~\ref{shabat_poly_schwarz_sec}.

\subsubsection{Quadrature boundary degenerations}
The family $S_\cT$ is defined by restrictions of Shabat polynomials to maximal disks of univalence. Gaining a complete understanding of the collection of such round disks for a general complex polynomial is a non-trivial problem (see \cite{Bra67,CR68,Suf72,She00} for various partial results in low degree and for restricted classes of polynomials). It was reasonable to believe that in our family, a `phase transition' from univalence to non-univalence occurs when the associated quadrature domain develops either a higher order cusp or a double point. 
We employ dynamical techniques, especially a local dynamical analysis of Schwarz reflections near conformal cusps and double points, to justify this heuristic (cf. \cite{Sak91}). This yields a complete picture of the interior and the boundary of $S_\cT$ (Theorem~\ref{para_space_jordan_thm_intro}), which  plays a crucial role in the proof of the uniformization of the escape locus and the ensuing tiling structure (Theorem~\ref{escape_unif_thm_intro}).

\subsubsection{A blend of unicriticality and multicriticality}
The existence of a single free critical value and the real two-dimensionality of the parameter space $S_\cT$ make a number of techniques from unicritical polynomial dynamics available in our setting. However, the fact that Schwarz reflections in $S_\cT$ have many free critical points (all of which map to the same point) gives rise to interesting distinctions. For instance, unlike in the Multibrot/Multicorn family, certain \emph{periodic} parameter rays of $S_\cT$ land at \emph{Misiurewicz} (critically pre-periodic) parameters. Further, each of the free critical points can become a superattracting fixed point, thus producing \emph{several} period one renormalization loci. It is worth mentioning that only one such renormalization locus furnishes pinched anti-polynomial-like restrictions, and we use the Parabolic Straightening Theorem of \cite{PartI} to study this renormalization locus (see Theorem~\ref{renorm_locus_thm_intro}). All other period one renormalization loci present standard anti-polynomial-like maps, whence the classical theory of straightening maps applies (cf. \cite{IK12,IM21,IM22}).

\subsubsection{Dessins in dynamical and parameter planes}

The dessin $\cT$ plays a pivotal role in the dynamics of Schwarz reflections in $S_\cT$ as well as in the structure of the connectedness locus. For maps in the connectedness locus, the dessin `embeds' in the non-escaping set (\S~\ref{dynamical_plane_sec}). Moreover, the iterated preimages of the quadrature domain provide a natural puzzle structure in the dynamical plane of a map in $\cC(S_\cT)$, and the \emph{contact graphs} of such puzzle pieces is given by pulling back the embedding of the dessin. This turns out to be useful in the analysis of non-renormalizable parameters in $\cC(S_\cT)$ (see \S~\ref{renorm_sec}). Further, the dessin $\cT$ also determines the period $1$ renormalization loci in the connectedness locus (see \S~\ref{multibrot_multicorn_copies_sec}).
Finally, $\cT$ sits as a `spine' of the abstract connectedness locus (see Theorem~\ref{dessin_in_conn_locus_thm}).

\subsection{Which correspondences arise as matings?}
In light of the main results of \cite{PartI}, it is desirable to have a good description of the class of degree $d+1$ polynomials $f$ such that the correspondences generated by a circular reflection and the local deck transformations of $f$ (see Equation~\eqref{corr_eqn}) realize matings between the anti-Hecke group $\pmb{\Gamma}_d$ and parabolic anti-rational maps in $\cF_d$. The results of \S~\ref{shabat_poly_schwarz_sec} reveal that for every Shabat polynomial $f$ satisfying a non-degeneracy condition (namely, the dessin of $f$ has a valence $2$ vertex adjacent to a tip), there exists a circular reflection such that the associated correspondence is a mating of $\pmb{\Gamma}_d$ and some $R\in\cF_d$. Note that in the special case when such an $f$ has two critical points in the plane (i.e., when its dessin is given by adding an edge to one of the arms of a star-tree), the corresponding anti-rational map $R$ has a unique free critical point (as well as a critical point in the persistent parabolic basin).

\subsection*{Acknowledgments}
This work is based on the thesis of the second author \cite{MazThesis}. This author is grateful for funding from Stony Brook University and Tata Institute of Fundamental Research. 
Part of this work was done during the authors’ visits to Institute for Theoretical Studies at ETH Z{\"u}rich, the IMS at Stony Brook, and SLMath (MSRI) in Berkeley. The authors thank these institutes for their hospitality and support. We also thank the referee for his/her careful reading of the manuscript and for making many useful comments.

\section{Background and prior results}\label{background_sec}

\subsection{Shabat polynomials}\label{shabat_tree_subsec}

\begin{definition}[Shabat polynomials]\label{shabat_def}
	A polynomial $f:\widehat{\C}\to\widehat{\C}$ (of degree at least three) is called a \emph{Shabat polynomial} if it has exactly two finite critical values. Two Shabat polynomials $f_1$ and $f_2$ are called \emph{equivalent} if there exist M{\"o}bius maps $A_1, A_2$ such that $f_2=A_1\circ f_1\circ A_2$.
\end{definition}

\begin{definition}[Bicolored plane trees]\label{bicolored_plane_tree_def}
	\noindent\begin{itemize}[leftmargin=8mm]
		\item A combinatorial \emph{bicolored plane tree} is a combinatorial tree $\mathcal{T}$ equipped with
		\begin{enumerate}
			\item a coloring of the vertices in two colors (say, black and white) in such a way that every edge connects a white vertex to a black vertex, and
			
			\item a circular order of the edges around every vertex.  
		\end{enumerate}
		
		\item Two combinatorial bicolored plane trees are said to be \emph{isomorphic} if there exists a tree isomorphism between them that preserves the colors of corresponding vertices and the circular order of the edges around corresponding vertices.
		
		\item For a combinatorial bicolored plane tree $\mathcal{T}$, the \emph{opposite} tree $\mathcal{T}^{\mathrm{op}}$ is a combinatorial bicolored plane tree that is isomorphic to $\mathcal{T}$ as a bicolored tree, but has the opposite circular order of the edges around every vertex.
	\end{itemize}
\end{definition}

Let $\gamma$ be a simple arc in $\C$ connecting the two finite critical values $y_1$ and $y_2$ of a Shabat polynomial $f$. Since $f$ has no pole in the plane, the preimage 
$$
\mathcal{T}_\gamma(f):=f^{-1}(\gamma)
$$ 
has the structure of a tree with vertices at $f^{-1}(\{y_1, y_2\})$. We color preimages of $y_1$ (respectively, $y_2$) black (respectively, white). Then $\mathcal{T}_\gamma(f)$ has the structure of a bicolored plane tree. Moreover, $\mathcal{T}_\gamma(f)$ has $\Deg{(f)}$ edges, and the valence of a vertex of $\mathcal{T}_\gamma(f)$ is equal to the local degree of $f$ at that point. 

\begin{proposition}\label{homotopic_trees}
	Let $\gamma_1, \gamma_2\subset\C$ be simple arcs connecting the two critical values $y_1$ and $y_2$ of a Shabat polynomial $f$. Then,
	\begin{enumerate}[leftmargin=8mm]
		\item $\mathcal{T}_{\gamma_1}(f)$ and  $\mathcal{T}_{\gamma_2}(f)$ have the same vertex set, and
		
		\item the trees $\mathcal{T}_{\gamma_1}(f)$ and  $\mathcal{T}_{\gamma_2}(f)$ are isotopic relative to the vertices.
	\end{enumerate}
\end{proposition}
\begin{proof}
	(1) The vertex set of both the trees is $f^{-1}(\{y_1, y_2\})$.
	
	(2) As $\gamma_1, \gamma_2\subset\C$ are simple arcs connecting $y_1$ and $y_2$, they are isotopic in $\C$ such that the isotopy fixes the endpoints (for instance, see \cite[Proposition~2.2]{FM}). Finally, since $f:\C\setminus f^{-1}(\{y_1, y_2\})\to\C\setminus\{y_1, y_2\}$ is a covering map, the above isotopy between $\gamma_1$ and $\gamma_2$ can be lifted to isotopies between corresponding edges of $\mathcal{T}_{\gamma_1}(f)$ and  $\mathcal{T}_{\gamma_2}(f)$ relative to $f^{-1}(\{y_1, y_2\})$.
\end{proof}

Proposition~\ref{homotopic_trees} tells that the isotopy type of $\mathcal{T}_\gamma(f)$ (relative to the vertices) is independent of the choice of an arc $\gamma$ connecting $y_1$ and $y_2$. In particular, the various $\mathcal{T}_\gamma(f)$ are isomorphic as combinatorial bicolored plane trees. We will denote this combinatorial bicolored plane tree by $\mathcal{T}(f)$, and call it the \emph{dessin d'enfants} of $f$ (compare \cite[Sec.2]{LZ04}). Moreover, it is easy to see that the isomorphism class of $\mathcal{T}(f)$ remains unchanged if $f$ is replaced by a Shabat polynomial equivalent to~$f$.

\begin{remark}
	Some authors allow Shabat polynomials to have a single critical value in the plane. On the contrary, we adopt the convention that a Shabat polynomial has exactly two finite critical values. Hence, the dessins d'enfants associated with the Shabat polynomials considered in this paper necessarily have at least one black and one white vertex of valence greater than one.
\end{remark}

\begin{theorem}\cite[Theorem~2.2.9]{LZ04}, \cite[Theorem~I.5]{schneps}\label{shabat_classification_theorem}
	The map $$f\mapsto\mathcal{T}(f)$$ induces a bijection between the set of equivalence classes of Shabat polynomials and the set of isomorphism classes of combinatorial bicolored plane trees (with at least one black and one white vertex of valence greater than one).
\end{theorem}

We refer the reader to \cite{dessin_book_schneps,BZ96,LZ04} for a comprehensive account on Shabat polynomials (and more generally, Belyi pairs), dessins d'enfants associated with such maps, and the action of the absolute Galois group on dessins d'enfants.

\subsection{Background on Schwarz reflections}\label{schwarz_background_subsec}

\begin{definition}\label{Schwarz_reflec_def}
	A bounded domain $\Omega\subset \C$ satisfying $\Omega=\text{int }\overline \Omega$ is called a \textit{quadrature domain} if there exists a continuous map $\sigma\colon\overline{\Omega}\to \widehat{\C}$ such that:
	\begin{itemize}[leftmargin=6mm]
		\item $\sigma$ is anti-meromorphic on $\Omega$ and
		\item $\sigma(z)=z$ for all $z\in \partial \Omega$.
	\end{itemize}
	We say that $\sigma$ is the \textit{Schwarz reflection} map for $\Omega$.
\end{definition}

In this work we will be exclusively concerned with simply connected quadrature domains, in which case the Schwarz reflection map takes on a special form. Here and in the rest of the paper, we will use the notation $\eta(z)=1/\overline{z}$. We note that by \cite{Sak91}, the boundary of a quadrature domain is real-analytic with at most finitely many singular points, each of which is a cusp or a double point.

\begin{proposition}\cite[Theorem~1]{AS}, \cite[Propositions~2.4,~2.5]{MR25}\label{simp_conn_quad}
	A simply connected domain $\Omega\subsetneq\widehat{\C}$ with $\infty\notin\partial\Omega$ and $\Int{\overline{\Omega}}=\Omega$ is a quadrature domain if and only if some (and hence each) Riemann uniformization $f:\D\to\Omega$ extends to a rational map on $\widehat{\C}$. The Schwarz reflection map $\sigma$ of $\Omega$ is given by $f\circ\eta\circ(f\vert_{\D})^{-1}\ \mathrm{:}$
	\[ \begin{tikzcd}
		\overline{\D} \arrow{r}{f} \arrow[swap]{d}{\eta} & \overline{\Omega} \arrow{d}{\sigma} \\
		\widehat{\C}\setminus\D \arrow[swap]{r}{f}& \widehat{\C}
	\end{tikzcd}
	\]
	In this case, if the degree of the rational map $f$ is $d+1$, then $\sigma:\sigma^{-1}(\Omega)\ \left(\subset\Omega\right)\to\Omega$ is a (branched) covering of degree $d$, and $\sigma:\sigma^{-1}(\Int{\Omega^c})\to\Int{\Omega}^c$ is a (branched) covering of degree $d+1$.
\end{proposition}

\begin{remark}   Several remarks about Proposition~\ref{simp_conn_quad} are in order.
\noindent\begin{enumerate}[leftmargin=*]
\item A Riemann uniformization $f:\D\to\Omega$ is unique up to pre-composition with a conformal automorphism of the disk. Since such automorphisms are restrictions of global M{\"o}bius maps, the existence of a global rational extension of some uniformizing map $f:\D\to\Omega$ implies the existence of global rational extensions of all Riemann uniformizations of $\Omega$. Further, as conformal automorphisms of $\D$ commute with $\eta$, it follows that replacing a Riemann uniformization $f:\D\to\Omega$ with $f\circ M:\D\to\Omega$, where $M$ is a conformal automorphism of $\D$, leaves the Schwarz reflection map $\sigma$ unchanged.

   \item Since the boundary of a quadrature domain is real-analytic, a Riemann uniformization $f:\D\to\Omega$ of a quadrature domain $\Omega$ extends continuously to the boundary. As explained in \cite[Proposition~2.4]{MR25}, the map $f$ extends holomorphically as $\sigma\vert_{\Omega}\circ f\vert_{\D}\circ\eta$ to $\widehat{\C}\setminus\overline{\D}$, where $\sigma$ is the Schwarz reflection map of the quadrature domain. This gives the desired description of $\sigma$ in terms of $f$ and $\eta$ in the commutative diagram above.
   Conversely, if the Riemann uniformization $f:\D\to\Omega$ of a simply connected domain $\Omega$ extends to a global rational map, then the map $\sigma:=f\circ\eta\circ(f\vert_{\D})^{-1}$ is meromorphic on $\Omega$ and extends continuously as the identity map to $\partial\Omega$. Thus, $\sigma$ is the Schwarz reflection of $\Omega$ in the sense of Definition~\ref{Schwarz_reflec_def}, and hence $\Omega$ is a quadrature domain.

   \item The formula $\sigma:=f\circ\eta\circ(f\vert_{\D})^{-1}$ shows that the Schwarz reflection map $\sigma$ does not extend to a global antiholomorphic map unless $f$ is a M{\"o}bius map (i.e., $\Omega$ is a round disk).

\begin{figure}[ht!]
\captionsetup{width=0.98\linewidth}
\begin{tikzpicture}
\node[anchor=south west,inner sep=0] at (0.75,0) {\includegraphics[width=0.66\textwidth]{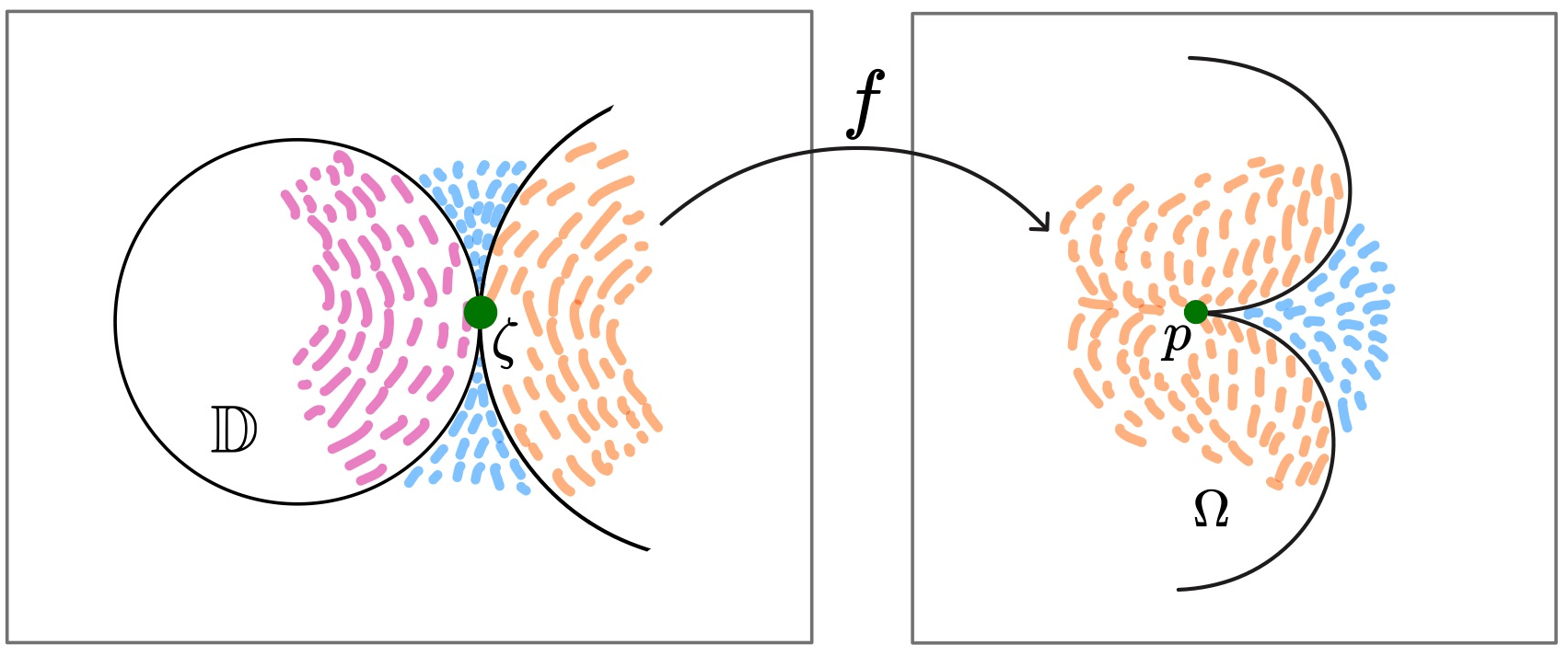}};
\end{tikzpicture}
\caption{The rational map $f$ carries $\D$ univalently onto $\Omega$ and has a simple critical point at $\zeta\in\mathbb{S}^1$. Each of the two regions shaded in pink and orange maps injectively under $f$ onto the orange part of $\Omega$ near the cusp $p=f(\zeta)$. However, only the orange preimage lies outside $\D$, while the pink preimage lies in $\D$. On the other hand, each of the two cuspidal blue regions lies outside $\overline{\D}$, and maps injectively under $f$ onto the blue part of $\Int\Omega^c$ near $p$. As a result, the Schwarz reflection map $\sigma$ has degree two over  $\Int\Omega^c$ and degree one over $\Omega$ locally near $p$.}
\label{cusp_fig}
\end{figure}   

\item The map $\sigma$ extends anti--analytically  to a neighborhood of the non-singular part of $\partial\Omega$, but it does not admit such a continuation to a neighborhood of a cusp or a double point. So strictly speaking, $\sigma$ is a local reflection around the regular part of the boundary only. 

   \item Since the degree $d+1$ rational map $f$ carries $\D$ injectively onto $\Omega$, it follows that $f:f^{-1}(\Omega)\setminus\D\to\Omega$ is a degree $d$ branched covering, $f^{-1}(\Int\Omega^c)\subset\widehat{\C}\setminus\overline{\D}$,  and $f:f^{-1}(\Int\Omega^c)\to\Int\Omega^c$ is a degree $d+1$ branched covering. The relation $\sigma:=f\circ\eta\circ(f\vert_{\D})^{-1}$ now directly implies that 
   $$
   \sigma:\sigma^{-1}(\Omega)=f\Big(\eta\big(f^{-1}(\Omega)\setminus\D\big)\Big)\longrightarrow\Omega
   $$
   is a degree $d$ branched covering, and
   $$
   \sigma:\sigma^{-1}(\Int\Omega^c)=f\Big(\eta\big(f^{-1}(\Int\Omega^c)\big)\Big)\longrightarrow\Int{\Omega^c}
   $$
    is a degree $d+1$ branched covering (see \cite[Proposition~2.5]{MR25} for details). 
    
 It is worth noting that while the above explanation of the mapping degrees of $\sigma$ remains valid whether or not $\partial\Omega$ contains singular points, the change of degree of $\sigma$ can actually be seen locally near singular points. For instance, if $p$ is a cusp on the boundary of a quadrature domain $\Omega$, then there exists a simple critical point $\zeta\in\mathbb{S}^1$ of $f$ with $f(\zeta)=p$ (cf. \cite[\S~2.3]{MR25}). As shown in Figure~\ref{cusp_fig}, a point in $\Omega$ near $p$ has exactly one $f-$preimage in $B(\zeta,\epsilon)\setminus\D$ (for $\epsilon\approx 0$), but a point in $\Int\Omega^c$ near $p$ has two $f-$preimages in $B(\zeta,\epsilon)$. Consequently, a point in $\Omega$ near $p$ has exactly one $\sigma-$preimage near $p$, but a point in $\Int\Omega^c$ near $p$ has two $\sigma-$preimages near $p$. 
\end{enumerate}
\end{remark}

We refer the reader to \cite{AS}, \cite{lee-makarov}, \cite{LM23} and the references therein for more background on quadrature domains and Schwarz reflection maps.

Let us now return to the degree $(d+1)$ rational map $f$ that is univalent on $\overline{\D}$. We set $\Omega:=f(\D)$ and denote the associated Schwarz reflection map by $\sigma$.

We define $T(\sigma):=\widehat{\mathbb{C}}\setminus \Omega$ and $S(\sigma)$ to be the singular set of $\partial T(\sigma)$. We further set $T^0(\sigma):=T(\sigma)\setminus S(\sigma)$, and $$T^\infty(\sigma):=\bigcup_{n\geq0} \sigma^{-n}(T^0(\sigma)).$$ We will call $T^\infty(\sigma)$ the \emph{tiling set} of $\sigma$. For any $n\geq0$, the connected components of $\sigma^{-n}(T^0(\sigma))$ are called \emph{tiles} of rank $n$. Two distinct tiles have disjoint interior. And the tiling set, $T^\infty$, is open \cite[Proposition~2.3]{PartI}. The \emph{non-escaping set} of $\sigma$ is defined as 
$$
K(\sigma):=\widehat{\C}\setminus T^\infty(\sigma)\subset \Omega\cup S(\sigma).
$$ 
The common boundary of the non-escaping set $K(\sigma)$ and the tiling set $T^\infty(\sigma)$ is called the \emph{limit set} of $\sigma$, denoted by $\Lambda(\sigma)$. 

\subsection{Schwarz reflections with the anti-Farey external map}\label{antiFarey_subsec}

Let $\Omega$ be a Jordan quadrature domain such that 
\begin{enumerate}[leftmargin=8mm]
    \item the Riemann uniformization $f:\D\to\Omega$ extends to a degree $d+1$ rational map of $\widehat{\C}$,
    \item there is a unique critical point of $\sigma$, of local degree $d+1$, in $\sigma^{-1}(\Int T)$ with the associated critical value at $\infty$, and
    \item there is exactly one cusp singularity on the boundary $\partial\Omega$.
\end{enumerate}    
After possibly pre-composing $f$ with an automorphism of the disk, we may assume that $f(0)$ is the unique critical point of $\sigma$
in $\Int T$. By Proposition~\ref{simp_conn_quad}, this is equivalent to requiring that $f$ is a degree $d+1$ polynomial which is injective on $\overline{\D}$ and has exactly one (simple) critical point on $\partial \D$.

In \cite[Proposition 3.3]{PartI} it was shown that under the further condition that $K(\sigma)$ is connected, the dynamics of $\sigma$ on the tiling set $T^\infty(\sigma)$ is modeled by a certain explicit conformal class, which we now describe.

\subsubsection{Anti-Farey maps}\label{anti_farey_subsubsec} Consider an ideal, regular, hyperbolic, $(d+1)-$gon $\Pi$ contained in the unit disk $\D$ (we assume that $\Pi$ is closed in $\D$). By an isometric change of coordinates, we can place the ideal vertices of $\Pi$ at the $(d+1)-$st roots of unity. We set $\omega:=e^{2\pi i/(d+1)}$, and $M_\omega(z)=\omega z$. Let $\D_1$ be the component of $\D\setminus\Pi$ having $1$ and $e^{2\pi i/(d+1)}$ on its boundary. Further, let $Q$ be the sector of angle $2\pi/(d+1)$ from the origin that contains $\D_1$, such that the radial line $(0,1)$ is contained in $Q$ but $M_\omega(0,1)$ is not.
We define the map $\widehat{\mathcal R_d}\colon \overline{\D_1}\to Q$ given by reflection across $\partial\D_1\cap \D$ (which is one of the sides of $\Pi$) followed by a rotation of angle $ \frac{2\pi}{d+1}k(\cdot)$, where $k(\cdot)$ is the unique integer so that the image lands in $Q$. 

Letting $\mathcal{Q}:=\overline{Q}/\langle z\sim \omega z\rangle$ and $\mathcal{Q}_1=\Pi/\langle M_\omega\rangle$ we obtain an anti-holomorphic map 
$$
\mathcal{R}_d\colon \mathcal{Q}\setminus\Int{\mathcal{Q}_1}\to \mathcal{Q}
$$ 
induced by $\widehat{\mathcal R_d}$. We call this map the \textit{anti-Farey map}. The map $\mathcal R_d$ restricts to the identity on $\partial \mathcal{Q}_1$, and to an expansive, piecewise analytic, orientation-reversing degree $d$ covering on $\partial\mathcal{Q}\cong\mathbb{S}^1$. Moreover, $\mathcal{R}_d$ has a unique critical point of local degree $d+1$. We also note that points in $\mathcal{Q}\setminus\mathcal{Q}_1$ have $d$ preimages and points in $\mathcal{Q}_1$ have $d+1$ preimages under $\mathcal{R}_d$. We refer the reader to \cite[\S~3]{PartI} for more details on the anti-Farey map, and how it relates to the group dynamics of the \emph{anti-Hecke group} $\pmb{\Gamma}_d$, which is a discrete subgroup of $\mathrm{Aut}^\pm(\D)$ (the group of all conformal and anti-conformal automorphisms of $\D$) generated by the rotation $M_\omega$ and the reflection in any of the sides of $\Pi$.

\subsubsection{The class $\mathcal{S}_{\mathcal{R}_d}$ of Schwarz reflections}\label{anti_farey_external_fiber_subsubsec}. 
\begin{definition}\label{srd_def}
	We define $\mathcal{S}_{\mathcal{R}_d}$ to be the space of pairs $(\Omega,\sigma)$, where 
	\begin{enumerate}[leftmargin=8mm]
		\item $\Omega$ is a Jordan quadrature domain with associated Schwarz reflection map $\sigma:\overline{\Omega}\rightarrow\widehat{\C}$, and
		
		\item there exists a conformal map $\psi:(\Int{\mathcal{Q}},0)\rightarrow (T^\infty(\sigma),\infty)$ that conjugates $\mathcal{R}_d:\Int{\mathcal{Q}}\setminus\Int{\mathcal{Q}_1}\longrightarrow\Int{\mathcal{Q}}$ to $\sigma:T^\infty(\sigma)\setminus\Int{T^0(\sigma)}\longrightarrow T^\infty(\sigma)$.
	\end{enumerate}
\end{definition}

It was shown in \cite[Proposition~3.3]{PartI} that the second condition is equivalent to the Jordan quadrature domain $\Omega$ having a single cusp singularity on its boundary, a fully ramified critical value (of $\sigma$) in $\Int T=\widehat{\C}\setminus\overline{\Omega}$, and a connected non-escaping~set.

The anti-Farey map serves as the external map for the space $\cS_{\cR_d}$ as the power map $z\mapsto z^d$ serves as an external map for polynomials via B\"ottcher coordinates. Following \cite[\S 2]{LLMM1}, one can define an analog of external rays that exist in polynomial dynamics (cf. \cite{Milnor06,Lyu25}).

Let $\rho_1,\cdots,\rho_{d+1}$ be the reflections in the sides of $\Pi$.

\begin{definition}\label{Gd_rays}
	Let $(i_1, i_2, \cdots)\in\{1,\cdots, d+1\}^{\N}$ with $i_j\neq i_{j+1}$ for all $j$. The corresponding infinite sequence of tiles $\{\Pi,\rho_{i_1}(\Pi),\rho_{i_1}\circ\rho_{i_2}(\Pi),\cdots\}$ shrinks to a single point of $\mathbb{S}^1\cong \R/\Z$, which we denote by $\theta(i_1, i_2, \cdots)$. We define a \emph{$\pmb{\Gamma}_d-$ray at angle $\theta(i_1, i_2, \cdots)$} to be the concatenation of hyperbolic geodesics (in $\D$) connecting the consecutive points of the sequence $\{0,\rho_{i_1}(0),\rho_{i_1}\circ\rho_{i_2}(0),\cdots\}$. Clearly, this curve lands at the point $\theta(i_1, i_2, \cdots)\in\mathbb{S}^1$. 
	\end{definition}
	
	We note that in general there may be more than one $\pmb{\Gamma}_d-$rays at a given angle; for instance, both the symbol sequences $(1,d+1,1,d+1,\cdots)$ and $(d+1,1,d+1,1,\cdots)$ yield $\pmb{\Gamma}_d-$rays at angle $0$.
	
	The $\pmb{\Gamma}_d-$rays in $\D$ with angles in $[0,\frac{1}{d+1})$ yield rays in $\mathcal{Q}$ such that the image of (the tail of) a ray at angle $\theta$ under $\mathcal{R}_d$ is (the tail of) a ray at angle $\mathcal{R}_d(\theta)$. In \S~\ref{external_dynamics_subsec}, we will describe how these $\pmb{\Gamma}_d-$rays give rise to `dynamical rays' for Schwarz reflections in $\mathcal{S}_{\mathcal{R}_d}$. 
	
\subsection{Pinched anti-polynomial-like maps and a straightening theorem}

\subsubsection{Pinched anti-polynomial-like maps}
In \cite[\S 2.4]{PartI} the notion of a \textit{pinched polynomial-like map} was introduced, generalizing the standard definition so as to allow the possibility of the domain not being compactly contained in the range. We define a \emph{polygon} to be a closed Jordan domain with a piecewise smooth boundary. A \textit{pinched polygon} is a compact set in $\widehat{\C}$ which is homeomorphic to the quotient of a closed disk under a finite geodesic lamination and which has a piecewise smooth boundary. The cut-points of the
boundary of a pinched polygon will be called its \emph{pinched points}, and the non-cut points where two smooth local boundary arcs meet are called its \emph{corners}.

\begin{definition}
	Let $P_2\subset \widehat{\C}$ be a polygon, and let $P_1\subset P_2$ be a pinched polygon where $\partial P_1\cap\partial P_2$ is the set of corners of $P_2$ and is contained in the set of corners of~$P_1$. 

Suppose that there is a continuous map $g\colon P_1\to P_2$ such that
\begin{enumerate}[leftmargin=8mm]
\item $g$ is a proper (anti-)holomorphic map from each component of $\Int (P_1)$ onto $\Int (P_2)$, 
\item $g$ is locally injective at the corners of $P_1$, and 
\item the corners and pinched points of $P_1$ are the preimages of the corners of~$P_2$.
\end{enumerate}

We then call the triple $(g,P_1,P_2)$ a \emph{pinched (anti-)polynomial-like map}.
\end{definition}

\subsubsection{Parabolic anti-rational maps}\label{para_anti_rat_subsec}
Let us recall the family of anti-rational maps to which the straightening theorems of \cite{PartI} applies. Throughout we have the notation that $B_d\colon \D\to \D$ is the anti-holomorphic Blaschke product given by $z\mapsto \frac{(d-1)\bar z^d + (d+1)}{(d+1)\bar z^d + (d-1)}$.

\begin{definition}\label{para_anti_rat_family_def}
We define the family $\cF_d$ to be the collection of degree $d\geq 2$ anti-rational maps $R$ with the following properties.
\begin{enumerate}[leftmargin=8mm]
\item $\infty$ is a parabolic fixed point for $R$.
\item There is a marked immediate parabolic basin $\mathcal{B}(R)$ of $\infty$ which is simply connected and completely invariant.
\item $R\vert_{\mathcal{B}(R)}$ is conformally conjugate to $B_d\vert_{\D}$.
\end{enumerate}
\end{definition}

For a map $R\in \mathcal F_d$ we define the {\em filled Julia set} $\mathcal K(R)$ as $\widehat\C\setminus\mathcal B(R)$.

\subsubsection{A straightening theorem for pinched anti-polynomial-like maps}

 One of the main technical results of \cite{PartI} asserts that under some additional regularity assumptions, a pinched anti-polynomial-like map is hybrid conjugate to a parabolic anti-rational map. The precise conditions required for this straightening result can be found in \cite[Definition~4.5]{PartI}, where maps satisfying these additional properties are termed \emph{simple pinched anti-polynomial-like maps}.

\begin{theorem}\label{straightening_thm}\cite[Theorem 4.8]{PartI}
	\noindent\begin{enumerate}[leftmargin=8mm]\upshape
		\item Let $(F,P_1,P_2)$ be a simple pinched anti-polynomial-like map of degree $d\geq~2$. Then $F$ is hybrid conjugate to a simple pinched anti-polynomial-like restriction of a degree $d$ anti-rational map $R$ with a simple parabolic fixed point.
		
		\item If the filled Julia set of $F$ is connected, then $R$ is unique up to M{\"o}bius conjugacy, and has a unique representative in $\ \faktor{\cF_d^{\mathrm{simp}}}{\mathrm{Aut}(\C)}$, where $\cF_d^{\mathrm{simp}}$ consists of maps in $\cF_d$ for which $\infty$ is a simple parabolic fixed point.
	\end{enumerate}
\end{theorem}

\subsection{Natural bijection between parameter spaces}

The following lemma connects Schwarz reflections in the dense open subset $\mathcal{S}_{\mathcal{R}_d}^{\mathrm{simp}}$ of $\mathcal{S}_{\mathcal{R}_d}$ (consisting of maps in $\mathcal{S}_{\mathcal{R}_d}$ with a simple cusp on
the associated quadrature domain boundaries) to the theory of pinched anti-polynomial-like maps.

\begin{theorem}\label{simp_restrict_thm}\cite[Lemma~4.13, Theorem~4.15]{PartI}
\noindent\begin{enumerate}[leftmargin=8mm] 
\item Let $(\Omega,\sigma)\in \mathcal{S}_{\mathcal{R}_d}^{\mathrm{simp}}$. Then, there exists a Jordan domain $V\subset \Omega$ with $\overline V\supset K(\sigma)$ and a conformal map $\beta:\overline{V}\hookrightarrow\widehat{\C}$ such that $\beta$ conjugates $\sigma\colon \overline{\sigma^{-1}(V)}\to \overline{V}$ to a simple pinched anti-polynomial-like map $(F,P_1,P_2)$ of degree $d$. Moreover, the filled Julia set of this pinched anti-polynomial-like map is $\beta(K(\sigma))$.

\item The above simple pinched anti-polynomial-like map is hybrid conjugate to an anti-rational map in $\cF_d^{\mathrm{simp}}$, which is unique up to M{\"o}bius conjugation. 
\end{enumerate}
\end{theorem}

In \cite[Theorem 4.12]{PartI} another straightening technique was developed, which allows one to straighten all Schwarz reflections in $\cS_{\cR_d}$, albeit at the expense of losing some parameter control of the straightening map.

\begin{theorem}\cite[Theorems~A,~B]{PartI}\label{qc_bijection_thm}
\noindent\begin{enumerate}[leftmargin=8mm]\upshape
\item There exists a bijection
$$
\chi\colon \faktor{\mathcal S_{\mathcal R_d}}{\mathrm{Aut}(\C)}\to \faktor{\mathcal F_d}{\mathrm{Aut}(\C)}
$$
such that each Schwarz reflection $(\Omega,\sigma)$ admits a pinched anti-polynomial-like restriction that is hybrid conjugate to a pinched anti-polynomial-like restriction of $\chi((\Omega,\sigma))$.

\item Let $R\in\cF_d$, $(\Omega,\sigma)=\chi^{-1}(R)$, and $f$ be a degree $d+1$ polynomial such that $f:\overline{\D}\to\overline{\Omega}$ is a homeomorphism. Then, the antiholomorphic correspondence $\mathfrak{C}$ defined by the polynomial relation
\begin{equation*}
\frac{f(w)-f(1/\overline{z})}{w-1/\overline{z}}=0,
\end{equation*}
is a mating of $\pmb{\Gamma}_d$ and $R$.
\end{enumerate}
\end{theorem}

In this paper we will need to establish the existence of Schwarz reflections for which all bounded critical orbits are finite and have certain prescribed combinatorics. The strategy will be to first construct anti-polynomials with desired combinatorics using a result of Poirier's described in the following subsection, and then to apply the following theorem, also established in \cite{PartI}.

\begin{theorem}\cite[Theorem C]{PartI}\label{Thm_C_part_I}
For any semi-hyperbolic anti-holomorphic polynomial $p$ with connected Julia set, there is a pinched neighborhood $U$ of the filled Julia set $K(p)$ such that $p$ is David conjugate on $U$ to a Schwarz reflection $(\Omega,\sigma)\in \mathcal S_{\mathcal R_d}$.
\end{theorem}

\subsection{PCF maps and Hubbard trees}

Recall that for a polynomial $p\colon \C\to \C$ the \textit{post-critical} set is defined to be

$$P(p)= \overline{\bigcup_{n\geq 1,\,c\text{ critical}}p^{\circ n}(c)}.$$

A polynomial with a finite post-critical set (called a \emph{postcritically finite or PCF} polynomial) is described by its \textit{Hubbard tree}, which is a forward invariant subset of the filled Julia set containing all finite critical points. Hubbard trees were first considered in work of Douady and Hubbard \cite{DH1} and were later expanded upon by Poirier \cite{PoiThesis,Poi13}.

\begin{definition}\label{def_augmented_hubbard_tree_poly}
    The \emph{augmented hubbard tree} of a postcritically finite (anti-)polynomial is defined to be the minimal hull in the filled Julia set of the post-critical set, preimages of the critical values, the landing point of the $0$ ray, and its preimages.\\
Furthermore, if the landing point of the $0-$ray has valence at least two, we record the circular ordering of the corresponding edges and the $0-$ray.
\end{definition}

As defined above, Hubbard trees are rigid geometric objects which are acted upon by polynomials. Combinatorial analogs were considered by Poirier in \cite{PoiThesis,Poi13}.

\begin{definition}\label{angled_tree_dyn_def}
    
An \textit{angled tree} is an abstract planar embedded tree in which one assigns angles $\angle_v(e,e')\in\R/2\pi\Z$ to any two edges $e$ and $e'$ incident at a common vertex $v$. Throughout the paper angles will be normalized so that if $e$ and $e'$ are consecutive in the anti-clockwise circular ordering around $v$ then $\angle_v(e,e')=2\pi/\mathrm{valence}(v)$. \textit{A (normalized) orientation-reversing angled tree dynamics} is an angled tree $T$ together with a map $f\colon T\to T$ such that:

\begin{enumerate}[leftmargin=8mm]
    \item $f$ sends vertices of $T$ to vertices of $T$ and injectively maps edges of $T$ to unions of edges of $T$,
    \item $f$ reverses the circular ordering of edges at every vertex,
    \item for every vertex $v$ there is a positive integer $d(v)$ such that $\angle_{f(v)}(f(e),f(e')) = -d(v)\cdot \angle_v(e,e')$.
\end{enumerate}
    Say that a vertex is a \textit{critical point} for angled tree dynamics if $d(v)\geq 2$. We may then classify vertices as being of \textit{Fatou type} if their orbit eventually lies in a periodic critical cycle, and of \textit{Julia type} otherwise.
    Lastly, we say that an angled tree dynamics is \textit{expanding} if given any two adjacent Julia type vertices $v$ and $v'$, there is some $n$ for which $f^{\circ n}(v)$ and $f^{\circ n}(v')$ are not adjacent.
    \end{definition}
    For more details on angled tree dynamics and generalizations thereof, we refer the reader to \cite{Poi13}.

    It is not too difficult to see that a Hubbard tree for an anti-holomorphic PCF polynomial is described combinatorially by an expanding, orientation-reversing angled tree dynamics. The following theorem due to Poirier (following W. Thurston's realization theorem) gives a converse to this statement.

\begin{theorem}\cite[Theorem~5.1]{Poi13}\label{poi_thm} 
A normalized orientation-reversing angled tree dynamics can be realized as the dynamics associated to a postcritically finite anti-holomorphic polynomial map if and only if it is expanding. Such a realization is unique up to affine conjugation in the dynamical plane.
\end{theorem}

\subsection{Constructing special Shabat anti-polynomials}\label{dyn_shabat_existence_subsec}

We say that an anti-polynomial $p$ is \emph{Shabat} if it has exactly two finite critical values. One can associate a dessin d'enfants to a Shabat anti-polynomial $p$ as in the holomorphic case. 

A Shabat anti-polynomial $p$ is called \emph{dynamically Shabat} if the postcritical set of $p$ (in $\C$) coincides with the set of finite critical values of $p$. 
Clearly, dynamically Shabat anti-polynomials are postcritically finite. In the holomorphic setting, such polynomials were studied in \cite{pilgrim}.

\begin{lemma}\label{dyn_shabat_existence_lem}
	Let $\mathcal{T}$ be a bicolored plane tree with at least one black and one white vertex of valence greater than one. Then, there exists a dynamically Shabat anti-polynomial $p$ such that 
	\begin{enumerate}[leftmargin=8mm]
		\item the dessin d'enfants of $p$ is isomorphic to $\mathcal{T}$,
		
		\item the augmented Hubbard tree $\mathcal{H}(p)$ of $p$ is planar isomorphic to $\mathcal{T}$.
	\end{enumerate} 
\end{lemma}

\begin{proof}
Without loss of generality, we can assume that $\cT$ is embedded in the plane so that it has a black vertex of valence one at $1$, and its adjacent white vertex (which is necessarily of valence at least two) is at $0$. Let $q$ be a Shabat polynomial realizing this tree, that is a polynomial with two critical values $a=q(0)$ and $b=q(1)$ in the plane such that for a closed arc $\gamma$ with endpoints at $a, b$, the set $q^{-1}(\gamma)$ has the structure of a bicolored plane tree isomorphic to $\cT$. Now define
$$
p(z) = \overline{\psi \circ {q(z)}},
$$
where $\psi(w)=\frac{w-a}{b-a}$; i.e., the affine map which sends $a$ to $0$ and $b$ to $1$.
	
We first note that $p$ is an anti-holomorphic polynomial with critical values $0$ and $1$ (and hence Shabat), and its dessin $p^{-1}([0,1])$ is isomorphic to $\cT$. Furthermore, note that $0$ and $1$ are fixed, and hence $\{0,1\}$ is the entire postcritical set of $p$. We also note that $0$ is a superattracting fixed point and $1$ is a repelling fixed point of~$p$.
	
Let $\widetilde{\gamma}\subset \cK(p)$ be an arc connecting $0$ and $1$.  Since $p^{-1}(\widetilde{\gamma})$ is isomorphic to $\cT$, it now suffices to prove that $p^{-1}(\widetilde{\gamma})$ is the augmented Hubbard tree of $p$. By construction, $p^{-1}(\widetilde{\gamma})$ is the hull in $\cK(p)$ of the post-critical set and the preimages of the critical values. To complete the proof, we will show that the 
the dynamical $0-$ray of $p$ lands at the repelling fixed point $1$. Since any fixed dynamical ray of $p$ can be turned into the $0-$ray by changing the normalization of the B{\"o}ttcher coordinate of $p$ at $\infty$, we only need to argue that some fixed dynamical ray of $p$ lands at $1$. By way of contradiction, assume that this is not the case.
Then by \cite[Lemma~2.5]{Muk1}, an even number of period two rays land at $1$. 
In this case, it is easy to see using the local orientation-reversing dynamics of $p$ that there are exactly two locally $p-$invariant components of $\mathcal{K}(p)\setminus\{1\}$ near $1$. One of these two components of $\mathcal{K}(p)\setminus\{1\}$ would contain $p^{-1}(\widetilde{\gamma})$ and hence all the finite critical points of $p$. 
This would force the other component in question to be $p-$invariant and disjoint from the Hubbard tree $\widetilde{\gamma}$ of $p$. But this would contradict the fact that cut-points of $\mathcal{K}(p)$ eventually land on its Hubbard tree under iteration (see \cite[Proposition~6.19]{kiwi1}).
\end{proof}

\begin{remark}\label{zero_ray_rem}
Since the arc $\widetilde{\gamma}$ is $p-$invariant and it contains no attracting/neutral fixed points in its interior, a simple argument from one-dimensional dynamics shows that the forward orbits of all points in $\widetilde{\gamma}\setminus\{1\}$ converge to $0$. Hence, $\widetilde{\gamma}\setminus\{1\}$ is contained in the fixed immediate superattracting basin of $0$. In other words, the repelling fixed point $1$ lies on the boundary of this Fatou component.
\end{remark}

\section{One-parameter families in $\cF_d$ and associated\\ Belyi Schwarz reflections}\label{shabat_poly_schwarz_sec}

\subsection{Slices $\mathfrak{L}_{\mathcal{T}}$ of Belyi anti-rational maps}\label{belyi_anti_rat_subsec}

An anti-rational map is called \emph{Belyi} if it has at most three critical values in $\widehat{\C}$. 
If $R\in\cF_d$ is Belyi with exactly three critical values, then the filled Julia set $\mathcal{K}(R)$ contains two of the three critical values of $R$. As in the holomorphic case, we can define the \emph{dessin d'enfants} of $R$ to be the bicolored plane tree obtained by taking the $R-$preimage of a simple arc $\gamma$ connecting the two critical values of $R$ in $\mathcal{K}(R)$. As the third critical value of $R$ (lying in $\mathcal{B}(R)$) is fully ramified, it follows that the dessin d'enfants of $R$ is a tree. 

If $R\in\cF_d$ is Belyi with exactly two critical values, then $\mathcal{K}(R)$ contains exactly one critical value of $R$ which must be fully ramified. Thus, this critical value must be different from the parabolic fixed point $\infty$. In this case, one can still construct a dessin d'enfants of $R$ by choosing $\gamma$ to be a simple arc that connects $\infty$ to the unique critical value of $R$ in $\mathcal{K}(R)$.

We will denote this (combinatorial) bicolored plane tree by $\mathcal{T}(R)$ (cf. \S~\ref{shabat_tree_subsec}).

A natural way of defining dynamically natural sub-families of Belyi anti-rational maps in $\cF_d$ is to introduce critical orbit relations. 

\begin{definition}\label{fd_slice_def}
We define
\begin{align*}
\mathfrak{L} & :=\{R\in\cF_d: R\ \textrm{is\ Belyi,\ and\ if\ $R$ has\ three\ critical\ values,\ then\ the\ parabolic}\\
&\qquad \textrm{fixed\  point}\ \infty\ \textrm{is\ a\ critical\ value\ of}\ R\}.
\end{align*}
\end{definition}

Note that for $R\in\mathfrak{L}$, the parabolic fixed point $\infty$ gives rise to a marked vertex of valence one (also called an endpoint or a tip) of $\mathcal{T}(R)$. As a convention, we will color this vertex black and denote it by $v_b$. The adjacent white vertex (of valence at least two) is denoted by $v_w'$.

\begin{definition}\label{fd_slice_tree_def}
We define
$$
\displaystyle \mathfrak{L}_{\mathcal{T}}:=\{R\in\mathfrak{L}: \left(\mathcal{T}(R),v_b\right)\cong(\mathcal{T},O)\},
$$
where $(\mathcal{T},O)$ is a bicolored plane tree with a black vertex $O$ of valence one as a root, and the isomorphism is understood to preserve the root and the bicolored plane structure. We also set
$$
\displaystyle \mathfrak{L}_{\mathcal{T}}^{\mathrm{simp}}:= \mathfrak{L}_{\mathcal{T}}\cap\cF_d^{\mathrm{simp}}=\{R\in\mathfrak{L}_{\mathcal{T}}: \infty\ \textrm{is\ a\ simple\ parabolic\ fixed\ point\ of}\ R\}.
$$
\end{definition}

We also remark that if $\mathcal{T}$ is a star-tree, then each map in $\mathfrak{L}_{\mathcal{T}}$ has a unique critical value in $\mathcal{K}(R)$.

\begin{lemma}\label{slice_non_emp_lem}
$\Int{\mathfrak{L}_{\mathcal{T}}}\neq \emptyset$.
\end{lemma}

\begin{proof}
Let $p$ be the Shabat anti-polynomial constructed in Lemma~\ref{dyn_shabat_existence_lem} (if $\mathcal T$ is a star-tree, instead use $p(z)=\overline{z}^d$). Recall that $p$ has three critical values, each a fixed point (or two critical values and a repelling fixed point). One of these critical values is $\infty$, which is fully ramified, one of the finite ones is a fixed critical point, which we denote $\pmb{v_w}$ and denote the remaining critical value as $\pmb{v_b}$ (in the case where $\cT$ is a star-tree use $\pmb{v_b}=1$.). Let $\psi(z)= \frac{z-\pmb{v_w}}{z-\pmb{v_b}}$ (i.e., $\psi$ is a M\"obius transformation taking $\pmb{v_w}$ to $0$ and $\pmb{v_b}$ to $\infty$) consider the map
	
	$$R(z) = \frac{1}{p'(\pmb{v_b})} \cdot \left(\psi\circ p\circ \psi^{-1}(z)\right).$$
	
	Necessarily we have that $0$ and $\infty$ are fixed points of $R$, and since conjugacies preserve multipliers of fixed points we have $R'(0)=0$ and $R'(\infty)=\frac{1}{p'(\pmb{v_b})}\cdot p'(\pmb{v_b})=1$. In particular $\infty$ is a parabolic fixed point, and the only other critical value of $R$  is $\frac{1}{p'(\pmb{v_b})}$. This is a fully ramified critical value. After applying a quasiconformal surgery on the parabolic basin to give this critical value height zero (cf. \cite[\S 2]{HS}, \cite[\S 3]{MNS}), we see that $R\in \mathfrak{L}_\cT$. Furthermore, a standard perturbation argument to turn superattracting fixed points to linearly attracting ones produces a non-enpty open set in $ \mathfrak{L}_\cT$.
\end{proof}

\begin{remark}
	Another proof of the above can be achieved using a \textit{David surgery} argument, by replacing the basin of infinity for $p$ with the parabolic Blaschke product $B_d(z) = \frac{(d+1)\overline z^d + (d-1)}{(d-1)\overline z^d +(d+1)}$.
\end{remark}

\subsection{The associated Belyi Schwarz reflections $S_\cT$}\label{belyi_schwarz_subsec}

We will now study the preimage of the family $\mathfrak{L}_{\mathcal{T}}$ under the straightening map $\chi$ (see Theorem~\ref{qc_bijection_thm}), and see that the corresponding Schwarz reflections come from univalent restrictions of Shabat polynomials whose dessin d'enfants can be explicitly read from $\mathcal{T}$.

\begin{proposition}\label{shabat_good_uni_prop}
Let $[\Omega,\sigma]\in\chi^{-1}(\mathfrak{L}_{\mathcal{T}})$, and $f$ be a degree $d+1$ polynomial such that $f:\D\to\Omega$ is a conformal isomorphism. Then the following hold.
\begin{enumerate}[leftmargin=8mm]
\item $f$ is a Shabat polynomial whose dessin d'enfants, denoted as $\mathcal{T}^{\mathrm{aug}}$, is obtained by adding a single edge to $\mathcal{T}^{\mathrm{op}}$ at the black vertex $v_b$ (see Definition~\ref{bicolored_plane_tree_def} for the notion of the opposite of a plane tree). The other endpoint of this new edge is necessarily white, and we denote it by $v_w$.

\noindent In particular, $\mathcal{T}^{\mathrm{aug}}$ has a black vertex $v_b$ of valence $2$ that lies between a terminal white vertex $v_w$ and a white vertex $v_w'$ of valence at least two.

\item If $emb:\mathcal{T}^{\mathrm{aug}}\to\C$ induces an isomorphism between the combinatorial tree $\mathcal{T}^{\mathrm{aug}}$ and a planar realization of it, then $emb(v_b)\in\mathbb{S}^1$ and $emb(v_w)\in\D$.
\end{enumerate}
\end{proposition}

\begin{remark}
(1) Recall that the vertex set of the embedded tree $\mathcal{T}_\gamma(f)$ (where $\gamma$ is an arc connecting the finite critical values of $f$) and the isomorphism $emb$ restricted to the vertex set of $\mathcal{T}^{\mathrm{aug}}$ are independent of the choice $\gamma$ (cf. \S~\ref{shabat_tree_subsec}).

(2) We use the notation $\mathcal{T}^{\mathrm{aug}}$ to remind the reader that it is an augmentation of the original tree $\mathcal{T}$.
\end{remark}

\begin{proof}
(1) Let $\chi(\Omega,\sigma)=R$. Then, $\sigma\vert_{K(\sigma)}$ is hybrid conjugate to $R\vert_{\mathcal{K}(R)}$.

Let us denote the unique cusp on $\partial\Omega$ by $y_c$. Note that $y_c$ is fixed under $\sigma$, and corresponds to the fixed point $\infty$ of $R$ under the hybrid conjugacy. 
It now follows that $\sigma$ has exactly $d-1$ critical points (counted with multiplicities) and at most two critical values in $K(\sigma)$. More precisely, if $\mathcal{T}$ is a star-tree, then $\sigma$ has a unique critical point of multiplicity $d-1$ and hence a single critical value in $K(\sigma)$; otherwise, $\sigma$ has exactly two critical values $K(\sigma)$ one of which is $y_c$. 
 
Consider a simple arc $\gamma_\sigma\subset\Omega$ connecting these two critical values of $\sigma$ (if $\sigma$ has only one critical value in $K(\sigma)$, then we choose $\gamma_\sigma$ to be an arc connecting this critical value to $y_c$). The existence of a hybrid conjugacy between $\sigma$ and $R$ implies that $\mathcal{T}_{\gamma_\sigma}(\sigma):=\sigma^{-1}(\gamma_\sigma)$ is a tree with a plane bicolored structure. We denote this combinatorial tree by $\mathcal{T}(\sigma)$, and note that it is isomorphic to $\mathcal{T}(R)\cong\mathcal{T}$. Moreover, the root $v_b$ of $\mathcal{T}(R)$ defines a root point for $\mathcal{T}(\sigma)$, and this root corresponds to the vertex $y_c$ of $\mathcal{T}_{\gamma_\sigma}(\sigma)$.
Abusing notation, we denote this root point of $\mathcal{T}(\sigma)$ by $v_b$. One can think of $\mathcal{T}(\sigma)$ as an analogue of dessin d'enfants for the Schwarz reflection map $\sigma$.

The degree $d+1$ polynomial $f$ has $d$ critical points in the complex plane (counted with multiplicities). As $f$ is injective on $\D$, it has no critical points in $\D$. Recall also that $f$ has a unique (simple) critical point on $\mathbb{S}^1$ with associated critical value $y_c$. Hence, we conclude that $f$ has precisely $d-1$ critical points in $\C\,\setminus\,\overline{\D}$ (counted with multiplicities). Since $\sigma= f\circ \eta\circ (f\vert_{\D})^{-1}$, we see that these critical points are given by $\eta((f\vert_{\D})^{-1}(\mathrm{crit}(\sigma)))$, and they are mapped by $f$ to the two critical values of $\sigma$ (respectively, to the unique critical value and $y_c$) in $K(\sigma)$. Therefore, $f$ is a Shabat polynomial.
\begin{figure}
\captionsetup{width=0.98\linewidth}
\begin{tikzpicture}
\node[anchor=south west,inner sep=0] at (6.6,0) {\includegraphics[width=0.45\textwidth]{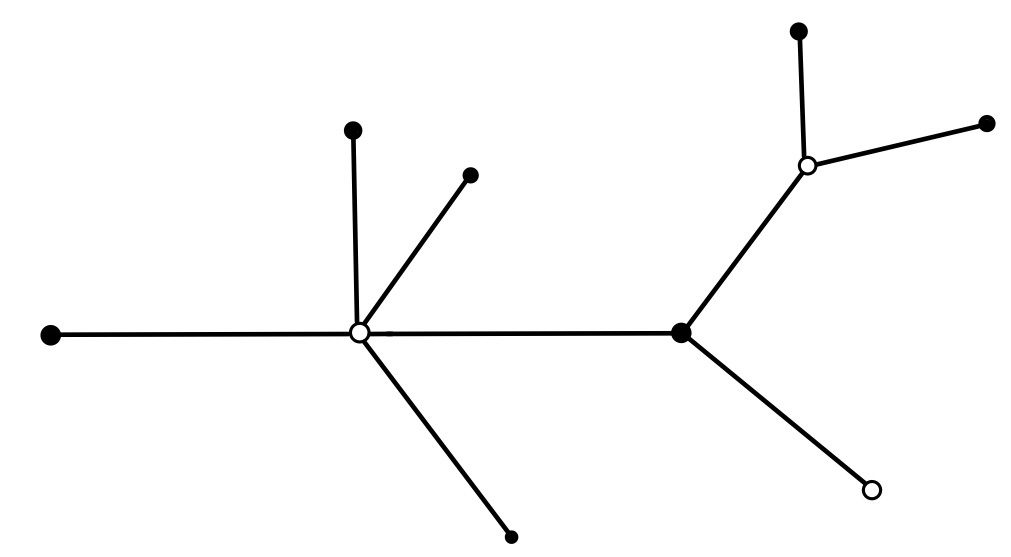}};
\node[anchor=south west,inner sep=0] at (-0.5,0) {\includegraphics[width=0.54\textwidth]{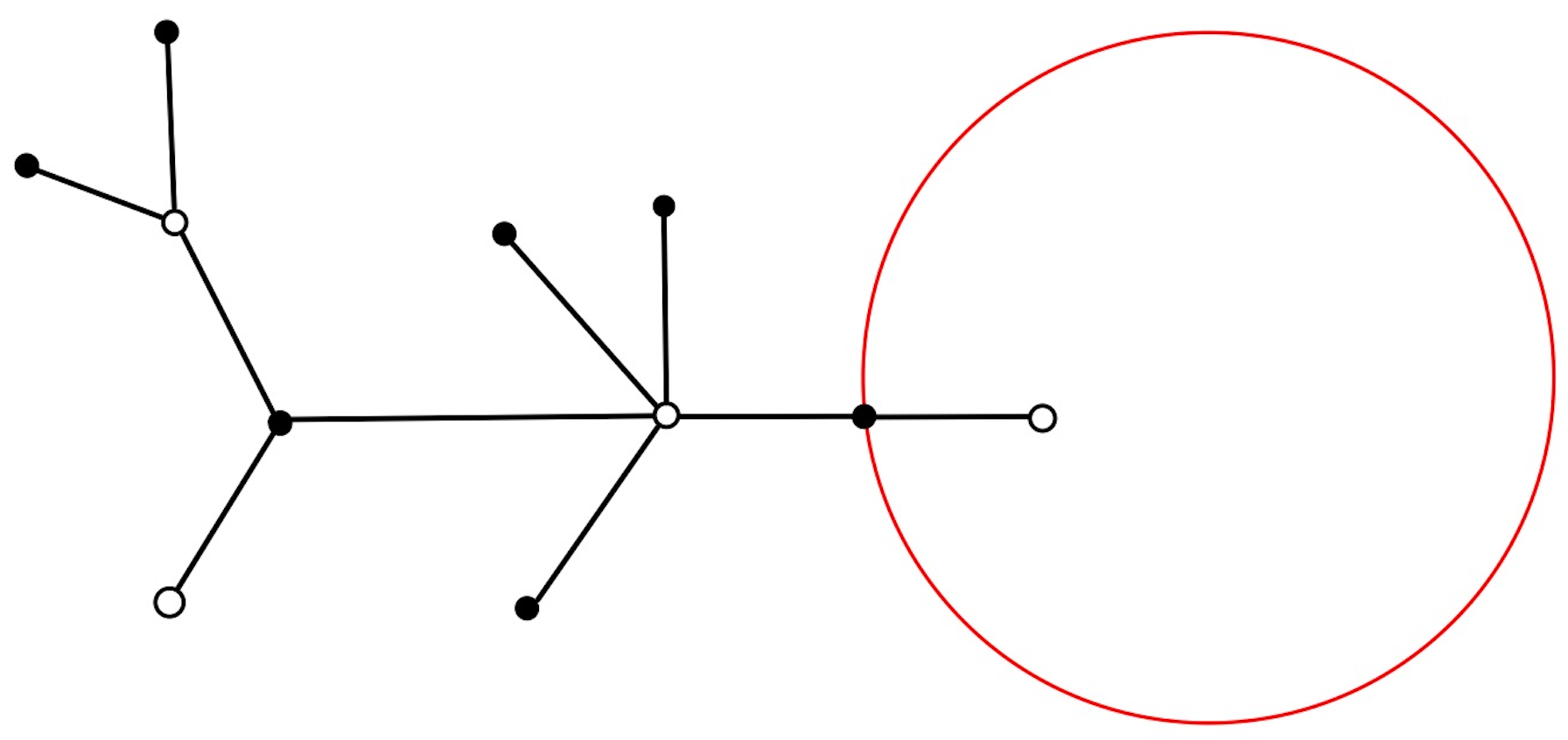}}; 
\node at (1.8,2.8) {$\mathcal{T}^{\mathrm{aug}}$};
\node at (9.8,2.8) {$\mathcal{T}$};
\node at (5.4,1.6) {$\D$};
\node at (1.4,1) {\begin{tiny}$\mathrm{emb}(v_w')$\end{tiny}};
\draw [->,line width=0.5pt] (1.8,1.2) to (2.2,1.32);
\node at (4.6,0.6) {\begin{tiny}$\mathrm{emb}(v_w)$\end{tiny}};
\draw [->,line width=0.5pt] (4.4,0.8) to (4.1,1.32);
\node at (2.9,0.6) {\begin{tiny}$\mathrm{emb}(v_b)$\end{tiny}};
\draw [->,line width=0.5pt] (3,0.8) to (3.2,1.32);
\node at (7,1) {$y_c$};
\end{tikzpicture}
\caption{Relation between the dessin d'enfants $\mathcal{T}^{\mathrm{aug}}\cong f^{-1}(\gamma_\sigma)$ of $f$ and the dessin d'enfants $\mathcal{T}\cong\sigma^{-1}(\gamma_\sigma)$ of the Schwarz reflection map $\sigma$.}
\label{augmented_dessin_fig}
\end{figure}

It is easy to see from the relation between $f$ and $\sigma$ that 
\begin{equation}
f^{-1}(\gamma_\sigma)=\eta((f\vert_{\overline{\D}})^{-1}(\sigma^{-1}(\gamma_\sigma)))\ \bigcup\ (f\vert_{\overline{\D}})^{-1}(\gamma_\sigma).
\label{dessin_f_eqn}
\end{equation}
Moreover, the closed arc $(f\vert_{\overline{\D}})^{-1}(\gamma_\sigma)$ is contained in $\overline{\D}$ and intersects $\mathbb{S}^1$ only at $(f\vert_{\overline{\D}})^{-1}(y_c)$, and $\eta((f\vert_{\overline{\D}})^{-1}(\sigma^{-1}(\gamma_\sigma)))$ is a topological tree (with $d$ edges) contained in $\widehat{\C}\,\setminus\,\D$ intersecting $\mathbb{S}^1$ only at $(f\vert_{\overline{\D}})^{-1}(y_c)$ (see Figure~\ref{augmented_dessin_fig}).
We conclude that combinatorially, the dessin d'enfants of $f$ is obtained by adding an edge to $\mathcal{T}^{\mathrm{op}}$ at the vertex $v_b$ (note the appearance of the orientation-reversing map $\eta$ in Relation~\eqref{dessin_f_eqn}). We denote this combinatorial tree by $\mathcal{T}^{\mathrm{aug}}$, and call the newly added vertex $v_w$ (it corresponds to the endpoint of $(f\vert_{\overline{\D}})^{-1}(\gamma_\sigma)$ different from $(f\vert_{\overline{\D}})^{-1}(y_c)$).

Consequently, $\mathcal{T}^{\mathrm{aug}}$ has a black vertex $v_b$ of valence $2$ with an adjacent white vertex $v_w$ of valence $1$. We further note that $f$ is univalent on the closed disk $\overline{\D}$ which has the vertex $(f\vert_{\overline{\D}})^{-1}(y_c)$ on its boundary, and contains the vertex $(f\vert_{\overline{\D}})^{-1}(y_f)$ (where $y_f$ is the unique `free' critical value of $\sigma$ in $K(\sigma)$ different from $y_c$) in its interior.

(2) We now consider an isomorphism $emb$ between the combinatorial bicolored plane tree $\mathcal{T}^{\mathrm{aug}}$ and its planar realization $f^{-1}(\gamma_\sigma)$. By construction, $emb(v_b)=(f\vert_{\overline{\D}})^{-1}(y_c)\in\mathbb{S}^1$ and $emb(v_w)=(f\vert_{\overline{\D}})^{-1}(y_f)\in\D$, where $y_f$ is the unique critical value of $\sigma$ (in $K(\sigma)$) different from $y_c$.
\end{proof}

We now wish to give an explicit description of $\chi^{-1}(\mathfrak{L}_{\mathcal{T}})$ as a real two-dimensional family of Schwarz reflections. 
To this end, we study an extended space of Schwarz reflections allowing possibly disconnected Julia sets, so that $\chi^{-1}(\mathfrak L_\cT)$ lies inside of the extended space as the connectedness locus. The extension itself is a real two dimensional space of Schwarz reflections whose actions over the rank zero tile are given by restrictions of $\mathcal R_d$.

Observe that by Proposition~\ref{shabat_good_uni_prop} and Theorem~\ref{shabat_classification_theorem}, all $[\Omega,\sigma]\in\chi^{-1}(\mathfrak{L}_{\mathcal{T}})$, the polynomial uniformizations $f:\D\to\Omega$ are equivalent Shabat polynomials. After possibly conjugating $\sigma$ by an affine map (which amounts to replacing $\Omega$ by an affine image of it), we can require that all such $\sigma$ have the same marked critical values. Then, the corresponding uniformizing polynomials $f$ only differ by pre-composition by an affine map.
Instead of fixing the domain of univalence $\D$ and varying the polynomial uniformizations (that differ by pre-composition by affine maps), it will be slightly more convenient to fix a polynomial uniformization and restrict it to various disks of univalency (that are affine images of $\D$). This leads to the following space of Schwarz reflections.

\noindent We recall the notation $v_b$, $v_w'$ introduced above Defintion~\ref{fd_slice_tree_def}.

\begin{definition}\label{shabat_schwarz_family_def}
Fix a degree $d+1$ Shabat polynomial $\pmb{f}$ produced by Proposition~\ref{shabat_good_uni_prop} together with an isomorphism $\pmb{emb}$ of the combinatorial tree $\mathcal{T}^{\mathrm{aug}}$ (which is the dessin d'enfants of $\pmb{f}$) and a planar realization of it. We also set $\pmb{v_b}:=\pmb{emb}(v_b)$, $\pmb{v_w}:=\pmb{emb}(v_w)$, $\pmb{v_w'}=\pmb{emb}(v_w')$, $\pmb{f}(\pmb{v_b})=y_c$, and $\pmb{f}(\pmb{v_w})=\pmb{f}(\pmb{v_w'})=y_f$.

\noindent Define
$$
S_{\mathcal{T}}:=\{ a\in\C: \pmb{v_w}\in\Delta_a:=B(a,\vert \pmb{v_b}-a\vert)\ \textrm{and}\ \pmb{f}\vert_{\overline{\Delta_a}}\ \textrm{is\ univalent}\},
$$ 
and
$$
S_{\mathcal{T}}^{\mathrm{simp}}:=\{a\in S_{\mathcal{T}}: \textrm{the\ unique\ cusp}\ y_c\ \textrm{on}\ \partial\Omega_a\ \textrm{is\ simple;\ i.e.,\ of\ type}\ (3,2)\}.
$$
Finally, we set $\Omega_a:= \pmb{f}(\Delta_a)$. 
\end{definition}

\begin{remark}
(1) Since $\pmb{v_b}\in\mathbb{S}^1$ and $\pmb{v_w}\in\D$, we have that $\Delta_0=\D$ and thus $0\in S_\cT$ (see Proposition~\ref{shabat_good_uni_prop}).

(2) For each $a\in S_\cT$, the quadrature domain $\Omega_a$ contains $y_f$ and its boundary $\partial\Omega_a$ has a conformal cusp at $y_c$. In particular, $$\overline{\Delta_a}\cap \pmb{f}^{-1}(\{y_c,y_f\})=\{\pmb{v_b},\pmb{v_w}\}.$$
\end{remark}

We denote the reflection in the circle $\partial\Delta_a$ by $\eta_a$, and the Schwarz reflection map of the quadrature domain $\Omega_a$ by $\sigma_a=\pmb{f}\circ\eta_a\circ (\pmb{f}\vert_{\overline{\Delta_a}})^{-1}$. By definition, the critical points of the Schwarz reflection map $\sigma_a$ are given by $$\{\pmb{f}(\eta_a(\zeta)): \pmb{f}'(\zeta)=0,\ \zeta\neq \pmb{v_b}\}.$$ If $\pmb{v_b}$ is the only critical point of $\pmb{f}$ over $y_c$ (see Figure~\ref{star_subtree_fig}), then the only critical values of $\sigma_a$ are $y_f, \infty$. Otherwise, the set of critical values of $\sigma_a$ is $\{y_c,y_f,\infty\}$. In light of this fact, we call the map $\sigma_a$ a Belyi map.

\begin{remark}\label{assumption_explanation}
(1) Abusing notation, we will identify the parameter space $S_\cT$ with the associated family of Schwarz reflection maps $\{\sigma_a:\overline{\Omega_a}\to\widehat{\C}: a\in S_\cT\}$.

(2) The family $S_\cT$ can be regarded as a generalization of the family of Schwarz reflection maps arising from the cubic Chebyshev polynomial studied in \cite{LLMM3}. 
\end{remark}

\begin{proposition}\label{connected_critical}
For $a\in S_{\mathcal{T}}$, the following are equivalent.
\begin{enumerate}[leftmargin=8mm]
\item $y_f\in K(\sigma_a)$.

\item $T^\infty(\sigma_a)$ is a simply connected domain.

\item $K(\sigma_a)$ is connected. 
\end{enumerate}
\end{proposition}
\begin{proof}
($1\implies2$) Let $E^k$ be the union of the tiles of $\sigma_a$ of rank $\leq k$. 

Note that since $\pmb{v_w}\in\Delta_a$, we have $y_f=f(\pmb{v_w})\in\Omega_a=f(\Delta_a)$. By definition, $y_c\notin T^0(\sigma_a)$ (see \S~\ref{schwarz_background_subsec} for the definition of $T^0(\sigma_a)$). Thus, the only critical value of $\sigma_a$ in $\Int{T^0(\sigma_a)}$ is $\infty$. Further, $\sigma_a^{-1}(\infty)=\{f(a)\}$; i.e., $\sigma_a$ maps $f(a)\in\Omega_a$ to $\infty\in \Int T^0(\sigma_a)$ with local degree $d+1$ (equivalently, $f(a)$ is a critical point of $\sigma_a$ multiplicity $d$). In light of the above discussion and Proposition~\ref{simp_conn_quad}, we have that 
$$
\sigma_a:\big(\sigma_a^{-1}\left(\Int{T^0(\sigma_a)}\right),f(a)\big)\to\left(\Int{T^0(\sigma_a)},\infty\right)
$$ 
is a degree $(d+1)$ branched cover branched only at $f(a)$. In particular, $\sigma_a$ maps each connected component of $\sigma_a^{-1}\left(\Int{T^0(\sigma_a)}\right)$ as a branched cover onto $T^0(\sigma_a)$, and thus each such component must contain a $\sigma_a-$preimage of $\infty$. But the fact that $\sigma_a^{-1}(\infty)$ is a singleton implies that $\sigma_a^{-1}\left(\Int{T^0(\sigma_a)}\right)$ must be connected. Applying the Riemann-Hurwitz formula to this branched covering, we obtain
\begin{equation}
\chi_{\mathcal{E}}\big(\sigma_a^{-1}(\Int{T^0(\sigma_a)})\big)= (d+1)\cdot \chi_{\mathcal{E}}\big(\Int{T^0(\sigma_a)}\big) - d =1,
\label{Riemann_Hurwitz_eqn}
\end{equation}
where $\chi_{\mathcal{E}}$ stands for the Euler characteristic. (In the above equality, we used the fact that the Euler characteristic of the simply connected domain $\Int{T^0(\sigma_a)}$ is $1$.) By Relation~\eqref{Riemann_Hurwitz_eqn}, the Euler characteristic of $\sigma_a^{-1}(\Int{T^0(\sigma_a)})$ is $1$, and hence $\sigma_a^{-1}(\Int{T^0(\sigma_a)})$ is a simply connected domain (cf. \cite[\S~1.7.4]{Lyu25}, \cite[\S~5.3]{Bea91}). Moreover, we have $\partial T^0(\sigma_a)\subset\partial\sigma_a^{-1}(T^0(\sigma_a))$. Hence, $\Int{E^1}$ is a simply connected domain. 

If $y_f\in K(\sigma_a)$, then every tile of rank $\geq 2$ is unramified, and we can iterate the arguments of the previous paragraph to conclude that $\Int{E^k}$ is a simply connected domain, for each $k$. Since $T^\infty(\sigma_a)=\bigcup_{k\geq 0}\Int{E^k}$ (cf. \cite[Proposition 2.3]{PartI}), we conclude that $T^\infty(\sigma_a)$ is a simply connected domain.

($2\implies3$) The complement of a simply connected domain is a full continuum.

($3\implies1$) Suppose that $K(\sigma_a)$ is connected. If $y_f\in T^\infty(\sigma_a)$, then the tile containing a critical point of $\sigma_a$ with corresponding critical value $y_f$ would be ramified, and disconnect $K(\sigma_a)$ (cf. \cite[Proposition~5.23]{LLMM1}). Therefore, $y_f$ must lie in $K(\sigma_a)$.
\end{proof}

Proposition~\ref{connected_critical} leads to the following definition.

\begin{definition}[Connectedness locus and escape locus]\label{conn_escape_def}
The connectedness locus of the family $S_\cT$ is defined as $$\cC(S_\cT)=\{a\in S_{\mathcal{T}}: y_f\in K(\sigma_a)\}=\{a\in S_{\mathcal{T}}: K(\sigma_a)\ \textrm{is\ connected}\}.$$ The complement of the connectedness locus in the parameter space $S_{\mathcal{T}}$ is called the \emph{escape locus}. We similarly define $\cC(S_\cT^{\mathrm{simp}})$ as the set of parameters $a\in S_\cT^{\mathrm{simp}}$ with connected $K(\sigma_a)$.
\end{definition}

The next proposition characterizes the connectedness locus of $S_\cT$ as those maps in $S_\cT$ that have anti-Farey as their external map.

\begin{proposition}\label{conn_locus_srd_prop}
$S_\cT\cap\mathcal{S}_{\mathcal{R}_d}=\cC(S_\cT)$. 
\end{proposition}
\begin{proof}
Clearly, if $(\Omega_a,\sigma_a)\in S_\cT\cap\mathcal{S}_{\mathcal{R}_d}$, then the tiling set $T^\infty(\sigma_a)$ is biholomorphic to a round disk and hence its complement $K(\sigma)$ is connected.

Conversely, let $\sigma_a\in\cC(S_\cT)$. By Proposition~\ref{connected_critical}, the critical value $y_f$ lies in $K(\sigma_a)$, and hence by \cite[Proposition~3.3]{PartI}, the tiling set dynamics of $\sigma_a$ is conformally conjugate to the anti-Farey map $\mathcal{R}_d$. Hence, $(\Omega_a,\sigma_a)\in S_\cT\cap\mathcal{S}_{\mathcal{R}_d}$. We conclude that $S_\cT\cap\mathcal{S}_{\mathcal{R}_d}=\cC(S_\cT)$.
\end{proof}

\begin{proposition}\label{shabat_family_prop}
$\chi$ induces a bijection between $\cC(S_\cT)$ and $\faktor{\mathfrak{L}_{\mathcal{T}}}{\mathrm{Aut}(\C)}$, and restricts to a bijection between $\cC(S_\cT^{\mathrm{simp}})$ and $\faktor{\mathfrak{L}_{\mathcal{T}}^{\mathrm{simp}}}{\mathrm{Aut}(\C)}$.
\end{proposition}
\begin{proof}
Let us first note that no two maps $\sigma_{a_1},\sigma_{a_2}\in \cC(S_\cT)$ are M{\"o}bius conjugate. This is because any M{\"o}bius map conjugating $\sigma_{a_1}$ to $\sigma_{a_2}$ would fix the unique critical value $\infty$ in the tiling sets $T^\infty(\sigma_{a_i})$, the conformal cusp $y_c$ on the boundaries $\partial\Omega_{a_i}$ and the other critical value $y_f\in K(\sigma_{a_i})$ of $\sigma_{a_i}$.

We proceed to show that $\chi^{-1}(\mathfrak{L}_{\mathcal{T}})$ is contained in $\cC(S_\cT)$. To this end, let $(\Omega,\sigma)\in\chi^{-1}(\mathfrak{L}_{\mathcal{T}})$. By Proposition~\ref{shabat_good_uni_prop}, there exists a Shabat polynomial $f_1$ such that the dessin d'enfants of $f_1$ is isomorphic to $\mathcal{T}^{\mathrm{aug}}$ and $f_1:\overline{\D}\to\overline{\Omega}$ is a homeomorphism.
After possibly replacing $\Omega$ by an affine image, we can assume that the unique cusp on $\partial\Omega$ is $y_c$ and the only other critical value of $\sigma$ in $K(\sigma)$ is $y_f$. Then, the proof of Proposition~\ref{shabat_good_uni_prop} shows that the vertices $v_b, v_w$ of the combinatorial tree $\mathcal{T}^{\mathrm{aug}}$ correspond to the critical and co-critical values $\pmb{x}_1:=(f_1\vert_{\overline{\D}})^{-1}(y_c), \pmb{x}_2:=(f_1\vert_{\overline{\D}})^{-1}(y_f)$, respectively.
\begin{figure}[ht!]
\captionsetup{width=0.98\linewidth}
\begin{tikzpicture}
\node[anchor=south west,inner sep=0] at (3.6,0) {\includegraphics[width=0.6\textwidth]{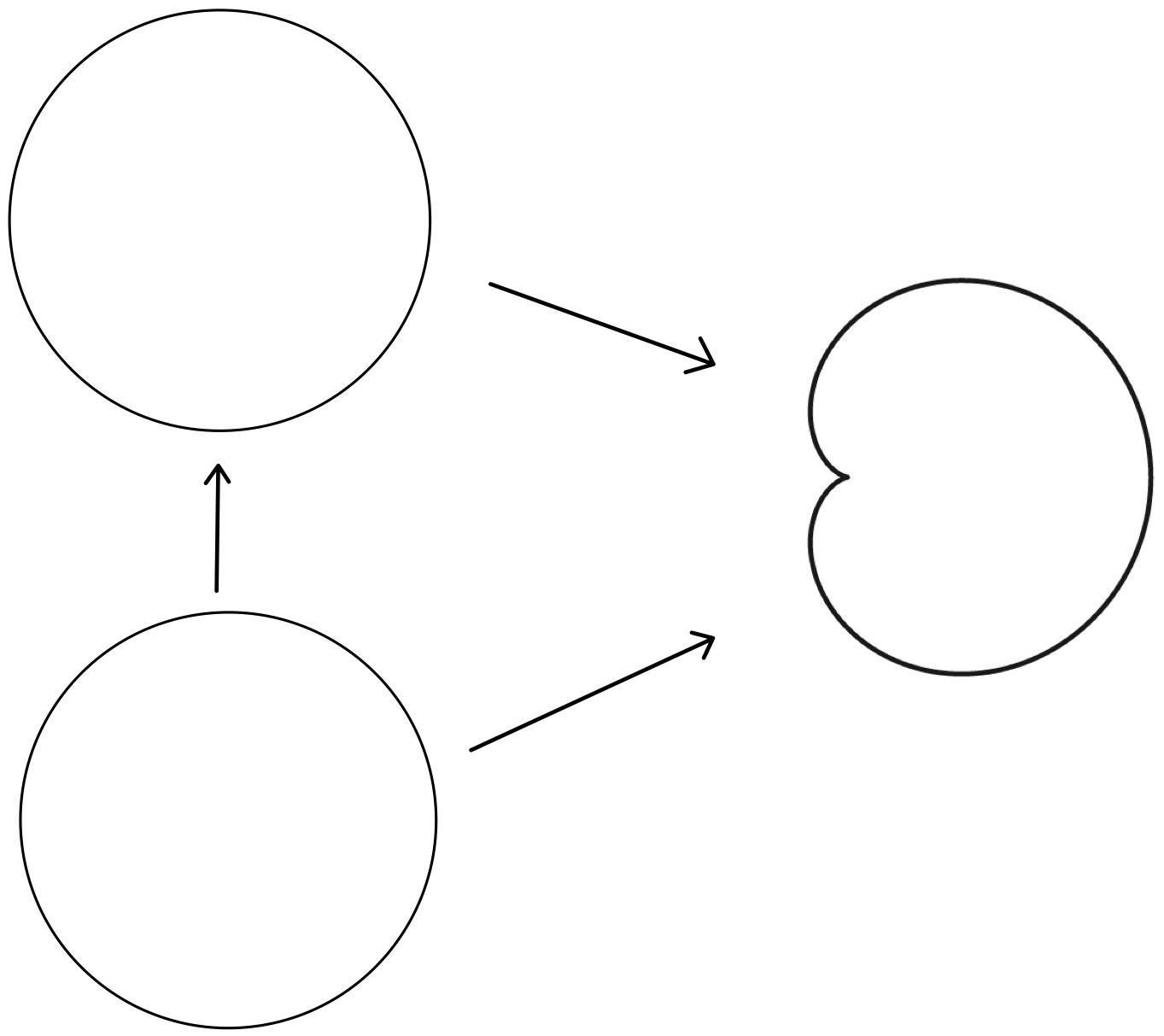}};
\node at (3.7,4.25) {$\D$};
\node at (7.5,1.8) {$\pmb{f}$};
\node at (7.5,5) {$f_1$};
\node at (3.8,0.4) {$\Delta_a$};
\node at (11,2.4) {$\Omega_a$};
\node at (9.12,3.64) [circle,fill=blue,inner sep=1.5pt] {};
\node at (9.45,3.64) {\textcolor{blue}{$y_c$}};
\node at (4.8,3.3) {$A$};
\node at (5.05,1.4) [circle,fill=black,inner sep=1.5pt] {};
\node at (5.05,1.16) {$a$};
\node at (5.7,2) [circle,fill=red,inner sep=1.5pt] {};
\node at (5.6, 2.24) {\textcolor{red}{$\pmb{v_w}$}};
\node at (5.05,5.4) [circle,fill=black,inner sep=1.5pt] {};
\node at (5.05, 5.08) {$0$};
\node at (4.5,5.9) [circle,fill=red,inner sep=1.5pt] {};
\node at (4.2,6) {\textcolor{red}{$\pmb{x}_2$}};
\node at (6.4,1) [circle,fill=blue,inner sep=1.5pt] {};
\node at (6.7,0.8) {\textcolor{blue}{$\pmb{v_b}$}};
\node at (6.2,6.06) [circle,fill=blue,inner sep=1.5pt] {};
\node at (6.56,6.12) {\textcolor{blue}{$\pmb{x}_1$}};
\node at (10.2,3.6) [circle,fill=black,inner sep=1.5pt] {};
\node at (10.2,3.28) {$\pmb{f}(a)$};
\node at (9.6,4.5) [circle,fill=red,inner sep=1.5pt] {};
\node at (9.96,4.5) {\textcolor{red}{$y_f$}};
\end{tikzpicture}
\caption{Pictured is the affine change of coordinate $A$ appearing in the proof of Proposition~\ref{shabat_family_prop}. The corresponding points are marked in the same color.}
\label{affine_change_fig}
\end{figure}

Theorem~\ref{shabat_classification_theorem} now implies that there exists an affine map $A$ with $A(\pmb{v_b})=\pmb{x}_1$, $A(\pmb{v_w})=\pmb{x}_2$, and $\pmb{f}\equiv f_1\circ A$. Setting $a:=A^{-1}(0)$, we conclude that $A^{-1}(\D)$ is a round disk centered at $a$ having $\pmb{v_b}$ on its boundary such that $\pmb{f}$ is univalent on $A^{-1}(\overline{\D})=\overline{\Delta_{a}}$ (see Figure~\ref{affine_change_fig}). Thus, $\Omega=f_1(\D)=\pmb{f}(A^{-1}(\D))=\Omega_{a}$, and hence $\sigma\equiv\sigma_{a}$. Finally, $\pmb{v_w}=A^{-1}(\pmb{x}_2)\in\Delta_{a}$. It follows that $a\in S_{\mathcal{T}}$, and hence $(\Omega,\sigma)\in S_\cT$. Finally, the fact that the $\chi-$preimage of any map in $\mathfrak{L}_{\mathcal{T}}\subset\cF_d$ has a connected non-escaping set implies that $(\Omega,\sigma)\in\cC(S_\cT)$.

Conversely, let $(\Omega_a,\sigma_a)\in \cC(S_\cT)$. We need to argue that $R:=\chi(\sigma_a)\in\mathfrak{L}_{\mathcal{T}}$. Since $\sigma_a$ has at most two critical values in $K(\sigma_a)$, it follows that $R$ has at most two critical values in $\mathcal{K}(R)$. Hence, $R$ has at most three critical values. Moreover, $R$ has three critical values if and only if $y_c$ is a critical value of $\sigma_a$. Since the hybrid conjugacy between $\sigma_a$ and $R$ carries $y_c$ to $\infty$, we conclude that if $R$ has three critical values, then $\infty$ is one of them. Therefore, $R\in\mathfrak{L}$.

We will now describe the dessin d'enfants of $R$ by dualizing some of the arguments of Proposition~\ref{shabat_good_uni_prop}. Let $\gamma_a\subset\overline{\Omega_a}$ be a simple arc connecting $y_c$ and $y_f$. Since $\pmb{f}:\overline{\Delta_a}\to\overline{\Omega_a}$ is a homeomorphism, it follows that $\pmb{f}^{-1}(\gamma_a)\cap\overline{\Delta_a}$ is a simple arc that connects $\pmb{v_b}$ to $\pmb{v_w}$. Note that the embedded tree $\pmb{f}^{-1}(\gamma_a)\setminus\Delta_a$ is isomorphic to $\mathcal{T}^{\mathrm{op}}$ as a combinatorial bicolored plane tree such that the isomorphism sends $\pmb{v_b}=(\pmb{f}\vert_{\overline{\Delta_a}}^{-1})(y_c)$ to $v_b$.
Thus, the tree
$$
\sigma_a^{-1}(\gamma_a)=\pmb{f}(\eta_a(\pmb{f}^{-1}(\gamma_a)\setminus\Delta_a))\subset\overline{\Omega_a}
$$
is isomorphic to $\mathcal{T}$ as a combinatorial bi-colored plane tree. Clearly, the hybrid conjugacy yields an isomorphism between $\sigma_a^{-1}(\gamma_a)$ and a planar realization of the dessin d'enfants $\mathcal{T}(R)$ of $R$. We conclude that $\mathcal{T}(R)\cong\mathcal{T}$; i.e., $R\in\mathfrak{L}_{\mathcal{T}}$. 

Finally, the second statement follows from \cite[Theorem~5.1]{PartI}.
\end{proof}

\begin{corollary}\label{non_emp_cor}
$\Int{S_{\mathcal{T}}}\neq\emptyset$.
\end{corollary}
\begin{proof}
This follows from Lemma~\ref{slice_non_emp_lem} and Proposition~\ref{shabat_family_prop}.
\end{proof}

\begin{corollary}\label{schwarz_tree_cor}
Let $a\in S_{\mathcal{T}}$, $\gamma_0\subset\overline{\Omega_a}$ be a simple arc connecting $y_c$ and $y_f$. Then, the embedded tree
$$
\sigma_a^{-1}(\gamma_a)=\pmb{f}(\eta_a(\pmb{f}^{-1}(\gamma_a)\setminus\Delta_a))\subset\overline{\Omega_a}
$$ 
is isomorphic to $\mathcal{T}$ as a combinatorial bi-colored plane tree, and the isomorphism identifies the vertex $y_c$ of $\sigma_a^{-1}(\gamma_a)$ with the root of $\mathcal{T}$.
\end{corollary}

Corollary~\ref{schwarz_tree_cor} can be restated as follows: $\sigma_a$ is a Belyi map whose \emph{dessin d'enfants} is given by $\mathcal{T}$.

The relationship between the dessin d'enfants of $\pmb{f}$ and that of $\sigma_a$ can be used to prove that the family $S_\cT$ is quasiconformally closed.

\begin{proposition}\label{qc_closed_prop}
Let $a\in S_{\mathcal{T}}$, $\mu$ be a $\sigma_a-$invariant Beltrami coefficient on $\widehat{\C}$, and $\Phi$ be the
quasiconformal map solving the Beltrami equation with coefficient $\mu$ such that $\Phi$ fixes $y_c, y_f$, and $\infty$. Then, there exists $a'\in S_{\mathcal{T}}$ such that $\Phi(\Omega_a)=\Omega_{a'}$, and  $\Phi\circ\sigma_a\circ\Phi^{-1}=\sigma_{a'}$ on $\Omega_{a'}$.
\end{proposition}

\begin{proof}[Sketch of Proof]
The assumption that $\mu$ is $\sigma_a-$invariant implies that $\widecheck{\sigma}:=\Phi\circ\sigma_a\circ\Phi^{-1}$ is anti-meromorphic on $\widecheck{\Omega}:=\Phi(\Omega_a)$, and continuously extends to the identity map on $\partial\widecheck{\Omega}$. Thus, $\widecheck{\Omega}$ is a simply connected quadrature domain with Jordan boundary and Schwarz reflection map $\widecheck{\sigma}$. Also, $\widecheck{\sigma}:\widecheck{\sigma}^{-1}(\Int{\widecheck{\Omega}^c})\rightarrow\Int{\widecheck{\Omega}^c}$ is a $(d+1):1$ branched cover with a critical point of local degree $d+1$. According to \cite[Proposition~3.3]{PartI}, there is a degree $d+1$ polynomial $g$ which carries $\overline{\D}$ injectively onto $\overline{\widecheck{\Omega}}$. Thus, $\widecheck{\sigma}\equiv g\circ\eta\circ\left(g\vert_{\overline{\D}}\right)^{-1}$ on $\overline{\widecheck{\Omega}}$. The fact that $\widecheck{\sigma}$ has at most two critical points in $K(\widecheck{\sigma})$ (recall that the same is true for $\sigma$) implies that $g$ has at most two finite critical values. Thus, $g$ is a Shabat polynomial.

That the dessin d'enfants of $\sigma_a$ is isomorphic to $\mathcal{T}$ implies that the same is true for $\widecheck{\sigma}$. Arguing as in Proposition~\ref{shabat_good_uni_prop}, one concludes that the dessin d'enfants of $g$ is isomorphic to $\mathcal{T}^{\mathrm{aug}}$. As $g$ and $\pmb{f}$ have isomorphic dessins d'enfants and the same marked critical values, one can use the arguments of Proposition~\ref{shabat_family_prop} to deduce that $\widecheck{\sigma}$ lies in the family $S_\cT$.
\end{proof}

\subsection{External dynamics in the escape locus}\label{external_dynamics_subsec}

\begin{definition}[Depth]\label{def_depth}
For any $a$ in the escape locus of $S_\cT$, the \emph{smallest} positive integer $n(a)$ such that $\sigma_a^{\circ n(a)}(y_f)\in T^0(\sigma_a)$ is called the \emph{depth} of $a$.
\end{definition}

\begin{lemma}\label{schwarz_group}
(1) For $a\in\cC(\mathcal{S_{\mathcal{T}}})$, the map $\sigma_a:T^\infty(\sigma_a)\setminus\Int{T^0(\sigma_a)}\to T^\infty(\sigma_a)$ is conformally conjugate to $\mathcal{R}_d:\Int{\mathcal{Q}}\setminus\Int{\mathcal{Q}_1}\to\Int{\mathcal{Q}}$.

(2) For $a\in S_{\mathcal{T}}\setminus\cC(\mathcal{S_{\mathcal{T}}})$, 
$$
\sigma_a:\displaystyle\bigcup_{n=1}^{n(a)} \sigma_a^{-n}(T^0(\sigma_a))\to\displaystyle\bigcup_{n=0}^{n(a)-1} \sigma_a^{-n}(T^0(\sigma_a))
$$ 
is conformally conjugate to 
$$
\mathcal{R}_d:\displaystyle\bigcup_{n=1}^{n(a)} \mathcal{R}_d^{-n}(\mathcal{Q}_1)\to\displaystyle\bigcup_{n=0}^{n(a)-1}\mathcal{R}_d^{-n}(\mathcal{Q}_1).
$$
\end{lemma}
\begin{proof}
Since $\mathcal{Q}_1$ is simply connected, we can choose a homeomorphism $$\psi_a:\mathcal{Q}_1\to T^0(\sigma_a)$$ such that it is conformal on the interior. We can further assume that $\psi_a(0)=\infty$, and its continuous extension sends the cusp point $1\in\partial\mathcal{Q}_1$ to the point $y_c\in\partial T^0(\sigma_a)$.

Note that $\sigma_a:\sigma_a^{-1}(T^0(\sigma_a))\to T^0(\sigma_a)$ is a $(d+1):1$ branched cover branched only at $\pmb{f}(a)$, and $\mathcal{R}_d:\mathcal{R}_d^{-1}(\mathcal{Q}_1)\to\mathcal{Q}_1$ is a $(d+1):1$ branched cover branched only at $\rho_1(0)$. Moreover, $\sigma_a$ fixes $\partial T^0(\sigma_a)$ pointwise, and $\mathcal{R}_d$ fixes $\partial\mathcal{Q}_1$ pointwise.

This allows one to lift $\psi_a$ to a conformal isomorphism from $\mathcal{R}_d^{-1}(\mathcal{Q}_1)$ onto $\sigma_a^{-1}(T^0(\sigma_a))$ such that the lifted map sends $\rho_1(0)$ to $\pmb{f}(a)$, and continuously matches with the initial map $\psi_a$ on $\mathcal{Q}_1$. We denote this extended conformal isomorphism by $\psi_a$. By construction, $\psi_a$ is equivariant with respect to the actions of $\mathcal{R}_d$ and $\sigma_a$ on $\mathcal{R}_d^{-1}(\mathcal{Q}_1)$ and $\partial\sigma_a^{-1}(T^0(\sigma_a))$, respectively. 

(1) If $a\in\cC(\mathcal{S_{\mathcal{T}}})$, then every tile of $T^\infty(\sigma_a)$ of rank greater than one maps diffeomorphically onto $\sigma_a^{-1}(T^0(\sigma_a))$ under some iterate of $\sigma_a$, and each tile of $\D_1$ (see \S~\ref{anti_farey_subsubsec} for the definition of $\D_1$) of rank greater than one maps diffeomorphically onto $\mathcal{R}_d^{-1}(\mathcal{Q}_1)$ under some iterate of $\mathcal{R}_d$. This fact, along with the equivariance property of $\psi_a$ mentioned above, enables us to lift $\psi_a$ to all tiles using the iterates of $\mathcal{R}_d$ and $\sigma_a$. This produces the desired biholomorphism $\psi_a$ between $\Int{\mathcal{Q}}$ and $T^\infty(\sigma_a)$ which conjugates $\mathcal{R}_d$ to $\sigma_a$.

(2) For $a\in S_{\mathcal{T}}\setminus\cC(\mathcal{S_{\mathcal{T}}})$, the above construction of $\psi_a$ can be carried out onto the tiles of $T^\infty(\sigma_a)$ that map diffeomorphically onto $\sigma_a^{-1}(T^0(\sigma_a))$, which includes all tiles of rank up to $n(a)$. This completes the proof.
\end{proof}

\begin{definition}[Dynamical Rays]\label{dyn_ray_schwarz}
The image of a $\pmb{\Gamma}_d-$ray at angle $\theta\in[0,\frac{1}{d+1})$ in $\mathcal{Q}$ under the map $\psi_a$ (see Lemma~\ref{schwarz_group}) is called a $\theta-$dynamical ray of~$\sigma_a$.
\end{definition}

Clearly, the image (of the tail) of a dynamical $\theta-$ray under $\sigma_a$ is (the tail of) a dynamical ray at angle $\mathcal{R}_d(\theta)$.

\section{The parameter space $S_\mathcal T$ is a quadrilateral}\label{conn_para_space_sec}

In this section we show that the parameter space $S_\mathcal T$ is a topological quadrilateral with boundary given by  real analytic arcs, and analyze the dynamics of the associated maps at those boundary components.

\subsection{Boundedness of $S_{\mathcal{T}}$}

Note that $\partial\Delta_a$ contains a fixed base point $\pmb{v_b}$. As $|a|\to\infty$, the radius $\vert a-\pmb{v_b}\vert$ of the disk $\Delta_a$ also goes to infinity. Hence, as $a\to\infty$ along some $\theta-$ray $\{\pmb{v_b}+R e^{i\theta}: R\in(0,\infty)\}$, the disk $\overline{\Delta_a}$ converges to some closed half-plane in $\widehat{\C}$.
Since the polynomial $\pmb{f}$ behaves like $c\cdot z^{d+1}$ near $\infty$ with $c\neq 0$ and $d+1\geq 3$, it follows that $\pmb{f}$ cannot be injective on a half-plane near $\infty$. Hence, for $\vert a\vert$ large enough, $\pmb{f}$ cannot be injective on $\Delta_a$.

This gives the following.

\begin{lemma}\label{para_space_bdd_lem}
The parameter space $S_{\mathcal{T}}$ is a bounded subset of $\C$.
\end{lemma}

\subsection{Dynamics near cusp points} 

Recall that the only singularities on the boundary of a quadrature domain are conformal cusps and double points. In this subsection, we collect some results on the local dynamics of a Schwarz reflection map near a cusp and deduce some immediate applications to our parameter space $S_\cT$. We refer the reader to \cite[Appendix~A.2]{PartI} for the notion of the \emph{type} of a conformal cusp.

\begin{definition}\label{hoc_pb_def}
\noindent\begin{enumerate}[leftmargin=8mm]
\item We define 
\begin{equation*}
\Gamma^{\mathrm{hoc}}:= \{a\in S_{\mathcal{T}}: y_c=\pmb{f}(\pmb{v_b})\ \textrm{is\ a\ cusp\ of\ type}\ (\nu,2),\ \nu\geq 5\},
\end{equation*}
and call $\Gamma^{\mathrm{hoc}}$ the \emph{higher order cusp locus}.
\vspace{1mm}

\item We define 
\begin{equation*}
\begin{split}
&\Gamma^{\perp}:=\{a\in\C: \pmb{v_w}\in\partial\Delta_a,\ \pmb{f}\vert_{\overline{\Delta_a}}\ \textrm{is\ univalent}\}\ \subset\\
& \textrm{The\ perpendicular\ bisector\ of\ the\ straight\ line\ segment\ joining}\ \pmb{v_w}\ \textrm{and}\ \pmb{v_b}.
\end{split}
\end{equation*}
\end{enumerate}
\end{definition}

After possibly postcomposing $\pmb{f}$ with an affine map, we can assume that $\pmb{f}$ has asymptotics $f(z)=(z-\pmb{v_b})^2+O((z-\pmb{v_b})^3)$ near $\pmb{v_b}$.

For $\theta\in\R/2\pi\Z$, denote by $\alpha_{\theta}$ the curve germ at $\pmb{v_b}$ that maps under $\pmb{f}$ to a straight line segment $\{y_c-r e^{2i\theta}: r\in[0,\epsilon]\}$ (for $\epsilon>0$ small).

The following lemma can be deduced from \cite[Remark~A.2, Proposition~A.4]{PartI} (also see \cite[Proposition~4.5]{LLMM3}).

\begin{lemma}\label{cusp_anal_dyn_lem}
Let $a\in S_{\mathcal{T}}$. Then the following are equivalent.
\begin{enumerate}[leftmargin=8mm]
\item $a\in\Gamma^{\mathrm{hoc}}$. 

\item The circle $\partial\Delta_a$ is an osculating circle to the curve germ $\alpha_{\theta}$ at $\pmb{v_b}$, where $\theta=\arg{(a-\pmb{v_b})}$.

\item The second iterate $\sigma_a^{\circ 2}$ has at least one attracting and at least one repelling direction in $\Omega_a$ at the cusp $y_c$. In particular, the unique free critical orbit of $\sigma_a$ (i.e., the forward orbit of $y_f$) non-trivially converges to $y_c$.
\end{enumerate}
\end{lemma}

Since $\sigma_a\in S_\cT$ has exactly one free critical value, it follows that $\sigma_a^{\circ 2}$ can have at most two attracting petals for a parabolic point (if there are two attracting petals, they must form a $2-$cycle under $\sigma_a$). Note also that every conformal cusp can be brought to the form $y^2=x^{2n+1}$ (for some $n\geq 1$), up to a conformal change of coordinates. The following proposition then follows from these observations and \cite[Proposition A.4]{PartI}.

\begin{proposition}\label{cusp_dyn_cor}
Let $a\in S_{\mathcal{T}}$. Then the following statements hold.
\begin{enumerate}[leftmargin=8mm]
\item $y_c$ is a $(\nu,2)-$cusp of $\partial\Omega_a$, where $\nu\in\{3,5,7\}$.

\item The invariant direction $\{y_c+re^{2i\arg{(a-\pmb{v_b})}}: r\in[0,\epsilon]\}$ is a repelling (respectively, attracting) direction for $\sigma_a$ at $y_c$ if $y_c$ is a $(3,2)$ or $(7,2)$ (respectively, $(5,2)$) cusp of $\partial\Omega_a$.
\end{enumerate}
\end{proposition}

\begin{lemma}\label{pb_hoc_disjoint_lem}
$\Gamma^{\mathrm{hoc}}\cap\overline{\Gamma^{\perp}}=\emptyset$.
\end{lemma}
\begin{proof}
For $a\in\Gamma^{\mathrm{hoc}}$, the cusp point $y_c$ non-trivially attracts the forward orbit of $y_f$. On the other hand, for $a\in\overline{\Gamma^{\perp}}$, the critical value $y_f$ lies on $\partial\Omega_a\setminus\{y_c\}$, and is thus fixed. Hence, $\overline{\Gamma^{\perp}}\cap\Gamma^{\mathrm{hoc}}=\emptyset$. 
\end{proof}

\subsection{Dynamics near double points}

A point $p\in\partial\Omega_a$ is said to be a \emph{double point} if for all sufficiently small $\varepsilon>0$, the intersection $B(p,\varepsilon)\cap\Omega_a$ is a union of two Jordan domains, and $p$ is a non-singular boundary point of each of them. In particular, two distinct non-singular (real-analytic) local branches of $\partial\Omega_a$ intersect tangentially at a double point $p$. One can further classify such double points according to the order of contact of the two real-analytic branches $\gamma_1$ and $\gamma_2$ of $\partial\Omega_a$ at $p$. Let $\iota_1$ and $\iota_2$ be the local Schwarz reflection maps associated with $\gamma_1$ and $\gamma_1$.

Recall that the \emph{order of contact} between two analytic curves which intersect at a point $p$ is the maximum $k$ such that the $k-$jets of the two curves agree at $p$.

\begin{proposition}\label{prop_double_point_germ}
	Suppose that $\gamma_1$ and $\gamma_2$ have a contact of order $k$ at $p$. Then $\iota_1\circ\iota_2$ is a parabolic germ of the form $z\mapsto z+a(z-p)^{k+1}+O((z-p)^{k+2})$ with $a\neq 0$.
\end{proposition}
\begin{proof}[Proof sketch]
	After a suitable analytic change of coordinates we may assume that $p=0$ and $\gamma_1$ is given locally by the real axis. The order of contact between $\gamma_1$ and $\gamma_2$ is then given by $\left(\min\left\{n:\frac{d^n \text{Im}(\gamma_2)}{d\text{Re}(\gamma_2)^n}\neq 0\right\}-1\right)$.
	
	On the other hand, if $\psi$ is a local analytic change of coordinates sending $\gamma_2$ to the real axis, so that one may parameterize $\gamma_2$ as $\gamma_2(t)=\psi^{-1}(t+0i)$, then the order of contact between $\gamma_1$ and $\gamma_2$ can be expressed as the largest $k$ such that for all $n\leq k$ the $n$th coefficient of the power series of $\psi^{-1}$ is real.
	
	A routine power series computation shows that the first $k$ terms of $\psi^{-1}$ are real if and only if the first $k$ terms of $\psi$ are real.
	
	Now in the given coordinates, $\iota_1(z)=\overline z$ and $\iota_2(z) = \psi^{-1}\left(\overline{\psi(z)}\right)$, so that $\iota_2\circ \iota_1 (z) = \psi^{-1}\left(\overline{\psi(\overline z)}\right)$. Now as exactly the first $k$ coefficients of $\psi$ are real, we conclude that $\overline{\psi(\overline z)}$ agrees with $\psi(z)$ up to the first $k$ terms of its power series, that is, $\overline{\psi(\overline z)}-\psi(z)=bz^{k+1}+O(z^{k+2})$ for some $b\neq 0$.
	
Finally, a standard computation involving the compositional inverse of a power series yields that 
$$
\iota_2\circ \iota_1 (z) = \psi^{-1}\left(\overline{\psi(\overline z)}\right)=z+c\cdot z^{k+1}+O(z^{k+2}),\ \textrm{for some}\ c\neq 0,
$$
in the chosen coordinates. This completes the proof of the proposition.
\end{proof}

We also note that in the above setting, $\iota_2\circ\iota_1$ is the inverse of $\iota_1\circ\iota_2$, and these two germs are anti-conformally conjugate via $\iota_1$. Further, $k$ is necessarily odd since otherwise the two branches $\gamma_1$ and $\gamma_2$ would cross at $p$. 

\begin{definition}
We define 
$$
\Gamma^{\mathrm{dp}}:=\{a\in\C: \pmb{v_w}\in\Delta_a,\ f\vert_{\Delta_a}\ \textrm{is\ univalent,\ and}\ \partial\Omega_a\ \textrm{has\ a\ double\ point}\},
$$
and call $\Gamma^{\mathrm{dp}}$ the \emph{double point locus}.
\end{definition}

We point out that if $a\in\Gamma^{\mathrm{dp}}$, then $y_f\in\Omega_a$ and hence $\partial\Omega_a$ has a unique cusp at~$y_c$.

We now study the local dynamics of $\sigma_a$ near a double point $p$ of $\partial\Omega_a$.

\begin{lemma}\label{dp_dyn_classification_lem}
Let $p$ be a double point on $\partial\Omega_a$. Then the following assertions hold.
\begin{enumerate}[leftmargin=8mm]
\item The two non-singular branches of $\partial\Omega_a$ at $p$ have contact of order one or three.

\item If the contact is of order one, then the two associated normal directions at $p$ are repelling directions for $\sigma_a^{\circ 2}$ and there is no attracting direction for $\sigma_a^{\circ 2}$ at $p$. 

\item If the contact is of order three, then the two associated normal directions at $p$ are the only attracting directions for $\sigma_a^{\circ 2}$, and the unique free critical orbit of $\sigma_a$ non-trivially converges to $p$. In particular, $\partial\Omega_a$ has at most one such double point. Furthermore, there are four repelling directions for $\sigma_a^{\circ 2}$.
\end{enumerate}
\end{lemma}
\begin{proof}
We denote the common tangent line for $\gamma_1$ and $\gamma_2$ at $p$ by $\ell$.
\smallskip

\noindent\textbf{Case 1 ($k=1$).} Up to second order, the local power series of $\iota_1, \iota_2$ depend only on the (signed) curvature of $\gamma_1, \gamma_2$ at $p$. Hence, it suffices to assume that $\iota_1, \iota_2$ are Schwarz reflections with respect to the osculating circles to $\gamma_1, \gamma_2$ at $p$ (cf. \cite[Sec.7]{davis74}). A simple computation now shows that the inward normal to $\gamma_1$ (respectively, the outward normal to $\gamma_1$) at $p$ is a repelling direction for $\gamma_2\circ\gamma_1$ (respectively, for $\gamma_1\circ\gamma_2$). In other words, these normal directions are repelling directions for $\sigma_a^{\circ 2}$. 
\begin{figure}
\captionsetup{width=0.98\linewidth}
\begin{tikzpicture}
\node[anchor=south west,inner sep=0] at (0,0) {\includegraphics[width=0.25\textwidth]{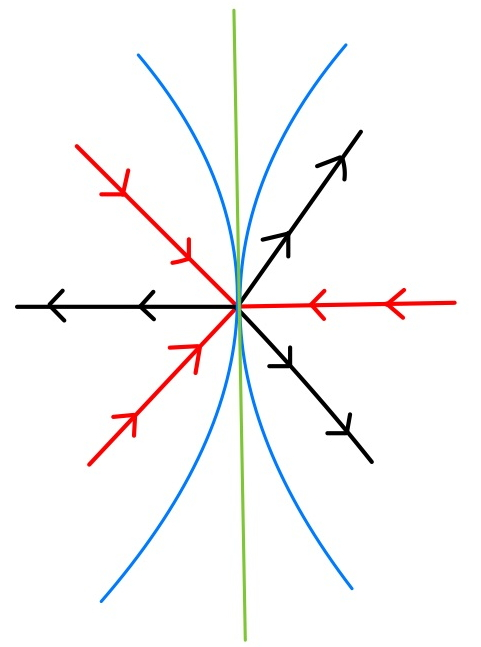}};
 \node[anchor=south west,inner sep=0] at (6.2,0) {\includegraphics[width=0.27\textwidth]{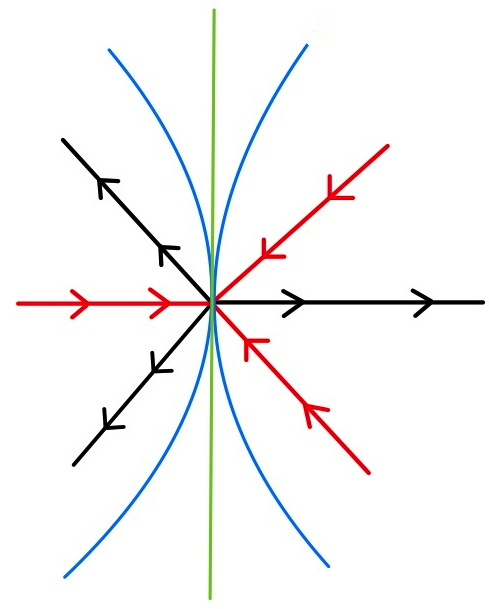}};
\node at (8.55,4.05) {$\gamma_1$};
\node at (2.5,4.05) {$\gamma_1$};
\node at (0.7,4.05) {$\gamma_2$};
\node at (6.75,4.05) {$\gamma_2$};
\node at (1.4,0.4) {$\ell$};
\node at (7.5,0.4) {$\ell$};
   \end{tikzpicture}
\caption{The situations considered in the two sub-cases 2a and 2b in the proof of Lemma~\ref{dp_dyn_classification_lem} are depicted. The arrows indicate attracting/repelling directions for the parabolic germ $\iota_2\circ\iota_1$.}
\label{triple_tangency_parabolic_fig}
\end{figure}
\smallskip

\noindent\textbf{Case 2 ($k=3$).} In this case, the germ $\iota_2\circ\iota_1$ has three attracting and three repelling directions. Note that the inward normal to $\gamma_1$ at $p$ is invariant under this germ. 
\smallskip

\noindent\textbf{Sub-case 2a.} If the inward normal to $\gamma_1$ at $p$ is an attracting direction for $\iota_2\circ\iota_1$, then the other two attracting directions of $\iota_2\circ\iota_1$ are on the opposite side of $\ell$ (see Figure~\ref{triple_tangency_parabolic_fig} (left)). Thus, these two directions are repelling directions for $\iota_1\circ\iota_2$ and hence for $\sigma_a^{\circ 2}$. On the other hand, the outward normal to $\gamma_1$ at $p$ is attracting for $\iota_1\circ\iota_2$ (since it is repelling for $\iota_2\circ\iota_1$). Therefore, the only attracting directions for $\sigma_a^{\circ 2}$ at $p$ are the above two normal vectors, and they are exchanged by $\sigma_a$ (see \cite[Proposition~5.15]{LLMM1} for the same situation in the circle-and-cardioid family of Schwarz reflections). By Fatou-type arguments, the unique free critical orbit of $\sigma_a$ must converge to $p$ asymptotic to this $2-$cycle of attracting directions (cf. \cite[Propositions~5.30, 5.32]{LLMM1}). The statement about the repelling directions for $\sigma_a^{\circ 2}$ also follows.
\smallskip

\noindent\textbf{Sub-case 2b.} If the inward normal to $\gamma_1$ at $p$ is a repelling direction for $\iota_2\circ\iota_1$, then so is the outward normal to $\gamma_1$ at $p$ for the inverse germ $\iota_1\circ\iota_2$. In this case, $\sigma_a^{\circ 2}$ has four attracting directions which are pairwise exchanged by $\sigma_a$ (see Figure~\ref{triple_tangency_parabolic_fig} (right)). Once again, a Fatou-type argument shows that there are two period two cycles of parabolic basins of $\sigma_a$ (at $p$) and each such cycle contains an infinite critical orbit of $\sigma_a$. But this contradicts the fact that $\sigma_a$ has at most one infinite critical orbit. Hence, this sub-case cannot occur. 
\smallskip

\noindent\textbf{Case 3 ($k\geq 5$).} The same arguments as in the previous case show that there must be at least two period two cycles of parabolic basins of $\sigma_a$ at $p$. But this would require at least two infinite critical orbits of $\sigma_a$, implying that this case is impossible.
\end{proof}

\begin{definition}
We call a double point of $\partial\Omega_a$ \emph{regular} (respectively, \emph{special}) if the two non-singular branches of $\partial\Omega_a$ at $p$ have contact of order one (respectively, three).
\end{definition}

For $a\in\Gamma^{\mathrm{dp}}$, the desingularized droplet $T^0(\sigma_a)$ is defined as the set obtained by removing the cusp and the double points from $\widehat{\C}\setminus\Omega_a$. Note that in this case, $T^0(\sigma_a)$ contains a unique unbounded component, denoted by $T_u^0(\sigma_a)$, and finitely many bounded components, whose union is denoted by $T^0_b(\sigma_a)$.

\begin{lemma}\label{no_bad_dp_lem}
Let $a\in\Gamma^{\mathrm{dp}}$, and $X$ be a component of $T^0_b(\sigma_a)$. Then, the set of singular points on $\partial X$ is either exactly two double points of $\partial\Omega_a$, or exactly one double point and the unique cusp of $\partial\Omega_a$.
\end{lemma}
\begin{figure}
\captionsetup{width=0.98\linewidth}
\begin{tikzpicture}
\node[anchor=south west,inner sep=0] at (2,0) {\includegraphics[width=0.3\textwidth]{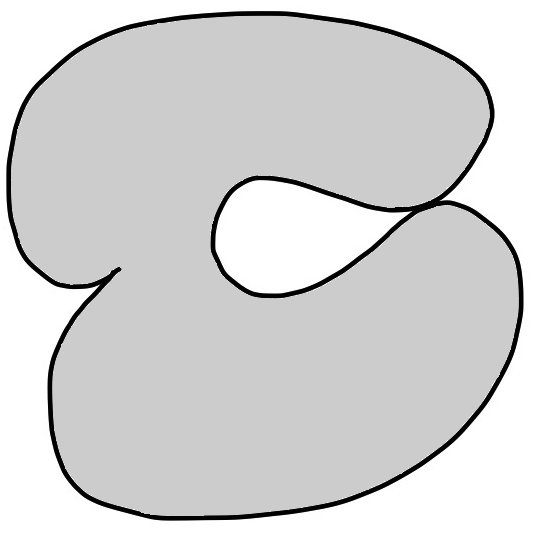}};
 \node[anchor=south west,inner sep=0] at (7,0) {\includegraphics[width=0.32\textwidth]{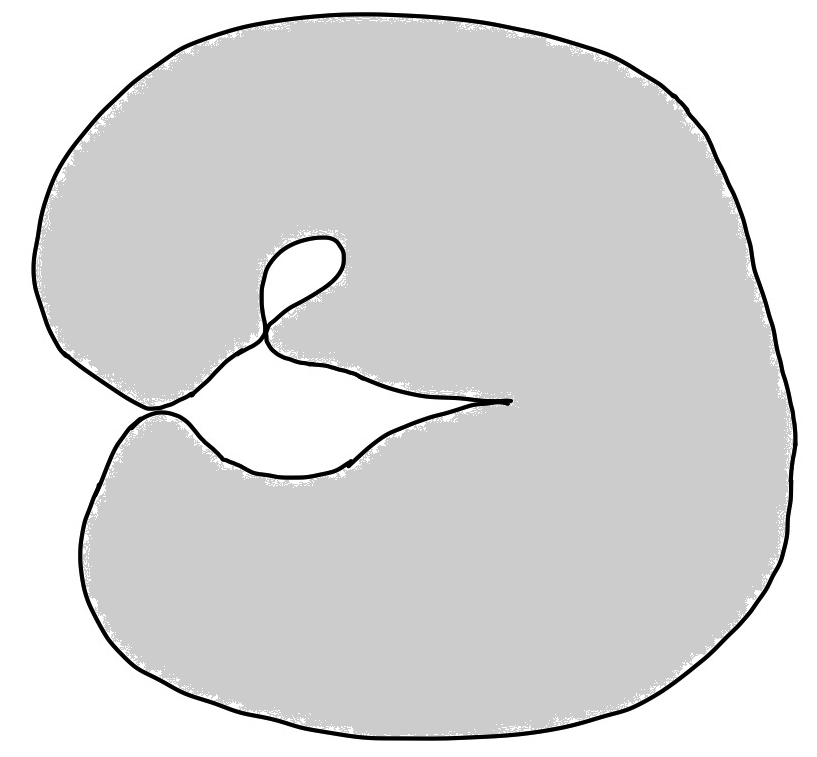}};
\node at (3.5,0.8) {$\Omega_a$};
\node at (4,2.12) {$X$};
\node at (5.04,2.1) {$p$};
\node at (9.6,0.8) {$\Omega_a$};
\node at (8.4,1.72) {$X$};
\node at (7.8,1.4) {$p$};
\node at (9.2,3) {$X'$};
\draw [->,line width=0.5pt] (9,2.8) to (8.48,2.4);
\end{tikzpicture}
\caption{Pictured are two a priori possible configurations of the desingularized droplet that are disallowed by Lemma~\ref{no_bad_dp_lem}.}
\label{bad_double_point_fig}
\end{figure}
\begin{proof}
By our hypothesis, $\partial X$ contains at least one double point $p$ of $\partial\Omega_a$.

We will first argue that $p$ cannot be the only singular point of $\partial\Omega_a$ on $\partial X$. By way of contradiction, let us assume that this is the case (see Figure~\ref{bad_double_point_fig} (left)). Then, $\partial X\setminus\{p\}$ is a non-singular real-analytic arc, and hence there exists a component $Y$ of $\sigma_a^{-1}(X)$ such that $\partial X\subsetneq\partial Y$. Since $a\in\Gamma^{\mathrm{dp}}$, the free critical value of $\sigma_a$ lies in $\overline{\Omega_a}$, which is disjoint from $\Int{X}$. Hence, $\sigma_a:\Int{Y}\to \Int{X}$ is a covering map. As $\Int{X}$ is a topological disk, $\sigma_a:\Int{Y}\to \Int{X}$ must be a homeomorphism. But this is impossible because each point on $\partial X\setminus\{p\}$ has at least two preimages under $\sigma_a$ on the boundary of $Y$.

Therefore, $\partial X$ must contain at least two double points of $\partial\Omega_a$ or a double point and the unique cusp of $\partial\Omega_a$. Also note that since $\partial\Omega_a$ is a real-algebraic curve, it has at most finitely many double points. Hence, if $\partial X$ contains an additional double point of $\partial\Omega_a$, then there must exist some component $X'\neq X$ of $T^0_b(\sigma_a)$ such that the only singularity on $\partial X'$ is a double point of $\partial\Omega_a$ (see Figure~\ref{bad_double_point_fig} (right)). But this contradicts the conclusion of the previous paragraph. This completes the~proof.
\end{proof}

We now proceed to study the connection between double points on $\partial\Omega_a$ and the global dynamics of the Schwarz reflection map $\sigma_a$.

\begin{figure}
\captionsetup{width=0.98\linewidth}
\begin{tikzpicture}
\node[anchor=south west,inner sep=0] at (0,0) {\includegraphics[width=0.42\textwidth]{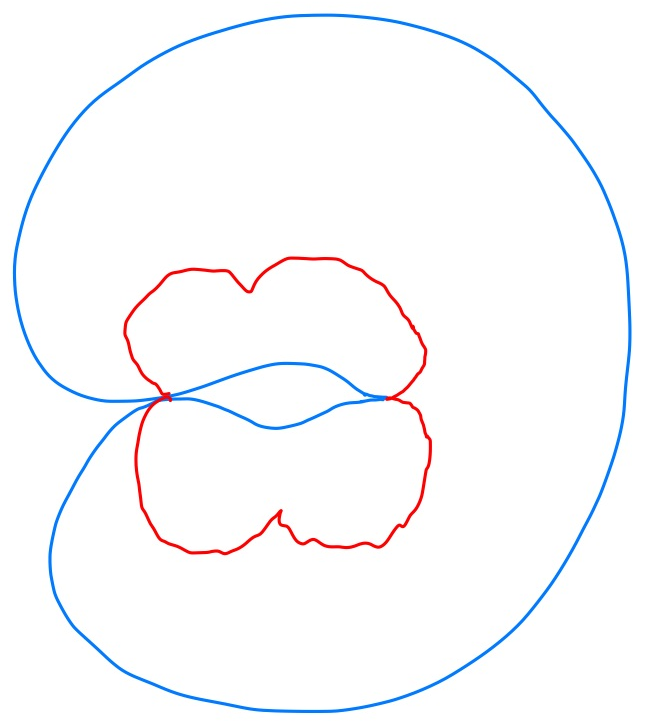}};
 \node[anchor=south west,inner sep=0] at (6.4,0) {\includegraphics[width=0.36\textwidth]{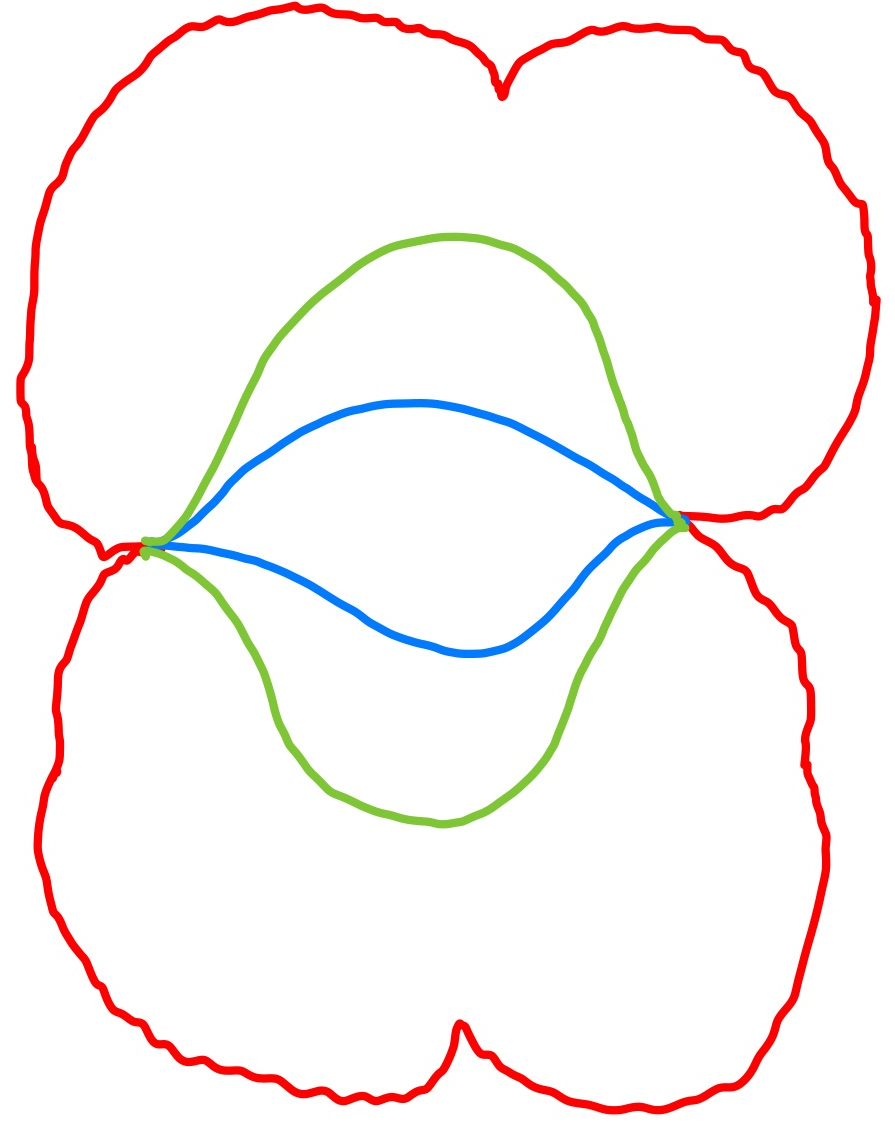}};
\node at (0.8,2.45) {$p$};
\node at (3.66,2.7) {$y_c$};
\node at (10.5,2.9) {$y_c$};
\node at (6.66,2.8) {$p$};
\node at (2.32,2.72) {\begin{large}$X$\end{large}};
\node at (8.48,3) {\begin{large}$X$\end{large}};
\node at (3.2,5) {\begin{Large}$\Omega_a$\end{Large}};
\node at (2.84,1.84) {\begin{large}$\mathcal{U}$\end{large}};
\node at (9.6,1.12) {\begin{Large}$\mathcal{U}$\end{Large}};
\node at (8.56,2.1) {\begin{footnotesize}$\sigma_a^{-1}(X)$\end{footnotesize}};
\node at (8.64,3.96) {\begin{footnotesize}$\sigma_a^{-1}(X)$\end{footnotesize}};
\end{tikzpicture}
\caption{The component $X$ of $T^0_a(\sigma_a)$ (introduced in Proposition~\ref{generic_bdd_droplet_dyn_prop}) and its $\sigma_a-$preimages in the tiling component $\mathcal{U}$ are displayed.}
\label{double_point_dyn_fig}
\end{figure}

\begin{proposition}\label{generic_bdd_droplet_dyn_prop}
Let $a\in\Gamma^{\mathrm{dp}}$. Then $T_b^0(\sigma_a)$ is connected and some forward iterate of the critical value $y_f$ of $\sigma_a$ lands in $T_b^0(\sigma_a)$.
\end{proposition}
\begin{proof}
By Proposition~\ref{no_bad_dp_lem} and the fact that $\partial\Omega_a$ has at most finitely many singular points, there must exist a component $X$ of $T^0_b(\sigma_a)$ which has the cusp $y_c$ and a double point $p$ (of $\partial\Omega_a$) on its boundary, and such that $\partial X\setminus\{p, y_c\}$ is the union of a pair of non-singular real-analytic arcs. In particular, $\partial X\setminus\{p, y_c\}$ has a small neighborhood contained in the tiling set of $\sigma_a$.
We denote the component of $T^\infty(\sigma_a)$ containing $X$ by $\mathcal{U}$ (see Figure~\ref{double_point_dyn_fig}).

By way of contradiction, let us suppose that the orbit $\{\sigma_a^{\circ n}(y_f)\}_{n\geq 0}$ does not intersect $X$. Under this assumption, the Riemann-Hurwitz formula implies that $\sigma_a^{-1}(\Int{X})\cap\mathcal{U}$ is the union of two simply connected domains each of which maps homeomorphically onto $\Int{X}$. Moreover, $\partial X\subsetneq\partial(\sigma_a^{-1}(X)\cap\mathcal{U})$, and $\Int{(X\cup(\sigma_a^{-1}(X)\cap\mathcal{U}))}$ is a simply connected domain (see Figure~\ref{double_point_dyn_fig}). Since the free critical orbit never meets $X$, one can now apply the above argument inductively to conclude that for each $n\geq 1$,
\begin{itemize}[leftmargin=6mm]
\item $\sigma_a^{-n}(\Int{X})\cap\mathcal{U}$ is the union of two simply connected domains each of which maps homeomorphically onto $\Int{X}$ under $\sigma_a^{\circ n}$, and

\item $\displaystyle\Int{\left(\bigcup_{i=0}^n(\sigma_a^{-i}(X)\cap\mathcal{U})\right)}$ is a simply connected domain.
\end{itemize}
Since $\displaystyle\mathcal{U}=\bigcup_{n=1}^\infty\left(\Int{\left(\bigcup_{i=0}^n(\sigma_a^{-i}(X)\cap\mathcal{U})\right)}\right)$ is an increasing union of simply connected domains, we conclude that $\mathcal{U}$ is also a simply connected domain. We denote the two connected components of $\mathcal{U}\setminus\left(X\cup\sigma_a^{-1}(X)\right)$ by $\mathcal{V}_1$ and $\mathcal{V}_2$, and the two connected components of $\mathcal{U}\setminus X$ by $\mathcal{W}_1$ and $\mathcal{W}_2$ such that $\mathcal{W}_j\supset\mathcal{V}_j$, $j\in\{1,2\}$. Note that the above arguments also demonstrate simple connectivity of the domains $\mathcal{V}_j, \mathcal{W}_j$, $j\in\{1,2\}$.
Moreover, $\sigma_a^{\circ 2}:\mathcal{V}_1\rightarrow \Int{\left(X\cup\mathcal{W}_1\right)}$ is a conformal isomorphism. We denote the corresponding inverse branch by $g:\Int{\left(X\cup\mathcal{W}_1\right)}\rightarrow\mathcal{V}_1\subsetneq\Int{\left(X\cup\mathcal{W}_1\right)}$, and observe that $g$ is a contraction with respect to the hyperbolic metric of $\Int{\left(X\cup\mathcal{W}_1\right)}$.

It follows by the previous paragraph and the local dynamics of $\sigma_a^{\circ 2}$ near $p$ (see Lemma~\ref{dp_dyn_classification_lem}) that there are points in $\mathcal{U}$ close to $p$ whose $g-$orbits converge to $p$ (asymptotically to a repelling direction of $\sigma_a^{\circ 2}$ at $p$). By \cite[Lemma~5.5]{Milnor06}, all $g-$orbits must converge to $p$.

By Lemma~\ref{cusp_anal_dyn_lem} and Corollary~\ref{cusp_dyn_cor}, the cusp point $y_c$ also has a repelling direction in $\Omega_a$. Hence, there are points in $\Omega_a$ near the cusp that eventually land in $X$ under iteration of $\sigma_a$. Consequently, there are points in $\mathcal{U}$ close to $y_c$ whose $g-$orbits converge to $y_c$. By \cite[Lemma~5.5]{Milnor06}, all $g-$orbits must converge to $y_c$. This contradicts the conclusion of the previous paragraph, and proves that the orbit $\{\sigma_a^{\circ n}(y_f)\}_{n\geq 0}$ must intersect $X$.

If $T_b^0(\sigma_a)$ had a component $X'$ other than $X$, then one can repeat the above argument to conclude that the orbit $\{\sigma_a^{\circ n}(y_f)\}_{n\geq 0}$ must also intersect $X'$, which is impossible. Thus, we conclude that $T_b^0(\sigma_a)=X$. \end{proof}

\begin{corollary}\label{no_special_dp_cor}
Let $a\in\Gamma^{\mathrm{dp}}$. Then the unique double point on $\partial\Omega_a$ is a regular double point.
\end{corollary}
\begin{proof}
By Lemma~\ref{dp_dyn_classification_lem} and Proposition~\ref{generic_bdd_droplet_dyn_prop}, the existence of a special double point on $\partial\Omega_a$ would force the free critical orbit of $\sigma_a$ to non-trivially converge to the special double point as well as escape to $T_b^0(\sigma_a)$, which is absurd.
\end{proof}

\begin{corollary}\label{hoc_dp_disjoint_cor}
$\Gamma^{\mathrm{hoc}}\cap\mathrm{\Gamma^\mathrm{dp}}=\emptyset$.
\end{corollary}
\begin{proof}
By way of contradiction, assume that $a\in\Gamma^{\mathrm{hoc}}\cap\mathrm{\Gamma^\mathrm{dp}}$. By Proposition~\ref{generic_bdd_droplet_dyn_prop}, the free critical orbit of $\sigma_a$ escapes to $T_b^0(\sigma_a)$. But this is impossible since the free critical orbit of $\sigma_a$ must also converge non-trivially to the cusp point $y_c$ by Lemma~\ref{cusp_anal_dyn_lem}.
\end{proof}

\begin{corollary}\label{free_crit_dp_cor}
Let $a\in\Gamma^{\mathrm{dp}}$. Then the forward orbit of the free critical value of $\sigma_a$ is disjoint from $\overline{T_u^0(\sigma_a)}\cup\{y_c\}$.
\end{corollary}
\begin{proof}
This follows from Propositions~\ref{generic_bdd_droplet_dyn_prop}.
\end{proof}

\subsection{Description of the boundary and interior of the parameter space $S_{\mathcal{T}}$}

\begin{lemma}\label{bdry_part_lem}
$\Gamma^{\mathrm{hoc}}\bigcup\overline{\Gamma^{\mathrm{dp}}}\bigcup\Gamma^{\perp}\subset\partial S_{\mathcal{T}}$.
\end{lemma}
\begin{proof}
Let $a\in\Gamma^{\mathrm{hoc}}$. By definition, $a\in S_{\mathcal{T}}$. By Lemma~\ref{cusp_anal_dyn_lem}, the circle $\partial\Delta_a$ is an osculating circle to the curve germ $\alpha_\theta$ at $\pmb{v_b}$ that maps under $\pmb{f}$ to a straight line segment $\{y_c-re^{2i\theta}:r\in[0,\epsilon]\}$, where $\theta= \arg{(a-\pmb{v_b})}$. Let $\overrightarrow{\ell_a}$ be the infinite ray from $\pmb{v_b}$ to $\infty$ passing through $a$, and $a'$ be a parameter obtained by pushing $a$ slightly along $\overrightarrow{\ell_a}$ away from $\pmb{v_b}$. Then, $\Delta_{a'}$ contains points of $\alpha_\theta$ that are identified under $\pmb{f}$, and hence
$\pmb{f}$ is not injective on $\Delta_{a'}$. Therefore, $a\in\partial S_{\mathcal{T}}$; i.e., $\Gamma^{\mathrm{hoc}}\subset\partial S_{\mathcal{T}}$.

Let $a\in\Gamma^{\perp}$. By Lemma~\ref{pb_hoc_disjoint_lem} and the definition of $\Gamma^{\perp}$, $\pmb{f}$ is univalent on $\overline{\Delta_a}$ and $y_c$ is a $(3,2)$ cusp. Thus, by Lemma~\ref{cusp_anal_dyn_lem}, the circle $\partial\Delta_a$ is not an osculating circle to the curve germ $\alpha_\theta$ at $\pmb{v_b}$. Hence, pushing $a$ along $\overrightarrow{\ell_a}$ slightly away from $\pmb{v_b}$ produces parameters $a'$ such that $\pmb{f}$ is injective on $\overline{\Delta_{a'}}$ and $\pmb{v_w}\in\Delta_{a'}$. Hence, $a\in\partial S_{\mathcal{T}}$. This proves that $\Gamma^{\perp}\subset\partial S_{\mathcal{T}}$.

Let $a\in\Gamma^{\mathrm{dp}}$. Once again, pushing $a$ along $\overrightarrow{\ell_a}$ slightly away from $\pmb{v_b}$ produces parameters $a'$ such that $\pmb{f}$ is not injective on $\Delta_{a'}$ (since $\Delta_{a'}\cup\{\pmb{v_b}\}\supset\overline{\Delta_a}$). On the other hand, pushing $a$ slightly along $\overrightarrow{\ell_a}$ towards $\pmb{v_b}$ produces parameters $a'$ such that $\pmb{f}$ is injective on $\overline{\Delta_{a'}}$ (since $\overline{\Delta_{a'}}\subset\Delta_a\cup\{\pmb{v_b}\}$) and $\pmb{v_w}\in\Delta_{a'}$ (see Figure~\ref{projection_fig}). Hence, $a\in\partial S_{\mathcal{T}}$. This proves that $\Gamma^{\mathrm{dp}}\subset\partial S_{\mathcal{T}}$.
The result now follows by taking closures.
\end{proof}

\begin{theorem}\label{bdry_thm}
We have 
\begin{equation*}
\partial S_{\mathcal{T}}= \Gamma^{\mathrm{hoc}}\bigcup\overline{\Gamma^{\mathrm{dp}}}\bigcup\Gamma^{\perp},
\end{equation*}
and
\begin{equation*}
\begin{split}
\Int{S_{\mathcal{T}}}
&= \{a\in\C: \pmb{v_w}\in\Delta_a,\ f\vert_{\overline{\Delta_a}}\ \textrm{is\ univalent,\ and}\ y_c\ \textrm{is\ a}\ (3,2)\ \textrm{cusp}\}.
\end{split}
\end{equation*}
\end{theorem}
\begin{proof}
If $\pmb{f}\vert_{\overline{\Delta_a}}$ is univalent and the cusp $y_c\in\partial\Omega_a$ is of type $(3,2)$, then $a$ has a neighborhood such that for all parameters $a'$ in this neighborhood, $\pmb{f}\vert_{\overline{\Delta_{a'}}}$ is univalent. Since parameters in $S_{\mathcal{T}}$ for which $y_c$ is a higher order cusp belong to the boundary of $S_{\mathcal{T}}$ (by Lemma~\ref{bdry_part_lem}), the description of the interior of $S_{\mathcal{T}}$ given in the statement of the theorem follows.

Now let $a\in\partial S_{\mathcal{T}}$. We consider two cases.
\vspace{1mm}

\noindent\textbf{Case 1: $a\in S_{\mathcal{T}}$.} By the description of $\Int{S_{\mathcal{T}}}$, we have that $y_c$ is a cusp of type $(\nu,2)$ of $\partial\Omega_a$, with $\nu\geq 5$; i.e., $a\in\Gamma^{\mathrm{hoc}}$.
\vspace{1mm}

\noindent\textbf{Case 2: $a\notin S_{\mathcal{T}}$.} Since $\pmb{f}\vert_{\Delta_a}$ is the local uniform limit of a sequence of injective holomorphic maps and $\pmb{f}$ is non-constant, we have that $\pmb{f}\vert_{\Delta_a}$ is injective. The assumption that $a\notin S_{\mathcal{T}}$ now implies that either $\pmb{f}$ is not injective on $\partial\Delta_a$ or $\pmb{v_w}\in\partial\Delta_a$ (or both). In the former case, there is a double point on $\partial\Omega_a$. Moreover, as $a\in\overline{S_{\mathcal{T}}}$, we have that $\pmb{v_w}\in\overline{\Omega_a}$. Hence, $a\in\overline{\Gamma^{\mathrm{dp}}}$. In the latter case, either $a\in\overline{\Gamma^{\mathrm{dp}}}$ or $a\in \Gamma^{\perp}$ depending on whether there is a double point on $\partial\Omega_a$ or not.

Combining the two cases, we conclude that 
$$
\partial S_{\mathcal{T}}\subset \Gamma^{\mathrm{hoc}}\cup\overline{\Gamma^{\mathrm{dp}}}\cup\Gamma^{\perp}.
$$
The description of the boundary of $S_{\mathcal{T}}$ follows from the above containment and Lemma~\ref{bdry_part_lem}.
\end{proof}

\subsection{Connectedness of $S_{\mathcal{T}}$} 

Our next goal is to prove that the parameter space $S_{\mathcal{T}}$ is connected. This will be done through a series of lemmas.

\begin{lemma}\label{para_comp_jordan_lem}
Each connected component of $\Int{S_{\mathcal{T}}}$ is a Jordan domain. Moreover, $\overline{S_{\mathcal{T}}}=\overline{\Int{S_{\mathcal{T}}}}$.
\end{lemma}
\begin{figure}
\captionsetup{width=0.98\linewidth}
\begin{tikzpicture}
\node[anchor=south west,inner sep=0] at (0,0) {\includegraphics[width=0.6\textwidth]{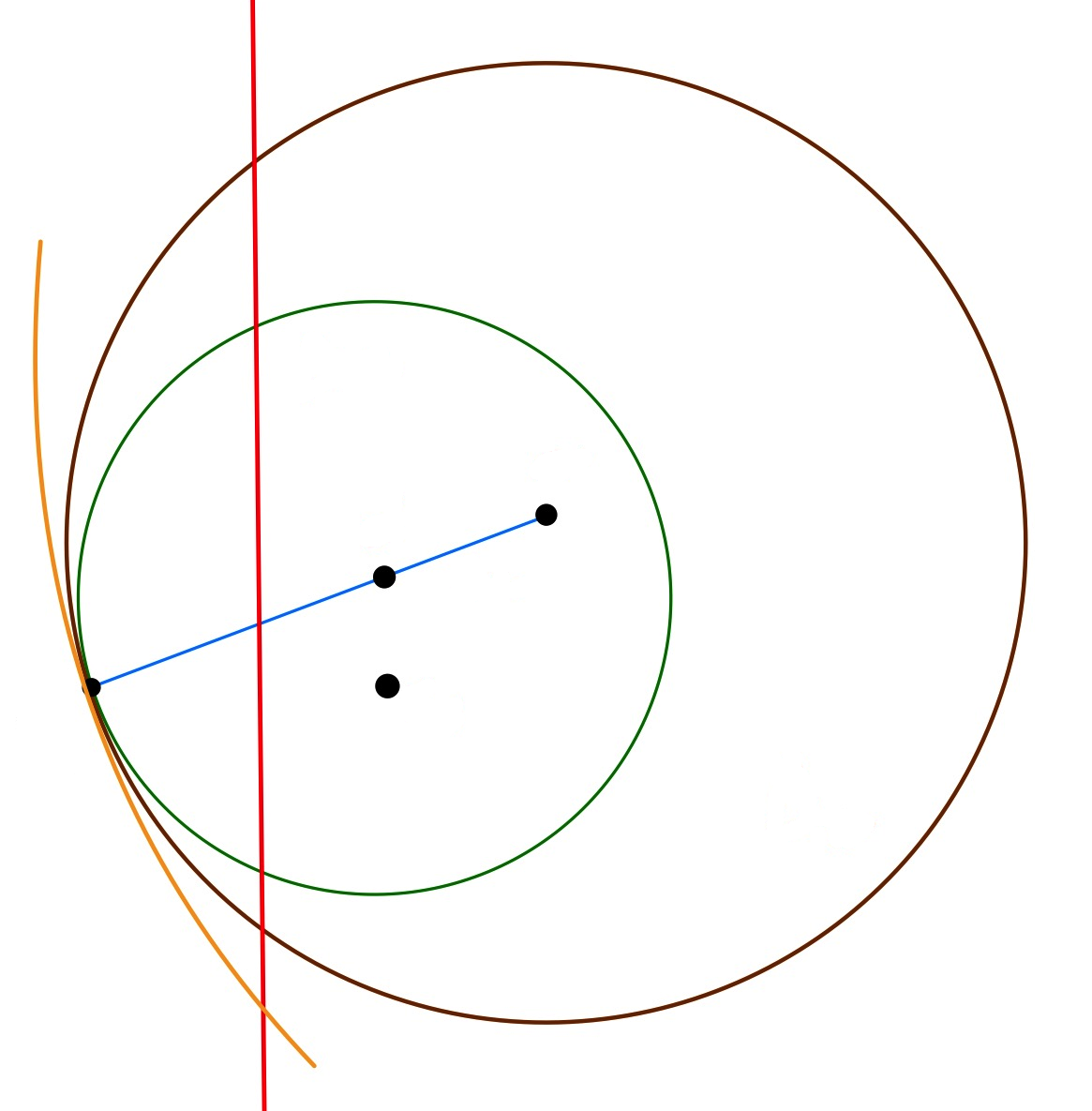}};
\node at (0.4,2.9) {\begin{large}$\pmb{v_b}$\end{large}};
\node at (2.8,2.7) {\begin{large}$\pmb{v_w}$\end{large}};
\node at (0,4.8) {\begin{large}$\alpha_\theta$\end{large}};
\node at (4.4,6.6) {\begin{large}$\Delta_a$\end{large}};
\node at (2.7,5.25) {\begin{large}$\Delta_{a'}$\end{large}};
\node at (2.7,4) {\begin{large}$a'$\end{large}};
\node at (3.8,4.4) {\begin{large}$a$\end{large}};
\node at (3.36,3.72) {\begin{large}$\ell_a$\end{large}};
\end{tikzpicture}
\caption{The red straight line is the perpendicular bisector of the line segment joining $\pmb{v_b}$ and $\pmb{v_w}$, the half-plane to the right of the red straight line is $\mathbf{P}$, and $\alpha_\theta$ is the curve germ at $\pmb{v_b}$ that maps under $\pmb{f}$ to a straight line segment $\{y_c-re^{2i\theta}: r\in[0,\epsilon]\}$, where $\theta=\arg{(a-\pmb{v_b})}$. There is a nesting structure of the disks $\Delta_a$ when $a$ lies on a fixed straight ray emanating from $\pmb{v_b}$.}
\label{projection_fig}
\end{figure}
\begin{proof}
Recall that by Corollary~\ref{non_emp_cor}, the interior of $S_{\mathcal{T}}$ is non-empty.

Given any $a\in \C$, let $\ell_a$ be the straight line segment connecting $\pmb{v_b}$ to $a$. Now let $a' \in \ell_a$. By definition, $|a-\pmb{v_b}| > |a'-\pmb{v_b}|$ and hence $\overline{\Delta_{a'}}\subset \overline{\Delta_a}$. Thus, if $\pmb{f}$ is univalent on $\overline{\Delta_a}$ then it is also univalent on $\overline{\Delta_{a'}}$. This means that if $a\in S_{\mathcal{T}}$, then the intersection of $\ell_a$ with the half plane $\mathbf{P}:=\{z\in\C: \vert z-\pmb{v_w}\vert<\vert z-\pmb{v_b}\vert\}$ (which is one of the complementary components of the perpendicular bisector of the line segment joining $\pmb{v_w}$ and $\pmb{v_b}$) also lies in the parameter space (see Figure~\ref{projection_fig}). This implies that each connected component of $S_{\mathcal{T}}$ is simply connected.

Now suppose that $a\in\Gamma^{\mathrm{dp}}\cup\Gamma^{\mathrm{hoc}}$, and $a'\in \ell_a\cap\mathbf{P}$, $a'\neq a$. The above argument and the ones used in the proof of Lemma~\ref{bdry_part_lem} show that $\pmb{f}$ is injective on $\overline{\Delta_{a'}}$ and $\pmb{v_w}\in\Delta_{a'}$ (see Figure~\ref{projection_fig}). We also claim that $y_c$ is a $(3,2)$ cusp on $\partial\Omega_{a'}$. Indeed, if this were not true, then pushing $a'$ along $\ell_a$ in the direction of $a$ (i.e., away from $\pmb{v_b}$) would result in $\pmb{f}$ to be non-univalent on the corresponding disk contradicting the assumption that $f\vert_{\Delta_a}$ is univalent. Therefore, $\ell_a\cap\mathbf{P}\setminus\{a\}\subset \Int{S_{\mathcal{T}}}$ for all $a\in\Gamma^{\mathrm{dp}}\cup\Gamma^{\mathrm{hoc}}$ (by the description of $ \Int{S_{\mathcal{T}}}$ given in Theorem~\ref{bdry_thm}). It follows that each component of $\Int{S_{\mathcal{T}}}$ is a Jordan domain.

By Theorem~\ref{bdry_thm}, if $a\in S_{\mathcal{T}}\setminus\Int{S_{\mathcal{T}}}$, then $a$ must lie on $\Gamma^{\mathrm{hoc}}$. But the above argument demonstrates that in this case, there are points $a'\in\ell_a$ arbitrarily close to $a$ such that $a'\in\Int{S_{\mathcal{T}}}$. This proves that $S_{\mathcal{T}}\subset \overline{\Int{S_{\mathcal{T}}}}$, and hence, $\overline{S_{\mathcal{T}}}=\overline{\Int{S_{\mathcal{T}}}}$.
\end{proof}

\begin{lemma}\label{dp_pb_two_lem}
$\overline{\Gamma^{\perp}}\cap\overline{\Gamma^{\mathrm{dp}}}$ contains at most two points.
\end{lemma}
\begin{proof}
Let $a\in\overline{\Gamma^{\perp}}\cap\overline{\Gamma^{\mathrm{dp}}}$. Then the free critical value $y_f$ of $\sigma_a$ lies on $\partial\Omega_a$, and thus is fixed. Moreover, by Proposition~\ref{generic_bdd_droplet_dyn_prop}, $y_f$ lies on $\partial T^0_b(\sigma_a)$, but is not the unique double point or the cusp point $y_c$ of $\partial\Omega_a$. We also remark that since $y_f$ is not a double point of $\partial\Omega_a$, the only $\pmb{f}-$preimage of $y_f$ on $\Delta_a$ is $\pmb{v_w}$, which is not a critical point of $\pmb{f}$. Hence, $y_f$ is not a cusp of $\partial\Omega_a$; i.e., $\partial\Omega_a$ has a unique cusp at $y_c$.

Therefore, for each $a\in\overline{\Gamma^{\perp}}\cap\overline{\Gamma^{\mathrm{dp}}}$, the interior of $T_b^0(\sigma_a)$ is a Jordan domain whose boundary contains three distinct distinguished points; namely, the unique double point of $\partial\Omega_a$, the unique cusp $y_c$ of $\partial\Omega_a$, and the unique free critical value $y_f$ of $\sigma_a$. On the other hand, the interior of $T_u^0(\sigma_a)$ is also a Jordan domain which contains the fully branched critical value $\infty$ of $\sigma_a$ and has the unique cusp $y_c$ of $\partial\Omega_a$ on its boundary.

Let us now assume that there are two parameters $a_1, a_2\in\overline{\Gamma^{\perp}}\cap\overline{\Gamma^{\mathrm{dp}}}$ such that for both parameters, the unique double point of $\partial\Omega_{a_i}$, the cusp $y_c$, and the unique free critical value $y_f$ of $\sigma_{a_i}$ lie in the same cyclic order on $\partial T_b^0(\sigma_{a_i})$. We claim that $a_1=a_2$. Since these three distinguished points on the boundary of  $T^0_b(\sigma_a)$ can lie in exactly two different cyclic orders, the proof will be complete once the claim is established.

Let $\mathfrak{g}_b:\Int{T_b^0(\sigma_{a_1})}\rightarrow\Int{T_b^0(\sigma_{a_2})}$ be the conformal isomorphism whose homeomorphic boundary extension (also denoted by $\mathfrak{g}_b$) carries the cusp, double point and the free critical value associated with $a_1$ to those associated with $a_2$. Furthermore, let  $\mathfrak{g}_u:\Int{T_u^0(\sigma_{a_1})}\rightarrow\Int{T_u^0(\sigma_{a_2})}$ be the conformal isomorphism which sends the fully branched critical value $\infty$ of $\sigma_{a_1}$ to the fully branched critical value $\infty$ of $\sigma_{a_2}$, and whose homeomorphic boundary extension (also denoted by $\mathfrak{g}_b$) takes the unique double point of $\partial\Omega_{a_1}$ to the unique double point of $\partial\Omega_{a_2}$. Since $y_f$ is fixed under $\sigma_{a_i}$, $i\in\{1,2\}$, Lemmas~\ref{cusp_anal_dyn_lem} and~\ref{dp_dyn_classification_lem} imply that the double points on $\partial\Omega_{a_1}$, $\partial\Omega_{a_2}$ are regular and the cusps on them are of type $(3,2)$. Hence, one can apply the arguments of \cite[Lemma~8.11, Appendix~B]{LLMM2} or \cite[Lemmas~5.3,~5.4]{LMM1} to conclude that there exists a global $K-$quasiconformal map $\mathfrak{G}_0$ of the Riemann sphere that continuously matches with $\mathfrak{g}_u$ and $\mathfrak{g}_b$ on their domains of definition.

Since the free critical orbits of $\sigma_{a_1}, \sigma_{a_2}$ lie in the respective tiling sets, one can apply classical arguments of Fatou to see that $\Int{K(\sigma_{a_1})}=\Int{K(\sigma_{a_2})}=\emptyset$ (see \cite[Propositions~5.30,~5.32]{LLMM1}). Hence, $\partial T^\infty(\sigma_{a_i})=K(\sigma_{a_i})$; i.e., $\overline{T^\infty(\sigma_{a_i})}=\widehat{\C}$. Moreover, the same fact also allows one to show that these non-escaping sets have zero area (cf. \cite[Corollary~6.3]{LLMM1}).

As the Schwarz reflection maps act as the identity on the boundaries of the desingularized droplets, a standard pullback argument as in \cite[Proposition~8.13]{LLMM2} or \cite[Theorem~5.1]{LMM1} can be employed to construct a sequence of $K-$quasiconformal maps $\{\mathfrak{G}_n\}$ such that 
\begin{enumerate}[leftmargin=8mm]
\item\label{ev_cond} $\sigma_{a_2}\circ\mathfrak{G}_n= \mathfrak{G}_{n-1}\circ \sigma_{a_1}$ on $\widehat{\C}\setminus\Int{T^0(\sigma_{a_1})}$,

\item\label{conf_cond} $\mathfrak{G}_n$ is conformal on $\displaystyle\bigcup_{i=0}^{n}\sigma_{a_1}^{-i}(T^0(\sigma_{a_1}))$, and

\item\label{matching_cond} $\mathfrak{G}_n=\mathfrak{G}_{n-1}$ on $\displaystyle\bigcup_{i=0}^{n-1}\sigma_{a_1}^{-i}(T^0(\sigma_{a_1}))$.
\end{enumerate}
By compactness of the family of $K-$quasiconformal homeomorphisms and Conditions~\eqref{ev_cond},~\eqref{matching_cond}, there exists a quasiconformal homeomorphism $\mathfrak{G}_\infty$ of $\widehat{\C}$ that conjugates $\sigma_{a_1}$ to $\sigma_{a_2}$ on the tiling set. By continuity and density of the tiling sets of $\sigma_{a_1}, \sigma_{a_2}$ in $\widehat{\C}$, the conjugacy relation holds on the entire domain of definition of $\sigma_{a_1}$. Also, Condition~\ref{conf_cond} implies that $\mathfrak{G}_\infty$ is conformal on the tiling set of $\sigma_{a_1}$. Since the non-escaping set of $\sigma_{a_1}$ has zero area, it follows by Weyl's lemma that $\mathfrak{G}_\infty$ is a M{\"o}bius map of $\widehat{\C}$. Finally, since the conjugacy $\mathfrak{G}_\infty$ fixes $\infty$, $y_f$ and $y_c$, it must be the identity map. Hence, $a_1=a_2$. 
\end{proof}

\begin{lemma}\label{hoc_arc_lem}
$\Gamma^{\mathrm{hoc}}\ \left(\subset\partial S_{\mathcal{T}}\right)$ is a closed real-analytic arc. 
\end{lemma}

\begin{proof}
For clarity of exposition, we split the proof into several steps.
\smallskip

\noindent\textbf{Step I: $\Gamma^{\mathrm{hoc}}\neq\emptyset$.} One can apply the arguments of \cite[Theorem A]{PartI} on the dynamically Shabat anti-polynomial $p$ constructed in Lemma~\ref{dyn_shabat_existence_lem} (as usual, if $\mathcal{T}$ is a star-tree, we perform the surgery on $\overline{z}^d$) to 
\noindent\begin{enumerate}[leftmargin=8mm]
\item replace the dynamics on the basin of infinity with $\mathcal{R}_d\vert_{\Int{\mathcal{Q}}}$, and 
\item the dynamics on the bounded fixed critical Fatou component $U$ with the unicritical parabolic anti-Blaschke product $B_k\vert_{\D}$ (see \S~\ref{para_anti_rat_subsec}),
\end{enumerate}
such that the unique parabolic fixed point of $\mathcal{R}_d$ as well as the unique parabolic fixed point of $B_k$ correspond to the repelling fixed point $\pmb{v_b}$ of $p$ (this is possible by Remark~\ref{zero_ray_rem}). Here, $k$ is the degree of $p$ on $U$. This produces a parameter $a_0\in S_{\mathcal{T}}$ such that $\sigma_{a_0}^{\circ 2}$ has a unique attracting direction in $K(\sigma_{a_0})$ at the cusp $y_c$ of $\partial\Omega_{a_0}$. Hence, by Lemma~\ref{cusp_anal_dyn_lem} and Corollary~\ref{cusp_dyn_cor}, $y_c$ is a $(5,2)$ cusp of $\partial\Omega_{a_0}$; and hence in particular, $a_0\in\Gamma^{\mathrm{hoc}}$. Note that $\mathcal{J}(p)$ is removable for $W^{1,1}$ functions (by \cite[Theorem~4]{JS00} and the fact that $\mathcal{B}_\infty(p)$ is a John domain). Thus, according to \cite[Theorem~2.7]{LMMN}, the limit set $\partial K(\sigma_{a_0})$ is conformally removable.
\smallskip

\noindent\textbf{Step II:  Constructing an arc $\Gamma'\subset\Gamma^{\mathrm{hoc}}$ consisting of parameters with $(5,2)-$cusps.}  
By \cite[Appendix A]{PartI} there is a forward invariant attracting petal of $\sigma_{a_0}$ at the cusp point $y_c$, and a Fatou coordinate (unique up to real translations) on such a petal that conjugates $\sigma_{a_0}$ to the glide reflection $\zeta\mapsto\overline{\zeta}+\frac12$ on a right half-plane. By construction of $\sigma_{a_0}$, such an attracting petal at $y_c$ contains the tail of the $\sigma_{a_0}^{\circ 2}-$orbit of the free critical value $y_f$ (of $\sigma_{a_0}$). We refer to the imaginary part of $\sigma_{a_0}^{\circ 2n}(y_f)$ (for $n$ large enough) in this coordinate as the critical \emph{{\'E}calle height} of $\sigma_{a_0}$ (cf. \cite[\S 2]{HS}, \cite[\S 3]{MNS}).  It is readily seen that the critical {\'E}calle height is a conformal conjugacy invariant of the map $\sigma_{a_0}$. The real-symmetry of $B_d$ tells us that the critical {\'E}calle height of the map $\sigma_{a_0}$ is $0$.

One can now apply a quasiconformal deformation argument as in the proof of \cite[Proposition~6.6 (part 2)]{LLMM3} (cf. \cite[Theorem~3.2]{MNS}) to 
obtain an open real-analytic arc $\Gamma'\subset S_{\mathcal{T}}$ containing $a_0$ such that for each $t\in\R$, there exists a unique parameter $a(t)\in\Gamma'$ with the following properties:
\begin{enumerate}[leftmargin=8mm]
\item the cusp $y_c$ has a unique attracting direction under $\sigma_{a(t)}$, and hence $y_c$ is a $(5,2)$ cusp of $\partial\Omega_{a(t)}$,

\item the forward orbit of the free critical value $y_f$ of $\sigma_{a(t)}$ converges to $y_c$, and

\item the critical {\'E}calle height of $\sigma_{a(t)}$ is $t$.
\end{enumerate}
Clearly, $\Gamma'$ is contained in $\Gamma^{\mathrm{hoc}}$. Since conformal removability is preserved under quasiconformal maps, the limit set of each map produced above is conformally removable. 
\smallskip

\noindent\textbf{Step III: Rigidity of parameters in $\Gamma'$.}
We claim that there are no other parameters in $\Gamma^{\mathrm{hoc}}$ such that $y_c$ is a $(5,2)-$cusp on the corresponding quadrature domain boundary. Indeed, let $a'\in\Gamma^{\mathrm{hoc}}$ be such that $y_c$ is a $(5,2)-$cusp of $\partial\Omega_{a'}$. Then the unique free critical orbit of $\sigma_{a'}$ converges to $y_c$ and hence the map has a finite critical {\'E}calle height $t_0$. We claim that $a'=a(t_0)$, where $a(t_0)$ is the unique parameter on $\Gamma'$ with critical {\'E}calle height $t_0$.
Since both $\sigma_{a'}, \sigma_{a(t_0)}$ have a unique free critical orbit, one can adapt the arguments of  \cite[Propositions~5.30,~5.32]{LLMM1} to show that both $\Int{K(\sigma_{a'})}$, $\Int{K(\sigma_{a(t_0)})}$ equal the basin of attraction of the cusp $y_c$. Moreover, the proofs of \cite[Lemma~8.5,~Proposition~8.7]{LLMM2} (or more generally, that of \cite[Proposition~6.19]{kiwi1}) apply to the maps $\sigma_{a'},\sigma_{a(t_0)}$ and imply that both $\partial T^\infty(\sigma_{a'}), \partial T^\infty(\sigma_{a(t_0)})$ are quotients of $\partial \mathcal{Q}\cong\mathbb{S}^1$ under the closed $\mathcal{R}_d-$invariant equivalence relation generated by the angles of the dynamical rays landing at the preimages of $y_c$. Note that the angles of these rays only depend on the plane tree $\mathcal{T}$, and hence the corresponding equivalence relation is the same for the two maps $\sigma_{a'},\sigma_{a(t_0)}$. It now follows that the non-escaping set dynamics of $\sigma_{a'}$ and $\sigma_{a(t_0)}$ are topologically conjugate where the conjugacy is conformal on the interior (conformality is a consequence of the fact that the maps have the same critical {\'E}calle height).
Moreover by Lemma~\ref{schwarz_group}, their tiling set dynamics are also conformally conjugate. These two conjugacies match up to yield a global orientation-preserving topological conjugacy between the two Schwarz reflection maps such that the conjugacy is conformal off the limit set. Conformal removability of the limit set of $\sigma_{a(t_0)}$ now implies that $\sigma_{a'}$ and $\sigma_{a(t_0)}$ are M{\"o}bius conjugate and hence equal. (Alternatively, one can employ the pullback arguments of \cite[Proposition~9.4]{LLMM2} to prove rigidity of parameters with $(5,2)-$cusps.)
\smallskip

\noindent\textbf{Step IV: The two endpoints of $\Gamma'$ are the only parameters with $(7,2)-$cusps.}
Since $\Gamma^{\mathrm{hoc}}$ is contained in a real-algebraic curve (defined by the higher order cusp condition), each end of the arc $\Gamma'$ has a unique limit point in $\Gamma^{\mathrm{hoc}}$. 
We claim that the two ends of $\Gamma'$ have distinct endpoints. If this were not true, since $\Gamma^{\mathrm{hoc}}\cap\overline{\Gamma^{\perp}}=\emptyset$ (by Lemma~\ref{pb_hoc_disjoint_lem}), the closure of $\Gamma'$ (in $\C$) would be a topological circle contained in the open half plane $\mathbf{P}=\{z\in\C: \vert z-\pmb{v_w}\vert<\vert z-\pmb{v_b}\vert\}$. As $\Gamma^{\mathrm{hoc}}\subset\partial S_{\mathcal{T}}$, this topological circle must entirely lie on the boundary of $S_{\mathcal{T}}$. On the other hand, there must exist sub-arcs $I', I''$ of this topological circle such that the union of the line segments connecting points of $I'$ to $\pmb{v_b}$ contain $I''$ in its interior. The `projection argument' of Lemma~\ref{para_comp_jordan_lem} now implies that $I''$ is contained in the interior of $S_{\mathcal{T}}$, a contradiction.
\begin{figure}[ht!]
\captionsetup{width=0.98\linewidth}
\begin{tikzpicture}
\node[anchor=south west,inner sep=0] at (0.5,0) {\includegraphics[width=0.66\textwidth]{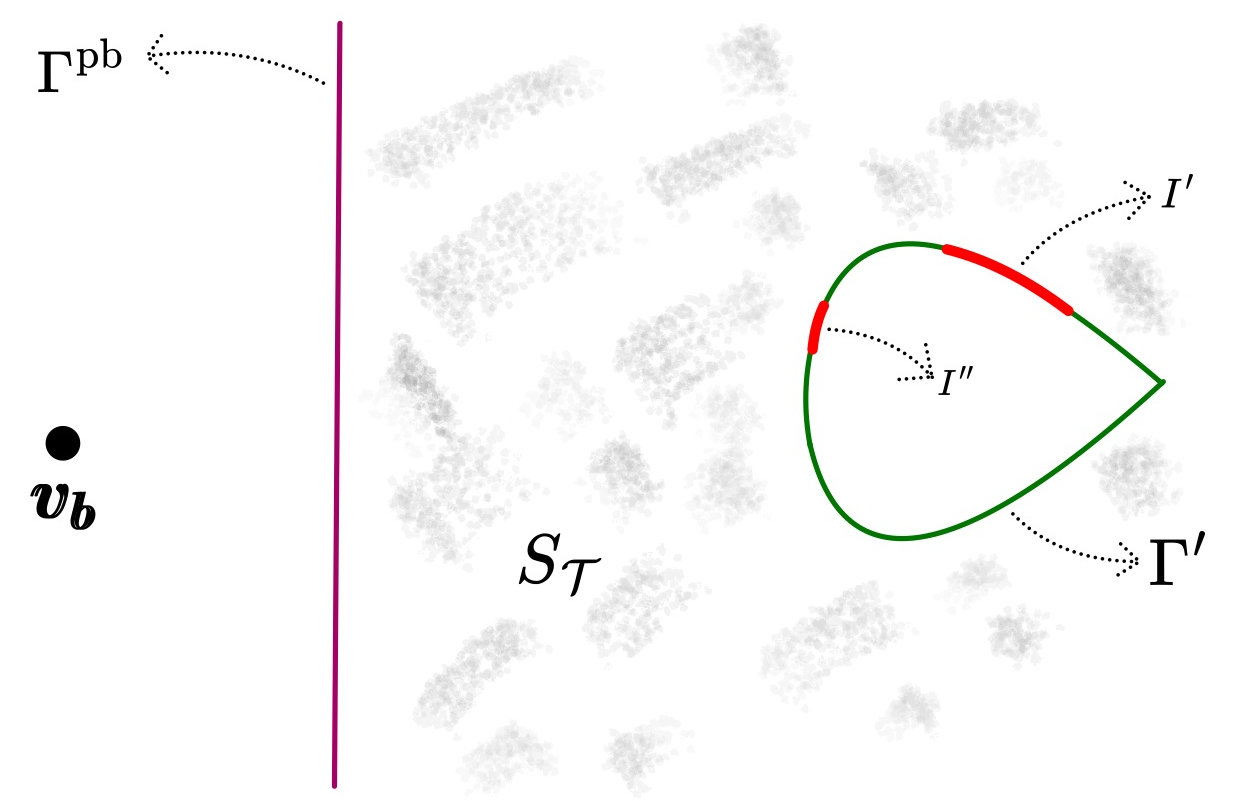}};
\end{tikzpicture}
\caption{}
\label{hoc_two_ends_fig}
\end{figure}

Let $a'$ be a limit point of $\Gamma'$. We will now argue that $a'\in S_{\mathcal{T}}$, and that $y_c$ is a $(7,2)-$cusp of $\partial\Omega_{a'}$. Clearly, $\pmb{f}$ is univalent on $\Delta_{a'}$, $\pmb{v_w}\in\overline{\Delta_{a'}}$, and $y_c$ is a cusp of type $(\nu,2)$ with $\nu>3$ on $\partial\Omega_{a'}$, where $\Omega_{a'}:=\pmb{f}(\Delta_{a'})$. Thus, by the arguments of Lemma~\ref{cusp_anal_dyn_lem}, the forward orbit of $y_f$ under the associated Schwarz reflection map $\sigma_{a'}$ converges non-trivially to the cusp $y_c$. It follows that $y_f=\pmb{f}(\pmb{v_w})\notin\partial\Omega_{a'}$; i.e., $\pmb{v_w}\in\Delta_{a'}$. If there were a double point on $\partial\Omega_{a'}$, then Proposition~\ref{generic_bdd_droplet_dyn_prop} would prevent the critical value $y_f$ of $\sigma_{a'}$ to converge non-trivially to $y_c$. Therefore, $\pmb{f}$ is injective on $\overline{\Delta_{a'}}$, and hence $a'\in S_{\mathcal{T}}$. Thanks to Corollary~\ref{cusp_dyn_cor}, it now suffices to show that $y_c$ is not a $(5,2)-$cusp on $\partial\Omega_{a'}$. We have already established that the parameters in $S_{\mathcal{T}}$ for which $y_c$ is a $(5,2)$ cusp on the corresponding quadrature domain boundaries comprise $\Gamma'$. Thus, we need to prove that $\Gamma'$ does not accumulate on itself. But this follows directly from the fact that $\ell_a\cap\mathbf{P}\setminus\{a\}\subset \Int{S_{\mathcal{T}}}$ for all $a\in\Gamma^{\mathrm{hoc}}$ (see the proof of Lemma~\ref{para_comp_jordan_lem}). Hence, $y_c$ is a $(7,2)-$cusp of $\partial\Omega_{a'}$.

As in the case of $(5,2)-$cusps, one can now apply the arguments of \cite[Proposition~8.15]{LLMM2} to prove a rigidity statement for parameters with $(7,2)-$cusps and conclude that there are no other parameters in $\Gamma^{\mathrm{hoc}}$ such that $y_c$ is a $(7,2)-$cusp of $\partial\Omega_a$.

By Corollary~\ref{cusp_dyn_cor}, these are all the parameters in $\Gamma^{\mathrm{hoc}}$. Therefore, $\Gamma^{\mathrm{hoc}}=\overline{\Gamma'}$ is a closed real-analytic arc.
\end{proof}

\begin{remark}
We do not know if $\Gamma^{\mathrm{hoc}}$ is necessarily a smooth arc.
\end{remark}

\begin{remark}
In Step I of the proof of Lemma~\ref{hoc_arc_lem}, the construction of the parameter $a_0\in S_\cT$ can be made more classical by using a two-step surgery. The first surgery, a la Ha{\"i}ssisnky-Lei \cite{HT04,Hai00}, manufactures a parabolic anti-rational map $R\in\mathfrak{L}_\cT$ that is David conjugate to $p$ on its Julia set and whose dynamics on the bounded invariant Fatou component is conformally conjugate to $B_k\vert_{\D}$ so that $\infty$ is a double parabolic fixed point (where $k$ is the degree of $p$ on its bounded invariant Fatou component). The second step involves the quasiconformal surgery of \cite[Theorem~5.1]{PartI} (cf. Theorem~\ref{qc_bijection_thm}), which produces a Schwarz reflection in $S_\cT$ that is hybrid conjugate to $R$ in neighborhoods of their non-escaping sets. We note that the boundary of the invariant Fatou component of the resulting Schwarz reflection $\sigma_{a_0}$ is a quasicircle. 
We refer the reader to \cite[Figure~14, \S 11.4]{LMMN} for the picture of a Schwarz dynamical plane such that the boundary of the quadrature domain has a $(5,2)-$cusp and the associated limit set forms a wedge at this cusp. 
\end{remark}

\begin{corollary}
For all $a\in\Gamma^{\mathrm{hoc}}$, the free critical orbit $\{\sigma_a^{\circ n}(y_f)\}$ converges to the cusp~$y_c$.
\end{corollary}

\begin{remark}
The existence of two distinct parameters $a$ for which $y_c$ is a $(7,2)-$cusp of $\partial\Omega_a$ can be interpreted as follows. For such a parameter, the interior of $K(\sigma_a)$ is the basin of attraction of $y_c$, and every component of $\Int{K(\sigma_a)}$ is mapped eventually to the $2-$cycle of Fatou components touching at $y_c$ (which correspond to the two attracting directions of $\sigma_a^{\circ 2}$ at $y_c$). Moreover, the first return map $\sigma_a^{\circ 2}$ to each of these two components is a unicritical holomorphic Blaschke product of degree equal to the valence of $\pmb{v_w'}$ in $\mathcal{T}$. Since such Blaschke products are rigid and since the external class of $\sigma_a$ is frozen, it follows that all parameters with a $(7,2)-$cusp at $y_c$ have conformally conjugate dynamics on the union of the tiling set and the interior of the non-escaping set. However, the internal and external conjugacies between two such maps would agree if and only if the circular order of the two periodic Fatou components which are marked because one of them contains the marked critical point $\pmb{f}(\eta_a(\pmb{v_w'}))$ while the other contains the free critical value $y_f$) and the droplet at $y_c$ are the same. This implies that there are at most two maps $\sigma_a$ for which $y_c$ is a $(7,2)-$cusp of $\partial\Omega_a$. The proof of Lemma~\ref{hoc_arc_lem} confirms that both possibilities are realized.
\end{remark}

\begin{theorem}\label{para_space_conn_thm}
\noindent\begin{enumerate}[leftmargin=8mm]
\item $\Int{S_{\mathcal{T}}}$ is a bounded Jordan domain. 

\item $\overline{S_{\mathcal{T}}}$ is a topological quadrilateral whose sides are given by $\Gamma^{\mathrm{hoc}}$, $\overline{\Gamma^{\perp}}$, and the two connected components of $\overline{\Gamma^{\mathrm{dp}}}$.
\end{enumerate}
\end{theorem}
\begin{proof}
(1) By Lemma~\ref{para_comp_jordan_lem}, it suffices to show that $\Int{S_{\mathcal{T}}}$ is connected.

Let us first note that by the projection argument of Lemma~\ref{para_comp_jordan_lem}, the boundary of each component of $\Int{S_{\mathcal{T}}}$ contains a non-degenerate interval in $\overline{\Gamma^{\perp}}$. As $\Gamma^{\mathrm{hoc}}\cap\overline{\Gamma^{\perp}}=\emptyset$, we conclude that the boundary of each component of $\Int{S_{\mathcal{T}}}$ must contain at least two points of $\overline{\Gamma^{\perp}}\cap\overline{\Gamma^{\mathrm{dp}}}$. It follows that if $\Int{S_{\mathcal{T}}}$ had two distinct components, then $\overline{\Gamma^{\perp}}\cap\overline{\Gamma^{\mathrm{dp}}}$ would contain at least three distinct points, which would contradict Lemma~\ref{dp_pb_two_lem}. Hence, $\Int{S_{\mathcal{T}}}$ is connected.

(2) By Lemma~\ref{para_comp_jordan_lem}, $\overline{S_{\mathcal{T}}}=\overline{\Int{S_{\mathcal{T}}}}$. By the proof of the first part of the theorem, $\partial S_{\mathcal{T}}$ is a Jordan curve in $\C$ consisting of a closed real-analytic arc $\Gamma^{\mathrm{hoc}}$, a closed interval $\overline{\Gamma^{\perp}}$, and a pair of closed real-analytic arcs $\overline{\Gamma^{\mathrm{dp}}}$ that connect the two endpoints of $\Gamma^{\mathrm{hoc}}$ to the two endpoints of $\overline{\Gamma^{\perp}}$. This yields the desired topological quadrilateral structure of $\overline{S_{\mathcal{T}}}$. 
\end{proof}

\section{Combinatorics of connected filled Julia sets}\label{dynamical_plane_sec}

Before we delve deeper into the topological structure of the connectedness locus $\cC(S_\cT)$, we collect some useful results on the dynamics of the associated Schwarz reflection maps.

 Recall from Corollary~\ref{schwarz_tree_cor} that for $a\in S_\cT$, the dessin d'enfants $\cT(\sigma_a)$ of the Belyi Schwarz reflection $\sigma_a$ is isomorphic to $\cT$ as a plane bicolored tree.

\textbf{Notation:} We denote the black (respectively, white) vertices of $\cT(\sigma_a)$ as $\{b_j\}$ and $\{w_i\}$, where $b_0=y_c=\pmb{f}(\pmb{v_b})$ is the root of the tree and $w_1=\pmb{f}(\eta_a(\pmb{v_w'}))$ is the white vertex adjacent to the root (cf. Definition~\ref{shabat_schwarz_family_def}). Note that $y_f=\sigma_a(w_i)$ is the free critical value of $\sigma_a$.

\begin{proposition}\label{TembK}For all $a\in \mathcal{C}(S_\mathcal{T})$ for which $K(\sigma_a)$ is locally connected, the tree $\mathcal T$ embeds into $K(\sigma_a)$, and the embedding is unique up to a homotopy fixing the vertices.\end{proposition}

\begin{proof}
Let $\gamma$ be a path in $K(\sigma_a)$ with endpoints $y_c$ and $y_f$. Then $\sigma_a^{-1}(\gamma)$ is an embedding of the tree $\cT$ in $K(\sigma_a)$. The simple connectivity of $K(\sigma_a)$ guarantees the uniqueness.
\end{proof}

We will show that even when the Julia set is not locally connected, the combinatorial structure of $\mathcal T$ is still present in $K(\sigma_a)$.
A dynamical $\theta-$ray of $\sigma_a$ (see Definition~\ref{dyn_ray_schwarz}) is said to be \emph{periodic} if $\theta$ is periodic under $\mathcal{R}_d$.

\begin{proposition}\label{per_rays_land_prop}
Let $a\in\cC(S_\cT)$. Then, all periodic dynamical rays of $\sigma_a$  land on $\partial K(\sigma_a)$. In particular, the $0-$ray lands at the cusp $y_c$.
\end{proposition}

\begin{proof}[Sketch of proof]
The proof is similar to that of \cite[Proposition~5.44]{LLMM1} (cf. \cite[Theorem~18.10]{Milnor06}). Let $R_a(\theta)$ be a periodic ray of period $p$. In the hyperbolic metric on $\widehat{\C}\setminus K(\sigma_a)$ the ray is a union of geodesic segments tending towards the boundary at infinity such that all segments have the same length and $\sigma_a$ sends one endpoint of the segment to the other. In particular, it must accumulate to at least one point on $K(\sigma_a)$, which we denote as $z_0$. We claim that $z_0$ is a periodic point with period dividing $p$. Let $U$ be a neighborhood of $z_0$. Let $I_{k}$ denote the union of $p$ fundamental segments of the ray, with endpoints at a rank $pk$ tile and a rank $p(k+1)$ tile. There is some $k$ large enough such that $I_{k}\subset U$. But now as $U\cap \sigma_a^{\circ p}(U)\supset I_{k}\cap I_{(k-1)}\neq \emptyset$, it follows that $z_0$ is fixed by $\sigma_a^{\circ p}$.

The set of accumulation points of a ray must be connected, and is thus either a single point or uncountable. But as the number of periodic points of $\sigma_a$ (on $\partial K(\sigma_a)$) of period dividing $p$ is finite, it follows that $z_0$ must be the only accumulation point of this ray.
\end{proof}

\begin{proposition}\label{rep_para_landing_points_prop}
Let $a\in\cC(S_\cT)$, and $z_0$ be a repelling or parabolic periodic point of $\sigma_a$ on $\partial K(\sigma_a)$. Then the following assertions are true.
\begin{enumerate}[leftmargin=8mm]
\item At least one periodic dynamical ray of $\sigma_a$ lands at $z_0$.

\item The number of rays landing at $z_0$ is equal to the number of connected components of $K(\sigma_a)\setminus\{z_0\}$.
\end{enumerate}
\end{proposition}
\begin{proof}
The proofs of \cite[Theorems~18.11,~18.13]{Milnor06}, \cite[Theorems~24.5,~24.6]{Lyu25} apply mutatis mutandis to the present setting.
\end{proof}

\begin{corollary}\label{corollary: crit_points_separated}
For $a\in \cC(S_\cT)$ no two distinct vertices of $\cT(\sigma_a)$ can lie in the same Fatou component of $\sigma_a$.
\end{corollary}
\begin{proof}
We note that the black vertices $b_j$ of $\cT(\sigma_a)$ lie on $\partial K(\sigma_a)$. Further, by Proposition~\ref{per_rays_land_prop}, each of them is the landing point of at least $\mathrm{val}(b_j)$ many dynamical rays, where $\mathrm{val}(b_j)$ is the valence of the vertex $b_j$ in the tree $\cT(\sigma_a)$.

Let $w_i$ and $w_{i'}$ be distinct white vertices of $\cT(\sigma_a)$. Then, any sub-continuum $K'$ of $K(\sigma_a)$ containing $w_i, w_{i'}$ also contains some vertex $b_j$ of $\cT(\sigma_a)$ with $\mathrm{val}(b_j)\geq 2$ (where $b_j$ lies between $w_i$ and $w_{i'}$ in the tree). It follows from the previous paragraph that $b_j$ is a cut-point of $K(\sigma_a)$ and hence also a cut-point of $K'$. This clearly implies that $w_i$ and $w_{i'}$ cannot lie in the same Fatou component of $\sigma_a$.
\end{proof}

Define the \textit{principal sectors} for a Schwarz reflection $\sigma_a\in S_\mathcal T$ to be the planar regions whose boundaries are given by the dynamical rays landing at preimages of the cusp $y_c$, and such that no such dynamical ray is in the interior. When $\sigma_a\in \mathcal C(S_\mathcal T)$ we may intersect these principal sectors with $K(\sigma_a)$ to arrive at the following analogue of \ref{TembK}:

\begin{proposition}\label{partition_prop}
Let $a\in \mathcal C(S_\cT)$. Then $K(\sigma_a)$ is the union of finitely many sub-continua, any two of which intersect in at most one point, namely at a critical preimage of the cusp $y_c$ (i.e., at a black vertex $b_j$ with $\mathrm{val}(b_j)\geq 2$). Each such sub-continuum contains exactly one white vertex $w_i$.

Let $G$ be the plane tree whose vertices are preimages of the cusp $y_c$ and the free critical value $y_f$, and with edges connecting a preimage of $y_c$ to a preimage of $y_f$ if they lie in the same sub-continuum as above. Then $G$ is isomorphic as a bi-colored plane rooted tree to $\mathcal T$.
\end{proposition}

\subsection{Hubbard trees for Schwarz reflections}

We will say that a Schwarz reflection $\sigma_a\in S_\cT$ is \emph{postcritically finite} if the orbit of the free critical value $y_f$ is~finite.

To describe the combinatorics of Schwarz reflections where all non-escaping critical points have finite forward orbits, one can define its \emph{augmented Hubbard tree} precisely as in Definition~\ref{def_augmented_hubbard_tree_poly}. 

\begin{definition}\label{aug_hub_tree_schwarz_def}
    For a postcritically finite Schwarz reflection $\sigma_a\in S_\cT$, define the \textit{augmented Hubbard tree} $\mathcal{H}(\sigma_a)$ of $\sigma_a$ to be the minimal hull in the filled Julia set of the post-critical set, preimages of the critical values, the landing point of the $0$ ray (i.e., the cusp), and its preimages.\\
  Furthermore, if the cusp has valence larger than one in the augmented Hubbard tree, we record the cyclic ordering at the cusp of the edges of the tree and the $0-$ray.
\end{definition}

\begin{remark}
    For maps in $\mathfrak{L}_\cT$ such that the critical points not in the marked parabolic basin have finite forward orbits, we may define augmented Hubbard trees in an analogous way.
\end{remark}

For a PCF Schwarz reflection $\sigma_a\in S_\cT$, one can choose the dessin $\cT(\sigma_a)$ so that it lies entirely in the non-escaping set $K(\sigma_a)$. Indeed, this can arranged by defining $\cT(\sigma_a)$ to be the tree $\sigma^{-1}(\gamma_a)$, where $\gamma_a$ is an arc connecting $y_c, y_f$ in $K(\sigma_a)$ (cf. Corollary~\ref{schwarz_tree_cor}). The following result is a consequence of this construction, the definition of the augmented Hubbard tree, and fullness of $K(\sigma_a)$.

\begin{proposition}\label{dessin_in_aug_hub_tree_prop}
    For a PCF Schwarz reflection $\sigma_a\in S_\cT$, the augmented Hubbard tree $\cH(\sigma_a)$ contains the dessin $\cT(\sigma_a)$, up to homotopy in the interior of $K(\sigma_a)$.
\end{proposition}

Poirier's realization theorem for PCF anti-polynomials (see Theorem~\ref{poi_thm}), together with Theorem~\ref{Thm_C_part_I} allows the construction of many post-critically finite maps in $S_\cT$. We give a construction below showing that any `critically periodic' map in $S_\cT$ can be enriched by gluing in the dynamics of PCF unicritical (anti-)holomorphic polynomials.

\begin{lemma}[Tuning Trees]\label{tuning_trees_lem}
    Let 
    \begin{enumerate}[leftmargin=8mm]
        \item $\sigma_a\in S_\cT$ be \emph{critically periodic of period $n$}; i.e., one of the white critical points $w_i$ of $\sigma_a$ lies in an $n-$cycle;
        \item $d_i$ be the local degree of $\sigma_a$ at $w_i$; and 
        \item $p$ be a degree $d_i$ unicritical, hyperbolic, PCF polynomial, with $p$ anti-holomorphic if $n$ is odd and holomorphic if $n$ is even.
        \end{enumerate}    
    Then there exists a PCF map $\sigma_{a'}\in S_\cT$ such that there is a neighborhood $U\ni y_f$ of the free critical value with $U\subset \sigma_{a'}^{\circ n} (U)$ and a David conjugacy from $p$, restricted to some neighborhood of $\cK(p)$, to $\sigma_{a'}^{\circ n}\vert_U$.
\end{lemma}

\begin{proof}
We will construct $\sigma_{a'}$ by constructing its Hubbard tree. Denote the augmented Hubbard tree of $p$ as $\cH(p)$ and the augmented Hubbard tree of $\sigma_a$ as $\mathcal{H}(\sigma_a)$.

Recall that $\cH(p)$ contains the landing point for the dynamical ray at angle $0$, which we denote $\beta$, as well as all preimages of $\beta$. We produce a \emph{tuning} of $\mathcal{H}(\sigma_a)$ by $\cH(p)$ as follows (see \cite[\S 5.2]{IK12} for the dual notion of tuning rational laminations). Delete from $\mathcal{H}(\sigma_a)$ the point $w_i$ and the open edges of $\mathcal{H}(\sigma_a)$ incident at $w_i$, giving $d_i$ components which are circularly ordered. Denote the resulting disconnected graph by $\cH^{\mathrm{del}}(\sigma_a)$. Now attach the component of $\cH^{\mathrm{del}}(\sigma_a)$ which contains the root $y_c$ of $\cT(\sigma_a)$ to $\cH(p)$ at $\beta$, and attach the remaining components of $\cH^{\mathrm{del}}(\sigma_a)$ to the preimages of $\beta$, maintaining the circular ordering. This amounts to `blowing up' $w_i$ to the tree $\cH(p)$. Note that every other point of the critical cycle maps to $w_i$ locally injectively under some iterate of $\sigma_a$. Hence, we can blow up these points of the critical cycle and glue in copies of $\cH(p)$, where the gluings are determined by the action of $\sigma_a$ and the gluing of $\cH(p)$ at $w_i$.

Denote the resulting `tuned' tree as $\widehat{\mathcal{H}}$. We now equip $\widehat{\mathcal{H}}$ with a tree map. The copy of $\cH(p)$ (in $\widehat{\mathcal{H}}$) corresponding to $w_i$ is mapped to the copy of $\cH(p)$ (in $\widehat{\mathcal{H}}$) corresponding to $y_f$ using the action induced by $p$ or $\overline{p}$ (ensuring that the map is orientation-reversing). Each of the other copies of $\cH(p)$ (i.e., the ones corresponding to the non-critical points in the critical cycle) is mapped cyclically to the next one as an orientation-reversing isomorphism. We define the map on the rest of $\widehat{\mathcal{H}}$ by mapping all black vertices of $\cT(\sigma_a)\subset\mathcal{H}(\sigma_a)$ to the root, and mapping all remaining white vertices of $\cT(\sigma_a)\subset\mathcal{H}(\sigma_a)$ to the critical value of the tree map contained in the embedding of $\cH(p)$ corresponding to $y_f$. Finally, we extend the map homeomorphically to the edges.

Recall from Definition~\ref{angled_tree_dyn_def} (cf. \cite{Poi13}) that a tree map is called expanding if any two adjacent Julia type vertices are eventually mapped to non--adjacent vertices. Since $p$ is a hyperbolic polynomial, it follows from the construction that the Julia type vertices of $\widehat{\cH}$ are precisely the black vertices of $\cH(\sigma_a)$; i.e., the vertices which map to $y_c$ in one iterate of the tree map. These vertices are already separated by some white vertex in $\cH(\sigma_a)$ and thus there is no pair of adjacent Julia type vertices in $\widehat{\cH}$. It follows that the tree map on $\widehat{\cH}$ is expanding, and hence by \cite[Theorem 5.1]{Poi13}, there exists a postcritically finite, anti-holomorphic polynomial $q$, whose augmented Hubbard tree dynamics is conjugate to the above tree map on $\widehat{\mathcal{H}}$. By construction, $q^{\circ n}$ has a (anti-)polynomial-like restriction that is hybrid equivalent to $p$.

Note that $q$ has two finite critical values, one of which is fixed. Furthermore for an arc $I$ connecting the two critical values, $q^{-1}(I)$ is isomorphic as a bi-colored planar embedded tree to $\mathcal T$.  By Theorem~\ref{Thm_C_part_I}, the map $q$ is David conjugate to a Schwarz reflection in $\cS_{\mathcal{R}_d}$, and by the critical point structure just described, such a Schwarz reflection $\sigma_{a'}$ belongs to $\cC(S_\mathcal T)$. The map $\sigma_{a'}$ is the desired PCF map in~$S_\cT$.
\end{proof}

\section{Hyperbolic components in $S_{\mathcal{T}}$}\label{hyp_comp_sec}

We say that a parameter $a\in S_\cT$ is \textit{hyperbolic} if $\sigma_a$ has an attracting cycle\footnote{Strictly speaking, we should call such parameters \textit{relatively} hyperbolic, as there is always a fixed point with parabolic dynamics. However, as this behavior is persistent throughout our parameter space, we call these maps hyperbolic for the sake of brevity.}. Let $\cU_a$ be the immediate basin of the attracting cycle; i.e., the union of the Fatou components which contain the attracting periodic cycle. Necessarily there is a critical point $c_a$ contained in $\cU_a$ \cite{Milnor06}. 

\begin{proposition}
There are no critical or co-critical points in $\cU_a$ other than $c_a$.
\end{proposition}
\begin{proof}
This is a consequence of Corollary~\ref{corollary: crit_points_separated} and the fact that each Fatou component containing a white vertex $w_i$ maps under $\sigma_a$ to the periodic Fatou component containing $y_f$.
\end{proof}

With notation as above, we will call the critical point $c_a$ the \emph{distinguished critical point} of $\sigma_a$. Further, we denote the periodic Fatou component containing $c_a$ as $U_a$, and call it the \emph{distinguished Fatou component} of $\sigma_a$.

\begin{definition}
We define the \textit{primary} hyperbolic component $H$ to be the hyperbolic component of $\mathcal C(S_\mathcal T)$ such that each $\sigma_a\in H$ has an attracting fixed point and  the distinguished critical point $c_a$ of $\sigma_a$ is $w_1$, which is the critical point adjacent to the cusp point $y_c$ in $\mathcal T(\sigma_a)$.
\end{definition}

\begin{proposition}\label{hoc_bdry_primary}
$\Gamma^{\mathrm{hoc}}\subset\overline{H}\subset\mathcal{C}(S_\cT)$. Moreover, any sequence in $S_{\mathcal{T}}$ accumulating on the interior of $\Gamma^{\mathrm{hoc}}$ is eventually contained in $H$.
\end{proposition}
\begin{proof}
Let $a$ lie in the interior of $\Gamma^{\mathrm{hoc}}$. Then the cusp $y_c$ on $\partial\Omega_a$ is a $(5,2)-$cusp, and $y_c$ has an attracting direction in $\Omega_a$. Further, $\sigma_a$ has a Fatou component $V_a$ containing this attracting direction such that $w_1\in V_a$, $y_c\in\partial V_a$, and all points in $V_a$ converge to $y_c$ (asymptotically to the attracting direction).

One can perturb $a$ slightly into the interior of $S_\cT$ to produce parameters $a'$ for which the cusp $y_c$ splits into two distinct fixed points: one of them is the cusp $y_c$ itself with only a repelling direction in $\Omega_a$ (for the parameter $a'$, the cusp $y_c$ is of type $(3,2)$), and the other is an attracting fixed point. Moreover, the Fatou component $\cU_a$ containing this attracting fixed point also contains $w_1$. Hence, the parameter $a\in\Gamma^{\mathrm{hoc}}$ can be approximated by parameters $a'\in H$.
\end{proof}

\subsection{Uniformizing hyperbolic components}\label{hyp_comp_unif_subsec}

For a hyperbolic map $\sigma_a$, the orbit of the free critical value $y_f$ stays in $K(\sigma_a)$, and thus $a\in\cC(S_\cT)$. Suppose that for a hyperbolic map $\sigma_a$, the period of the $U_a$ is $p$.
Since the distinguished Fatou component $U_a\ni c_a$ is topologically equivalent to $\D$, we can choose a conformal isomorphism $\phi_a\colon \D\to U_a$ that sends the origin to $\sigma_a^{\circ p}(c_a)$. Then,
the first return map $\sigma_a^{\circ p}|_{U_a}$  is a proper (anti-)holomorphic map, and hence $\phi^{-1}\circ \sigma_a^{\circ p} \circ \phi\colon \D\to \D$ is a Blaschke product.
With the above normalization, the resulting holomorphic and anti-holomorphic (unicritical) Blaschke products are of the form
$$
B^+_{\alpha,\lambda,n}(z)=\lambda\left(\frac{z-\alpha}{\bar \alpha z-1}\right)^n \text{ and } B^-_{\alpha,\lambda,n}(z)= \overline{\lambda\left(\frac{z-\alpha}{\bar\alpha z-1}\right)^n},
$$
where $|\lambda|=1$, $\alpha\in \D$, and $n=\deg{\sigma_a^{\circ p}\vert_{U_a}}=\mathrm{val}(c_a)$. This argument, as explained in \cite[Lemma 3.2]{NS}, gives rise to the following proposition.

\begin{proposition}Let $\sigma_a\in S_\cT$ be hyperbolic, $U_a$ be the $p-$periodic Fatou component of $\sigma_a$ containing a critical point, and $n=\deg{\sigma_a^{\circ p}\vert_{U_a}}$. Then there exist $\alpha\in \D,\lambda\in\mathbb{S}^1$ such that the first return map $\sigma_a^{\circ p}|_{U_a}$ is conformally conjugate either to $B^+_{\alpha,\lambda,n}$ if $p$ is even or to $B^-_{\alpha,\lambda,n}$ if $p$ is odd.

Normalizing in such a way that $1$ is a fixed point of the Blaschke product, there are $d-1$ choices of $(\alpha,\lambda)$ in the holomorphic case, and $d+1$ choices of $(\alpha,\lambda)$ in the anti-holomorphic case.\end{proposition}

\begin{remark}As noted in \cite{NS}, not every choice of $(\alpha,\lambda)$ gives rise to a fixed point of the Blaschke product in $\D$.\end{remark}

We will also record a parameter version of this statement. Let $\mathcal B_n^\pm$ be the spaces of maps $B^\pm_{\alpha,\lambda,n}$ such that there exists an attracting fixed point in $\D$.

For an anti-holomorphic map $f$ with an attracting fixed point at $\zeta$, it was shown in \cite[Lemma~5.1]{NS} that there exists a local conformal coordinate $\kappa\colon V(\ni\zeta)\hookrightarrow \C$ which linearizes $f$, in the sense that $\kappa\circ f = \bar\partial f(\zeta) \cdot \overline \kappa$.

For a hyperbolic map $\sigma_a\in S_\cT$, let $\zeta_a$ be the periodic point that lies in the distinguished Fatou component $U_a\ni c_a$. As mentioned in the previous paragraph, the first return map $\sigma_a^{\circ p}$ has a local linearizing coordinate $\kappa_a$ near $\zeta_a$. This coordinate can be extended to a closed neighborhood $\overline{W_a}\ni \zeta_a$ such that $c_a\in \partial W_a$. The map $\kappa_a$ becomes unique if we require $\kappa_a(c_a)=1$.

\begin{definition}
Let $H$ be a hyperbolic component of odd period $p$. Then, the \emph{critical value map} for $H$ is defined as $a\mapsto \kappa_a(\sigma_a^{\circ p}(c_a))$.
\end{definition}

The next result can be proved by adapting the arguments of \cite[Theorems~5.6,~5.9]{NS} to the current situation.

\begin{theorem} 
Every even period hyperbolic component whose critical point in the periodic critical Fatou component has local degree $n$ is homeomorphic to $\mathcal B^+_n$, and the homeomorphism respects the multiplier of the attracting fixed point.
Likewise, every odd period hyperbolic component whose critical point in the periodic critical Fatou component has local degree $n$ is homeomorphic to $\mathcal B^-_n$, and the homeomorphism respects the critical value map.
In particular, all hyperbolic components are simply connected.
\end{theorem}

\subsection{Boundaries of hyperbolic components}\label{hyp_comp_bdry_subsec}

The unicriticality of our hyperbolic components allows many statements about the boundaries of hyperbolic components in the Multicorn family to carry over to our setting as well (cf. \cite{MNS}).
We will say that a parameter is \textit{parabolic} if $\sigma_a$ has a periodic cycle (other than the fixed point $y_f$) whose multiplier is a root of unity.

\begin{proposition}
\noindent\begin{enumerate}[leftmargin=8mm]
\item The boundary of a hyperbolic component of odd period $k$ of $\mathcal C(S_\mathcal T)$ which is not the primary hyperbolic component is contained in the interior of $S_\mathcal T$, and consists entirely of parameters having a parabolic orbit of exact period $k$. In suitable local conformal coordinates, the $2k-$th iterate of such a map has the form $z\mapsto z+z^{q+1}+\cdots$ with $q\in\{1, 2\}$.

\item Every parameter on the boundary of the primary hyperbolic component $\mathcal H$ is either contained in $\Gamma^\text{hoc}$ or has a parabolic fixed point (with local power series as above).

\item Every parameter on the boundary of an even period hyperbolic component has a neutral cycle.
\end{enumerate}
\end{proposition}

\begin{proof}
(1) The boundary of any hyperbolic component $H$ will either intersect $\Gamma^\text{hoc}$ or be contained in the interior of $\mathcal C(S_\mathcal T)$. By \ref{hoc_bdry_primary} it follows that for a non-primary hyperbolic component that its boundary is contained in the interior of $\mathcal C(S_\mathcal T)$. Now an application of \cite[Lemma 2.5]{MNS} combined with the fact that the Schwarz reflection maps under consideration have unique free critical values show that for every parameter on the boundary of $H$, the $k-$cycle to which the critical orbit converges must be parabolic with the desired local Taylor series expansion.

(2) The proof is analogous to part (1), accounting for the fact that $\Gamma^\text{hoc}\subset \overline{\mathcal H}$ by Proposition~\ref{hoc_bdry_primary}.

(3) As a parameter moves from the interior to the boundary of an even period hyperbolic component, the multiplier of the unique attracting cycle tends to $1$ in absolute value and produces a neutral cycle in the limit.
\end{proof}

\begin{definition}
Let $a$ be a parameter with a parabolic periodic point of odd period. If for the above proposition $q=1$ we say that it is a \textit{simple} parabolic parameter. If $q=2$, we say that it is a \textit{parabolic cusp}.
\end{definition}

Recall that in Lemma~\ref{hoc_arc_lem}, we constructed a real-analytic arc in $\Gamma^{\mathrm{hoc}}\subset S_\cT$ (consisting of parameters for which the quadrature domain has a $(5,2)-$cusp on its boundary) using a quasiconformal deformation argument that allows one to change the {\'E}calle height of the free critical value. This deformation argument applies verbatim to simple parabolic parameters (of odd period) and yields the following result.

\begin{proposition}
Let $\widetilde a$ be a simple parabolic parameter of odd period. Then $\widetilde a$ is on a parabolic arc in the following sense: there exists a real-analytic arc of simple parabolic parameters $a(h)$ (for $h \in \R$) with quasiconformally equivalent but conformally distinct dynamics of which $\widetilde a$ is an interior point, and the {\'E}calle height of the free critical value of $\sigma_{a(h)}$ is $h$. This arc is called a \textit{parabolic arc}.
\end{proposition}

For a critical point of local degree $n$, there are $n+1$ distinct combinatorial ways for parabolic arcs to occur. This fact can be used to conclude the following (cf. \cite[Theorem~1.2]{MNS}, \cite{LLMM2}).
\begin{proposition}
(1) The boundary of every non-primary hyperbolic component of odd period of $\mathcal C(S_\mathcal T)$ is a topological polygon with $n+1$ sides having parabolic cusps as vertices and parabolic arcs as sides, where $n$ is the local degree of the distinguished critical point.

(2) The boundary of the primary hyperbolic component consists of $m$ parabolic arcs, $\Gamma^\text{hoc}$, and $m-1$ parabolic cusps, where $m$ is the valence of the vertex $w_1$ (which is adjacent to the root $y_c$) in $\cT(\sigma_a)$.
\end{proposition}

\begin{proposition}
If $\sigma_a$ has a neutral periodic point of period $k$, then every neighborhood of $a$ in $S$ contains parameters with attracting periodic points of period $k$, so the parameter $a$ is on the boundary of a hyperbolic component of period $k$ of $\mathcal C(S_\mathcal T)$.
\end{proposition}
\begin{proof}
See \cite[Theorem 2.1]{MNS} for a proof in the Multicorn family. Since the proof given there only uses local dynamical properties of anti-holomorphic maps near neutral periodic points, it applies to the family $S_\mathcal T$ as well.
\end{proof}

The following proposition is specific to our setting.
We say that two hyperbolic components of $\mathcal C(S_\mathcal T)$ are adjacent if they have intersecting closures.

\begin{proposition}
Let $H_1$ and $H_2$ be two adjacent hyperbolic components in $\cC(S_\cT)$. Then the distinguished critical points associated with their centers correspond to the same white vertex in $\mathcal T$.
\end{proposition}
\begin{proof}
Let us first note that the distinguished critical point of a hyperbolic map $\sigma_a$ has the same address in $\cT$ as $a$ runs over a fixed hyperbolic component. Indeed, the principal sectors and the white critical points of $\sigma_a$ depend continuously on $a$, and the distinguished Fatou component remains in the same principal sector throughout a hyperbolic component.

By way of contradiction, suppose that there were two adjacent hyperbolic components $H_1, H_2$ such that the distinguished critical points of maps in these components have different addresses in $\mathcal T$. By the arguments of the previous paragraph, this would force the free critical value $y_f$ (which is the image of the distinguished critical point) to cross the common boundary of two principal sectors as $H_1$ bifurcates to $H_2$. But this would produce a parameter $a'\in\partial H_1\cup\partial H_2$ such that $y_f$ coincides with one of the $\sigma_{a'}-$preimages of the cusp $y_c$.
But a PCF parameter cannot be on the boundary of a hyperbolic component.
\end{proof}

The above proposition allows us to consider bifurcations of hyperbolic parameters analogously to the standard unicritical case. Specifically, the proofs of \cite[Theorem~1.1,]{MNS}, \cite[Proposition~3.7,~Theorem~3.8,~Corollary~3.9]{HS} can be adapted for the family $S_\cT$ to prove the following results.

\begin{proposition}
If $\sigma_a$ has a 2$k-$periodic cycle with multiplier $e^{2\pi p/q}$ with $\gcd(p, q) =1$, then the parameter $a$ sits on the boundary of a hyperbolic component of period $2kq$ (and is the root thereof) of $\mathcal C(S_\mathcal T)$.
\end{proposition}

\begin{proposition}
Every parabolic arc of odd period $k>1$ intersects the boundary of a hyperbolic component of period $2k$ along an arc consisting of the set of parameters where the parabolic fixed point index is at least $1$. In particular, every parabolic arc has, at both ends, an interval of positive length at which bifurcation from a hyperbolic component of odd period $k$ to a hyperbolic component of period $2k$ occurs.
\end{proposition}

\section{Tessellation of the escape locus of $S_\cT$}\label{escape_unif_sec}

We will now provide a dynamically natural uniformization of the escape locus $S_\cT\setminus\cC(S_\cT)$ in terms of the conformal position of the free critical value $y_f$. This will be done using the dynamical uniformization $\psi_a$ of the tiling set (see Proposition~\ref{schwarz_group}).

\subsection{Uniformization of The escape locus}\label{escape_unif_subsec}

\begin{theorem}\label{escape_unif_thm}
The map 
$$
\pmb{\Psi}:S_{\mathcal{T}}\setminus\cC(S_\cT)\to\Int{(\mathcal{Q})}\setminus\mathcal{Q}_1,\quad a\mapsto\psi_a^{-1}(y_f)
$$
is a homeomorphism. In particular, the connectedness locus is connected.
\end{theorem}
\begin{proof}
The proof is analogous to that of \cite[Theorem~1.3]{LLMM2}. We only indicate the key differences.

Note that for all $a\in S_{\mathcal{T}}$, the critical value $y_f$ of $\sigma_a$ lies in $\Omega_a$; i.e., $y_f\notin T^0(\sigma_a)$. It now follows from the definition of $\psi_a$ that $\psi_a^{-1}(y_f)\in\Int{(\mathcal{Q})}\setminus\mathcal{Q}_1$ for each $a\in S_{\mathcal{T}}\setminus\cC(S_\cT)$.

Continuity of $\pmb{\Psi}$ follows from the fact that the rank zero tile $T^0(\sigma_a)$ changes continuously as the
parameter $a$ runs over $S_\cT$, and hence the conformal map 
$$
\psi_a^{-1}:\displaystyle\bigcup_{n=0}^{n(a)} \sigma_a^{-n}(T^0(\sigma_a))\to\displaystyle\bigcup_{n=0}^{n(a)} \mathcal{R}_d^{-n}(\mathcal{Q}_1)
$$
also varies continuously with $a$.

We will show that $\pmb{\Psi}$ is proper and injective. Injectivity will imply (by the Invariance of Domain Theorem) that $\pmb{\Psi}$ is a homeomorphism onto its image, and properness will give surjectivity of the map.
\medskip

\noindent\textbf{Properness of $\pmb{\Psi}$.}
We need to consider several cases to show that $\pmb{\Psi}$ is proper. Let us first assume that $\{a_k\}_k$ is a sequence in $S_{\mathcal{T}}\setminus\cC(S_\cT)$ such that $a_k\to\overline{\Gamma^{\perp}}$. It follows from the definition of $\Gamma^{\perp}$ that the spherical distance between the co-critical point $\pmb{v_w}$ of $\pmb{f}$ and the circle $\partial\Delta_a$ tends to zero as $k\to\infty$. Hence, $d_{\mathrm{sph}}\left(y_f,\partial\Omega_{a_k}\right)$ tends to $0$ as $k\to\infty$. Therefore, $\pmb{\Psi}(a_k)=\psi_{a_k}^{-1}(y_f)$ accumulates on $\partial\mathcal{Q}_1$.

Now suppose that $\{a_k\}_k\subset S_{\mathcal{T}}\setminus\cC(S_\cT)$ is a sequence with $\{a_k\}_k\to a\in\Gamma^{\mathrm{dp}}$. It then follows from Proposition~\ref{generic_bdd_droplet_dyn_prop} that the free critical value $y_f$ of $\sigma_{a}$ lands in the bounded component $T_{b}^0(\sigma_{a})$ of $T^0(\sigma_{a})$ under $\sigma_{a}^{\circ n}$, for some $n\equiv n(a)\geq 1$. Note that for $k$ sufficiently large, $\sigma_{a_k}$ is a small perturbation of $\sigma_{a}$. We set 
$$
U_k:= \Int{(T^0(\sigma_{a_k})\cup\sigma_{a_k}^{-1} (T^0(\sigma_{a_k})))}.
$$ 
Then, for $k$ large enough, the critical value $y_f$ of $\sigma_{a_k}$ lands in $T^0(\sigma_{a_k})$ under $\sigma_{a_k}^{\circ n'}$, where $n'\in\{n,n+1\}$. Moreover, the hyperbolic geodesic in $U_k$ connecting $\sigma_{a_k}^{\circ n'}(y_f)$ and $\infty$ passes through a narrow channel formed by the splitting of the double point on $\partial T^0(\sigma_{a})$ such that the thickness of this channel goes to zero as $k\to\infty$ (i.e., it gets pinched in the limit). Since the Euclidean distance between this part of the geodesic and the boundary of $U_k$ tends to zero as $k\to\infty$, the hyperbolic distance between $\sigma_{a_k}^{\circ n'}(y_f)$ and $\infty$ (in $U_k$) tends to $\infty$ as $k\to\infty$. Furthermore, as $\psi_{a_k}^{-1}$ is a conformal isomorphism from $U_k$ onto $\Int{(\mathcal{Q}_1\cup\mathcal{R}_d^{-1}(\mathcal{Q}_1))}$, we have that the hyperbolic distance between $\psi_{a_k}^{-1}(\sigma_{a_k}^{\circ n'}(y_f))$ and $0$ (in $\Int{(\mathcal{Q}_1\cup\mathcal{R}_d^{-1}(\mathcal{Q}_1))}$) tends to $\infty$ as $k\to\infty$. Consequently, $\{\psi_{a_k}^{-1}(\sigma_{a_k}^{\circ n'}(y_f))\}_k$ escapes to the boundary of $\Int{(\mathcal{Q}_1\cup\mathcal{R}_d^{-1}(\mathcal{Q}_1))}$ as $k\to\infty$. But the sequence $\{\psi_{a_k}^{-1}(\sigma_{a_k}^{\circ n'}(y_f))\}_k$ is contained in $\mathcal{Q}_1$, and hence, $\{\psi_{a_k}^{-1}(\sigma_{a_k}^{\circ n'}(y_f))\}_k$ must converge to $1\in\partial\mathcal{Q}_1$. In fact, it follows from the proof of Proposition~\ref{generic_bdd_droplet_dyn_prop} that for $k$ sufficiently large, each $\psi_{a_k}^{-1}(\sigma_{a_k}^{\circ j}(y_f))$, $j\in\{0,\cdots, n'\}$, is close to $1\in\partial\mathcal{Q}_1$. Hence $\pmb{\Psi}(a_k)=\psi_{a_k}^{-1}(y_f)$ converges to $1\in\partial\mathcal{Q}_1$ as $k\to\infty$.

Finally let $\{a_k\}_k\subset S_{\mathcal{T}}\setminus\cC(S_\cT)$ be a sequence accumulating on $\cC(S_\cT)$. Suppose that $\{\pmb{\Psi}(a_k)\}_k$ converges to some $u\in\Int{(\mathcal{Q})}\setminus\mathcal{Q}_1$. Then, $\{\psi_{a_k}^{-1}(y_f)\}_k$ is contained in a compact subset $\mathcal{X}$ of $\Int{(\mathcal{Q})}\setminus\mathcal{Q}_1$. After passing to a subsequence, we can assume that $\mathcal{X}$ is contained in a single tile of $\Int{(\mathcal{Q})}\setminus\mathcal{Q}_1$. But this implies that each $a_k$ has a common depth $n_0$ (see Definition~\ref{def_depth}), and $\psi_{a_k}^{-1}(\sigma_{a_k}^{\circ n_0}(y_f))$ is contained in the compact set $\mathcal{R}_d^{\circ n_0}(\mathcal{X})\subset\mathcal{\mathcal{Q}}_1$ for each $k$. Note that the map $\sigma_{a'}$, the fundamental tile $T^0(\sigma_{a'})$ as well as (the continuous extension of) the conformal isomorphism $\psi_{a'}:\mathcal{Q}_1\to T^0(\sigma_{a'})$ change continuously with the parameter as $a'$ runs over $S_{\mathcal{T}}$. Therefore, for every accumulation point $a$ of $\{a_k\}_k$, the point $\sigma_{a}^{\circ n_0}(y_f)$ belongs to the compact set $\psi_{a}(\mathcal{R}_d^{\circ n_0}(\mathcal{X}))$. In particular, the critical value $y_f$ of $\sigma_{a}$ lies in the tiling set $T^\infty(\sigma_a)$. This contradicts the assumption that $\{a_k\}_k$ accumulates on $\cC(S_\cT)$, and proves that $\{\pmb{\Psi}(a_k)\}_k$ must accumulate on the boundary of $\Int{(\mathcal{Q})}\setminus\mathcal{Q}_1$.
\medskip

\noindent\textbf{Injectivity of $\pmb{\Psi}$.}
For $a\in S_\cT$, we denote the union of the tiles of rank zero through $k$ for the Schwarz reflection map $\sigma_a$ by $E_a^k$; i.e., 
$$
E_a^k:=\displaystyle\bigcup_{i=0}^{k} \sigma_{a}^{-i}(T^0(\sigma_{a})).
$$

Let $\pmb{\Psi}(a_1)=\pmb{\Psi}(a_2)$, for $a_1,a_2\in S_\cT\setminus\cC(S_\cT)$. Then, the parameters $a_1$ and $a_2$ have the same depth, which we denote by $n_0$. 
By Proposition~\ref{schwarz_group}, for $i\in\{1,2\}$, there exist conformal maps 
$$
\psi_{a_i}^{-1}:E_{a_i}^{n_0}\longrightarrow \displaystyle\bigcup_{k=0}^{n_0} \mathcal{R}_d^{-k}(\mathcal{Q}_1)
$$ 
conjugating $\sigma_{a_i}$ to $\mathcal{R}_d$ (wherever the maps are defined).

We define a conformal homeomorphism 
$$
\psi_{a_1}^{a_2}:=\psi_{a_2}\circ\psi_{a_1}^{-1}: E_{a_1}^{n_0}\longrightarrow E_{a_2}^{n_0}
$$
that conjugates $\sigma_{a_1}$ to $\sigma_{a_2}$. Clearly, $\psi_{a_1}^{a_2}$ extends continuously to the unique cusp of $\partial\Omega_{a_1}$ and sends it to the unique cusp of $\partial\Omega_{a_2}$. Since $\psi_{a_1}^{-1}(y_f)=\psi_{a_2}^{-1}(y_f)$, the map $\psi_{a_1}^{a_2}$ carries the (unique) free critical value $y_f$ of $\sigma_{a_1}$ to that of $\sigma_{a_2}$. Further, the maps $\sigma_{a_1}$ and $\sigma_{a_2}$ have the same branching structure over $y_f$. Hence, $\psi_{a_1}^{a_2}$ can be lifted via the maps $\sigma_{a_1}, \sigma_{a_2}$ to a conformal homeomorphism from $E_{a_1}^{n_0+1}$ onto $E_{a_2}^{n_0+1}$ such that it respects the cusp and the marked critical points. Due to the equivariance property of $\psi_{a_1}^{a_2}$ (with respect to $\sigma_{a_1}, \sigma_{a_2}$), we can normalize the lift so that it agrees with $\psi_{a_1}^{a_2}$ on their common domain of definition. Thus, the lifted map is an extension of $\psi_{a_1}^{a_2}$ (to the tiles of rank $n_0+1$) and we denote this extension by $\psi_{a_1}^{a_2}$ as well.

As $a_i$ lies in the escape locus of $S_\cT$, the unique cusp on $\partial\Omega_{a_i}$ is a $(3,2)-$cusp, $i\in\{1,2\}$ (by Lemma~\ref{cusp_anal_dyn_lem}). Thus, by \cite[Lemmas~B.2,~B.3]{LLMM2}, the conformal map $\psi_{a_1}^{a_2}$ is asymptotically linear near the cusp. 
Since $\sigma_{a_i}$ is anti-holomorphic on $\Omega_{a_i}$, it follows that $\psi_{a_1}^{a_2}$ is also asymptotically linear near the singular points of $E_{a_1}^{n_0+1}$ (these singular points are the iterated preimages of the cusp). Since $\partial E_{a_i}^{n_0+1}$ is non-singular real-analytic away from the iterated preimages of the cusp, it follows that $\psi_{a_1}^{a_2}$ extends as a quasiconformal homeomorphism of $\widehat{\C}$.
Recall that by design, $\mathfrak{g}_1$ is a conformal conjugacy between $\sigma_{a_1}$ to $\sigma_{a_2}$ on the union of the tiles of rank $0$ through $n_0+1$. 

A pullback argument as in the proof of Lemma~\ref{dp_pb_two_lem} (cf. \cite[Theorem~1.3]{LLMM2}) now yields a global quasiconformal homeomorphism $\mathfrak{g}_\infty$ that is a conformal conjugacy between $\sigma_{a_1}$ and $\sigma_{a_2}$ on their tiling sets. As the non-escaping sets of these maps are Cantor sets of zero area (cf. \cite[Proposition~5.52, Proposition~6.8]{LLMM1}), the map $\mathfrak{g}_\infty$ is a global conjugacy between the two Schwarz reflections. Moreover, by Weyl's lemma, $\mathfrak{g}_\infty$ is conformal on $\widehat{\C}$ and hence a M{\"o}bius map. 
Finally, since the M{\"o}bius conjugacy $\mathfrak{g}_\infty$ fixes the critical values $y_f, \infty$, and the cusp $y_c$, it is the identity map. We conclude that $a_1=a_2$.
\end{proof}

\begin{remark}\label{bdry_psi_rem}
The proof of properness of $\pmb{\Psi}$ in Theorem~\ref{escape_unif_thm} yields the following finer information on the boundary behavior of $\pmb{\Psi}$. As a sequence $a_n\to \overline{\Gamma^{\perp}}$, the image $\pmb{\Psi}(a_n)\to \partial\mathcal{Q}_1$.  As $a_n\to \Gamma^{\mathrm{dp}}$, the image $\pmb{\Psi}(a_n)\to 1\in\partial\mathcal{Q}$. In fact, there are two accesses to $1$ in $\Int{(\mathcal{Q})}\setminus\mathcal{Q}_1$, and these two accesses correspond to the two components of $\Gamma^{\mathrm{dp}}$; i.e., $\pmb{\Psi}(a_n)$ converges to $1$ through one of these two accesses depending on which component of $\Gamma^{\mathrm{dp}}$ the sequence $a_n$ accumulates on. Finally,
if $a_n\to \cC(S_\cT)$ through an infinite sequence of tiles of increasing rank, then the image $\pmb{\Psi}(a_n)\to \partial\mathcal{Q}$.
\end{remark}

\begin{remark}
A dynamically natural uniformization for the exterior of the Tricorn (the connectedness locus of quadratic anti-polynomials) was given by Nakane \cite{Nak93}. In parameter spaces of antiholomorphic maps, such uniformizing maps given by the `conformal position of the free critical value' are not holomorphic. This makes studying the regularity of the candidate uniformizing map subtler than in the holomorphic case. In the case of the Tricorn, Nakane showed that the candidate uniformizing map is a proper covering map, where the covering property was established using the fact that the exterior of the Tricorn is a single quasiconformal deformation class. For the escape locus of our family $S_\cT$, there is an additional difficulty: the escape locus is not a single quasiconformal deformation class (indeed, the depth of a parameter in the escape locus is a topological conjugacy invariant). This forces us to first prove injectivity of the map $\pmb{\Psi}$ using a pullback argument and then appeal to the Invariance of Domain Theorem in the proof of Theorem~\ref{escape_unif_thm}.   
\end{remark}

\subsection{Parameter rays and their accumulation properties}\label{para_rays_subsec}

Recall that under the natural projection map $\overline{\D}\to\mathcal{Q}=\overline{\D}/\langle M_\omega\rangle$ (where $M_\omega(z)=\omega z$, $\omega=e^{2\pi i/(d+1)}$), the $\pmb{\Gamma}_d-$rays in $\D$ with angles in $[0,\frac{1}{d+1})$ descend to rays in $\mathcal{Q}$ with angles in $[0,1)$. Further, the image of (the tail of) such a ray at angle $\theta$ under $\mathcal{R}_d$ is (the tail of) a ray at angle $\mathcal{R}_d(\theta)$ (see Definition~\ref{Gd_rays}). Abusing notation, we call such rays $\pmb{\Gamma}_d-$rays in $\mathcal{Q}$. 

Recall also that in general there may be more than one $\pmb{\Gamma}_d-$rays in $\D$ at a given angle $\theta$, and these rays are homotopic in $\D$ relative to their endpoints. The same is true for two $\pmb{\Gamma}_d-$rays in $\mathcal{Q}$ at a given angle.

\begin{definition}[Parameter Rays of $S_{\mathcal{T}}$]\label{para_ray_schwarz}
The preimage of a $\pmb{\Gamma}_d-$ray at angle $\theta$ in $\mathcal{Q}$ under the map $\pmb{\Psi}$ is called a $\theta-$parameter ray of $S_\cT$, and is denoted by $\mathfrak{R}_\theta$.
\end{definition}

\begin{lemma}\label{ray_acu_pt_lem}
Every accumulation point of a parameter ray of $S_\cT$ lies in $\cC(S_\cT)$.    
\end{lemma}
\begin{proof}
Note that if $a\in\overline{\Gamma^{\perp}}\cup\Gamma^{\mathrm{dp}}$, then the free critical $y_f$ value escapes to the fundamental tile $T^0(\sigma_a)$ in finite time. This property is stable under small perturbation; i.e., parameters $a'$ close to $a$ lie in the escape locus and have bounded depth. If $a$ is an accumulation point of a parameter ray, then there are parameters $a'$ close to $a$ on this parameter ray such that $\pmb{\Psi}(a')$ lies in a tile of arbitrarily large rank, or in other words, $a'$ has arbitrarily large depth. This is a contradiction, which proves that no parameter $a\in\overline{\Gamma^{\perp}}\cup\Gamma^{\mathrm{dp}}$ can be  an accumulation point of a parameter ray. Hence, all accumulation points of parameter rays of $S_\cT$ lie in $\cC(S_\cT)$.  
\end{proof}

In light of the discussion preceding Definition~\ref{para_ray_schwarz}, there may be more than one parameter rays in $S_\cT$ at a given angle. When these rays land on $\cC(S_\cT)$, they define the same access to a point in $\cC(S_\cT)$, with a single exception that we now describe.
 
There are $\pmb{\Gamma}_d-$rays at angles $0, 1/(d+1)$ in $\D$ such that these rays are contained in the sector $Q$ of angle $2\pi/(d+1)$ (see \S~\ref{antiFarey_subsec}). Under the projection map $\overline{\D}\to\mathcal{Q}$, these rays descend to $\pmb{\Gamma}_d-$rays in $\mathcal{Q}$ at angle $0$. These two rays are homotopic in $\mathcal{Q}$, but not in $\mathcal{Q}\setminus\mathcal{Q}_1$. We will shortly see that the corresponding parameter rays, which we denote by the $0/1-$ray and $1/1-$ray, land on $\cC(S_\cT)$ but do not define the same access to $\cC(S_\cT)$.

We start with accumulation properties of the parameter rays of $S_\cT$ at strictly pre-periodic angles (under $\cR_d$). 

\begin{proposition}\label{preper_para_rays_prop}
Let $\theta\in\R/\Z\setminus\{0\}$ be strictly pre-periodic under $\mathcal{R}_d$. Then, the $\theta-$parameter ray of $S_\cT$ lands at a parameter $a\in\cC(S_\cT)$ such that the free critical value $y_f$ of $\sigma_a$ is strictly pre-periodic. Furthermore, the corresponding dynamical ray at angle $\theta$ lands at $y_f$ in the dynamical plane of $\sigma_a$.
\end{proposition}
\begin{proof}
Let $a_0$ be an accumulation point of $\mathfrak R_\theta$. For a Schwarz reflection $\sigma_a$, we denote the dynamical ray at angle $\theta$ by $R_{a}^\theta$. Also recall, that by definition, there exist $a$ close to $a_0$ on $\mathfrak{R}_\theta$ for which the free critical value $y_f$ of $\sigma_{a}$ lies on $R_{a}^\theta$.

Note that $R_{a_0}^\theta$ is strictly pre-periodic and thus lands at a pre-periodic point, which we denote by $z_0$. The point $z_0$ eventually falls on a repelling or a parabolic periodic point.

First, we rule out the possibility that $z_0$ eventually hits a parabolic periodic point, with the possible exception of the cusp $y_c$ on the quadrature domain. To see this, we assume by way of contradiction that $\sigma_{a_0}$ has a parabolic cycle. Evidently, the strictly pre-periodic point $z_0$ is not the parabolic periodic point on the boundary of the Fatou component of $\sigma_{a_0}$ containing the critical value $y_f$. Since $\sigma_{a_0}$ has a parabolic cycle by assumption, it has a locally connected limit set and repelling periodic points are dense on the limit set. Hence, there exists a cut-line consisting of dynamical rays through repelling periodic points on $\Lambda(\sigma_{a_0})$ separating $R_{a_0}^\theta$ from the critical value $y_f$. As such cut-lines are stable under small perturbation, it follows that for parameters $a\approx a_0$, the dynamical ray $R_a^\theta$ remains disjoint from the critical value $y_f$. This is a contradiction, and hence $\sigma_{a_0}$ does not have a parabolic cycle. In particular, $z_0$ is either a repelling pre-periodic point or a preimage of the cusp $y_c$.

If $z_0$ is neither pre-critical nor equal to the free critical value $y_f$, then the ray $R^\theta_{a}$ will still land at a repelling pre-periodic point for all $a$ near $a_0$ and will stay away from $y_f$ (cf. \cite[Expos\'e~VIII, \S 2.3, Proposition~3]{DH1}), giving rise to a contradiction.

This means that $z_0$ is either equal to $y_f$ or pre-critical, and hence $\sigma_{a_0}$ is critically pre-periodic and hence locally connected. In fact, a similar cut-line argument as above implies that if $z_0$ is not equal to the free critical value $y_f$ (cf. \cite[Lemma~8.14]{LLMM3}), then we may separate the dynamical $\theta-$ray from $y_f$ for all nearby parameters, giving rise to a contradiction. Hence $z_0=y_f$. 
Since $\theta$ is strictly pre-periodic under $\mathcal{R}_d$, it follows that $a_0$ is a \emph{Misiurewicz} parameter; i.e., all critical points of $\sigma_{a_0}$ are strictly pre-periodic.

Lastly we show that $a_0$ is the unique accumulation point of $\mathfrak{R}_\theta$. By the preceding discussion, any accumulation point must be a critically pre-periodic parameter with pre-period and period defined by $\theta$. We claim that at most finitely many parameters satisfy this property for the given pre-period and period. Indeed, there are at most finitely many Hubbard trees with such combinatorics. The claim now follows by combinatorial rigidity of post-critically finite maps. Finally, as the accumulation set of a ray must be connected, we have that $a_0$ is the only accumulation point of~$\mathfrak{R}_\theta$.
\end{proof}

Now, we note one way in which our setup departs from the standard situation for unicritical maps.

For a unicritical (anti-)polynomial $f\colon \C\to \C$ if the critical value $c$ is periodic of period $p$ then one has $c=f^{\circ p}(c) = f(f^{\circ (p-1)}(c))$, and as $c$ has a unique preimage, namely the critical point, it follows that $f^{\circ (p-1)}(c)$ is the critical point, and hence is a superattracting periodic point of period $p$. In our setup however, one can have the free critical value lie in a repelling periodic orbit so long as the dessin $\cT$ has at least one white vertex of valence $1$. If $\sigma_a$ is a map for which such a white vertex is periodic, then the critical value will be periodic as well, but all critical \textit{points} will be strictly pre-periodic, and hence the limit set $\Lambda(\sigma_a)$ will have empty interior, and all periodic points will be repelling.

The above observation leads to an important distinction between the behavior of parameter rays at periodic angles in the classical setting and in our family. If a unicritical (anti-)polynomial lies on a parameter ray at a periodic angle $\theta$, or equivalently, if the critical value of a unicritical (anti-)polynomial lies on a periodic dynamical ray at angle $\theta$, then the corresponding dynamical ray bifurcates (cf. \cite[Expos\'e~VIII, \S 1, Proposition~1]{DH1}). This is essentially due to the fact that any preimage of such a dynamical ray hits the unique critical point. However, this is not true in our family. Indeed, if the free critical value $y_f$ lies in a repelling cycle of $\sigma_{a_0}$, then a periodic dynamical ray $R_{a_0}^\theta$ lands at $y_f$, and one can choose nearby parameters $a$ so that the dynamical ray $R_{a}^\theta$ contains the critical value $y_f$ and lands at the analytically continued repelling periodic point.

We will now see that as a consequence of the above differences, parameter rays at periodic angles can land at critically pre-periodic (i.e., Misiurewicz) parameters in the space $S_\cT$, while such parameter rays necessarily accumulate on parabolic parameters in unicritical parameter spaces (see Figure~\ref{per_para_ray_misi_fig}, cf. \cite{DH1,EMS16,IM16}).

\begin{figure}[ht!]
\captionsetup{width=0.98\linewidth}
\begin{tikzpicture}
\node[anchor=south west,inner sep=0] at (0,0) {\includegraphics[width=0.96\textwidth]{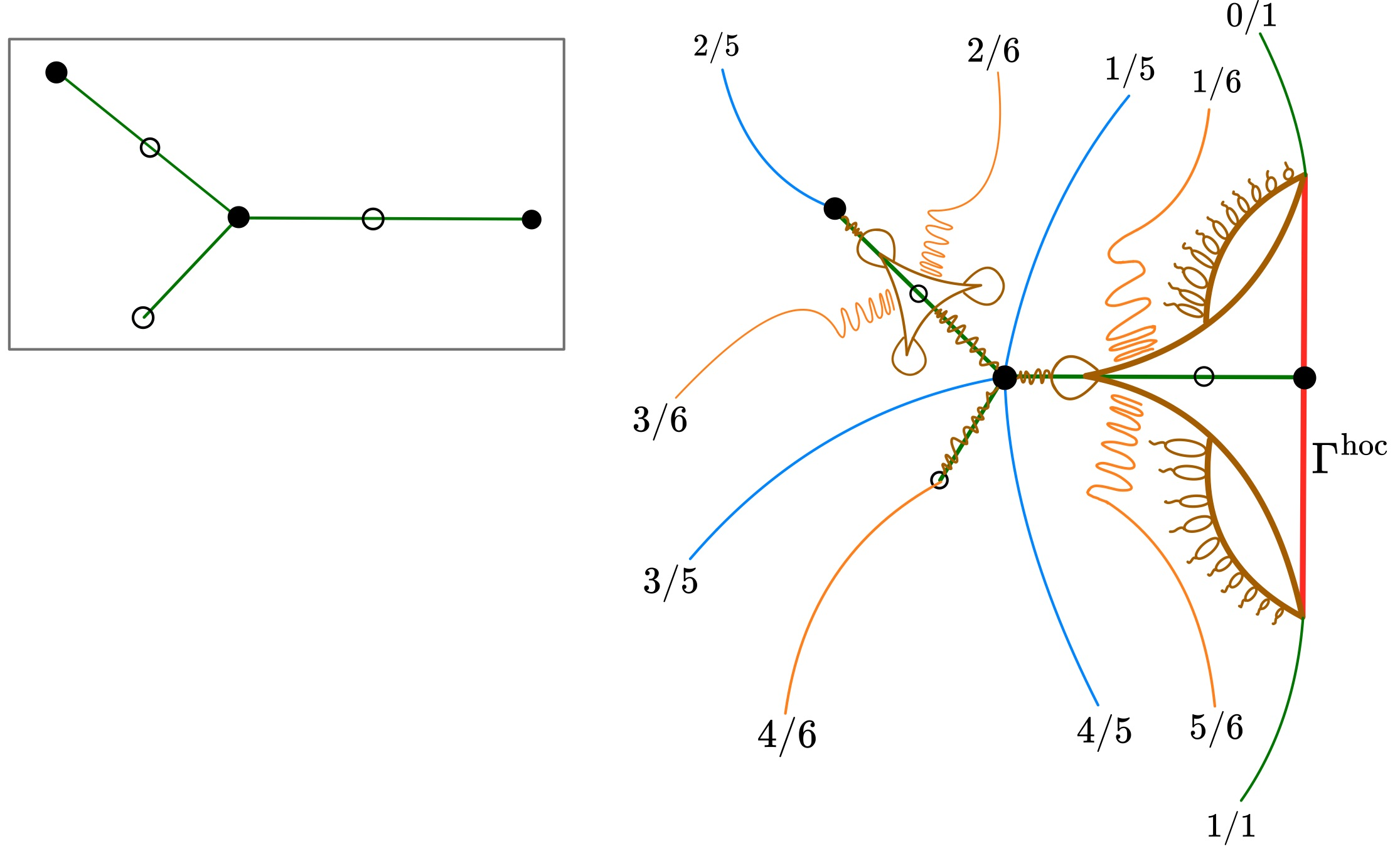}};
\end{tikzpicture}
\caption{Depicted is a cartoon of the connectedness locus $\cC(S_\cT)$ (on the right) for the dessin $\cT$ (in the left box). The blue (respectively, green) parameter rays land at the pre--cuspidal parameters (respectively, at the ends of $\Gamma^{\mathrm{hoc}}$). Four of the orange parameter rays (at $\cR_d$-fixed angles) accumulate on period $1$ parabolic arcs, while the fifth one lands at a Misiurewicz parameter. (The angle marked next to a ray $\mathfrak{R}_\theta$ is the image of $\theta$ under the circle homeomorphism conjugating $\cR_d$ to $\overline{z}^d$.)}
\label{per_para_ray_misi_fig}
\end{figure}

\begin{proposition}\label{per_para_rays_prop}
Let $\theta\in\R/\Z\setminus\{0\}$ be periodic under $\mathcal{R}_d$. Then one of the following three possibilities occurs.
\begin{enumerate}[leftmargin=8mm]
    \item $\mathfrak{R}_\theta$ lands at a parameter $a\in\cC(S_\cT)$ such that $\sigma_a$ has an even periodic parabolic cycle and the dynamical ray $R_a^\theta$ lands at the characteristic parabolic~point.
    \item $\mathfrak{R}_\theta$ accumulates on a parabolic arc lying on the boundary of an odd period hyperbolic component. For each parameter $a$ on this parabolic arc, the dynamical ray $R_a^\theta$ lands at the characteristic parabolic point.
    \item $\mathfrak{R}_\theta$ lands at a parameter $a\in \cC(S_\cT)$ such that the dynamical ray $R_a^\theta$ lands at the free critical value $y_f$, which is a repelling periodic point.
\end{enumerate}
(The characteristic parabolic point of a parabolic map is the unique parabolic periodic point on the boundary of the periodic Fatou component containing $y_f$.)
\end{proposition}
\begin{proof}[Sketch of proof]
Let $a_0$ be an accumulation point of $\mathfrak R_\theta$.  By definition of parameter rays, there exist $a\in\mathfrak{R}_\theta$ close to $a_0$ for which the free critical value $y_f$ of $\sigma_{a}$ lies on the periodic ray $R_{a}^\theta$.

The periodic ray $R_{a_0}^\theta$ lands at a repelling or parabolic periodic point $z_0$. The point $z_0$ is clearly not pre-critical; indeed, a critical point in a cycle is necessarily superattracting.

Suppose that $z_0$ is repelling. If $z_0\neq y_f$, then the ray $R^\theta_{a}$ will land at a repelling periodic point for all $a$ near $a_0$ and will stay away from $y_f$, giving rise to a contradiction. Thus, if $z_0$ is repelling, then $z_0=y_f$; i.e., $a_0$ is a Misiurewicz parameter for which $y_f$ is a repelling periodic point.

Now let $z_0$ is parabolic. We claim that $z_0$ is the characteristic parabolic point. Suppose by way of contradiction that this is not the case. Then arguments similar to the ones used in \cite[Lemma~3.7]{Sch00} imply that the characteristic parabolic point of $\sigma_{a_0}$ can be separated from $z_0$ by a pair of (pre-)periodic dynamical rays landing at a common repelling (pre-)periodic point. Since such cut-lines persist under small
perturbation and separate the free critical value $y_f$ from the dynamical ray at angle $\theta$, for parameters $a\approx a_0$, the free critical value $y_f$ stays away from $R_a^\theta$. Once again, this is a contradiction, which proves that $z_0$ is the characteristic parabolic point of~$\sigma_{a_0}$.

Finally, note that the set of Misiurewicz parameters of given pre-period and period as well as the set of even periodic parabolic parameters with a given ray period are finite. On the other hand, the set of odd periodic parabolic parameters of a given ray period is the union of finitely many parabolic arcs lying on boundaries of odd period hyperbolic components. Further, the closures of such parabolic arcs is disjoint from the set of Misiurewicz parameters and even periodic parabolic parameters mentioned above. As the accumulation set of $\mathfrak{R}_\theta$ is connected, the result now follows.
\end{proof}

Finally, we turn our attention to the $0/1$ and $1/1$ parameter rays.

\begin{proposition}\label{zero_rays_prop}
The $0/1-$ray and $1/1-$ray land at the two distinct endpoints of~$\Gamma^{\mathrm{hoc}}$.
\end{proposition}
\begin{proof}
Let $a$ be an accumulation point of the $0/1-$ray. The dynamical ray of $\sigma_a$ at angle $0$ lands at the cusp $y_c$, which is different from the critical value $y_f$ (this is true for every parameter in $\cC(S_\cT)$).
We claim that $a\in\Gamma^{\mathrm{hoc}}$. Let us assume by way of contradiction that this is not true. Then $\sigma$ has no attracting direction in $\Omega_a$ at $y_c$, and the free critical value $y_f$ can be separated from $R_a^0\cup\{y_c\}$ by a cut-line consisting of a pair of dynamical rays landing at a repelling (pre-)periodic point on $\Lambda(\sigma_a)$ (see \cite[Theorem~2, Corollary~2.18]{kiwi2}, cf. \cite[Definition~2.3]{Sch98}). (We remark that such a cut-line does not exist for parameters on $\Gamma^{\mathrm{hoc}}$ because for such parameters the forward orbit of $y_f$ converge non-trivially to $y_c$ through a Fatou component, see Lemma~\ref{cusp_anal_dyn_lem}.) This cut-line remains stable under small perturbation, and separates $y_f$ from the dynamical $0-$ray. But this contradicts the fact that there exist $a'$ close to $a_0$ (on the $0/1$ parameter ray) for which the free critical value $y_f$ lies on the dynamical ray $R_{a'}^0$. Thus, $a\in\Gamma^{\mathrm{hoc}}$. Since a parameter ray lies in the escape locus, it follows from Proposition~\ref{hoc_bdry_primary} that $a$ cannot lie in the interior of $\Gamma^{\mathrm{hoc}}$. Hence $a$ is one of the two endpoints of $\Gamma^{\mathrm{hoc}}$.

The same argument, applied to the $1/1-$parameter ray, implies that this ray also lands at one of the two endpoints of $\Gamma^{\mathrm{hoc}}$. 

It is now easily seen that the connectedness locus $\cC(S_\cT)$ lies between the above two landing parameter rays, and hence, the landing points of the $0/1$ and $1/1$ parameter rays must be distinct.
\end{proof}

\section{Puzzles, renormalizability, and combinatorial rigidity}\label{renorm_sec}

In this and the following section, we investigate  parameters which do \textit{not} lie on the boundaries of hyperbolic components. They naturally divide themselves into two subcategories - those parameters which are at most finitely renormalizable, and those which are infinitely renormalizable.

\begin{definition}\label{renormalizability_def}
Let $\sigma_a\in \cC(S_\cT)$ and $w_i$ be a free critical point of $\sigma_a$ (i.e., $w_i$ is some white vertex of $\cT(\sigma_a)$ with valence greater than one). We say that $\sigma_a$    
    is \emph{renormalizable at $w_i$} if there exist Jordan disks $U, V$ with $w_i\in U\Subset V$, such that $\sigma^{\circ p}\colon U\to V$ (for some $p\geq 1$) is a unicritical polynomial-like or anti-polynomial-like map (of degree at least two) with connected non-escaping set. The critical point $w_i$ will be called the \emph{renormalizable} critical point for the associated renormalization.

A map $\sigma_a$ is said to be pinched renormalizable at $w_i$ if the above condition holds for Jordan disks $U, V$ with $w_i\in U\subset V$ where $\overline{U}\cap\overline{V}$ is a singleton and $\sigma^{\circ p}\colon U\to V$ is a unicritical pinched (anti-)polynomial-like map with connected non-escaping set. 
\end{definition}
We note that a map $\sigma_a$ has at most one renormalizable critical point.

\subsection{Period 1 renormalizations}

Before describing the puzzle structure associated with renormalizable Schwarz reflection maps in $\cC(S_\cT)$, let us first mention a case with particularly simple renormalization combinatorics.

For a fixed map $\sigma_a\in \cC(S_\cT)$, denote by $\Omega^1_{1,a},\Omega^1_{2,a},\cdots,\Omega^1_{k,a}$ the preimages of $\Omega_a$ under $\sigma_a$. As a convention we will always choose $\Omega^1_{1,a}$ to be the component which contains the root $y_c$ on its boundary. It is easily seen from the bicolored structure of $\cT(\sigma_a)$ that each $\Omega^1_{j,a}$ contains a unique white vertex $w_i$.

Now suppose that for some $\sigma_a$, the critical orbit $\{\sigma_a^{\circ n}(w_i)\}_{n\geq 0}$ lies within a single component $\Omega^1_{j,a}$, where $w_i$ is a critical point of $\sigma_a$. If $j\neq 1$ then $\sigma_a\colon \Omega^1_{j,a}\to \Omega_a$ is a unicritical anti-polynomial-like map of degree $\mathrm{val}(w_i)$, with connected non-escaping set. When $j=1$, the restriction $\sigma_a\colon \overline{\Omega^1_{1,a}}\to \overline{\Omega_a}$ is a unicritical pinched anti-polynomial-like map of degree $\mathrm{val}(w_1)$. 
Further, if such a parameter $a$ does not lie in $\Gamma^\mathrm{hoc}$, then there is a restriction of $\sigma_a$ to a subset of $\overline{\Omega^1_{1,a}}$ which is a simple pinched anti-polynomial-like mapping (see Theorem~\ref{simp_restrict_thm}).

\subsection{Finitely renormalizable maps and rigidity}\label{fin_renorm_rigid_subsec}
For the remainder of this section we will be primarily concerned with those parameters which are at most finitely renormalizable. The following theorem appears in \cite{LLM24}, based on the main result of \cite{CDKvS}.

\begin{theorem}\cite[Theorem~10.1]{LLM24}\label{fin_renorm_geom}
    Let $\sigma_a \in \cC(S_\cT)$ be periodically repelling (i.e., all of its periodic points other than $y_c$ are repelling) and at most finitely renormalizable. Then
    \begin{enumerate}[leftmargin=8mm]
        \item the limit set $\Lambda(\sigma_a)$ is locally connected;
        \item $\sigma_a$ has no invariant line fields on $\Lambda(\sigma_a)$;
        \item if $\sigma_{a'}\in \cC(S_\cT)$ is periodically repelling and has the same rational lamination as $\sigma_a$, then $\sigma_a=\sigma_{a'}$.
    \end{enumerate}
\end{theorem} 

Here a \textit{rational lamination} is an equivalence relation on $\cR_d-$pre-periodic angles in $\mathbb{S}^1\cong\R/\Z$ where one identifies two angles if the corresponding dynamical rays of $\sigma_a$ land at the same point of $\Lambda(\sigma_a)$.

\begin{remark}
    The above theorem, as written in \cite{LLM24}, applies to a slightly different class of Schwarz reflections, namely those with the \textit{Nielsen} external map as opposed to the anti-Farey external map. However the proof of the above theorem applies mutatis mutandis to the setting considered here (cf. \cite[\S 12]{LLM24}).
\end{remark}

In the remainder of the section, we will give a description of what it means for two periodically repelling, non-renormalizable Schwarz reflections in $\cC(S_\cT)$ to share a rational lamination in terms of the dessin and its iterated preimages (see \cite[Lemma~8.5]{LLM24} for a related description).

\subsection{Puzzle pieces and rational laminations}

We now describe a \textit{puzzle piece} structure associated to Schwarz reflections in our family.

\begin{figure}
\captionsetup{width=0.98\linewidth}	
\begin{tikzpicture}
    \node[anchor=south west,inner sep=0] at (0,1) {\includegraphics[width=0.4\textwidth]{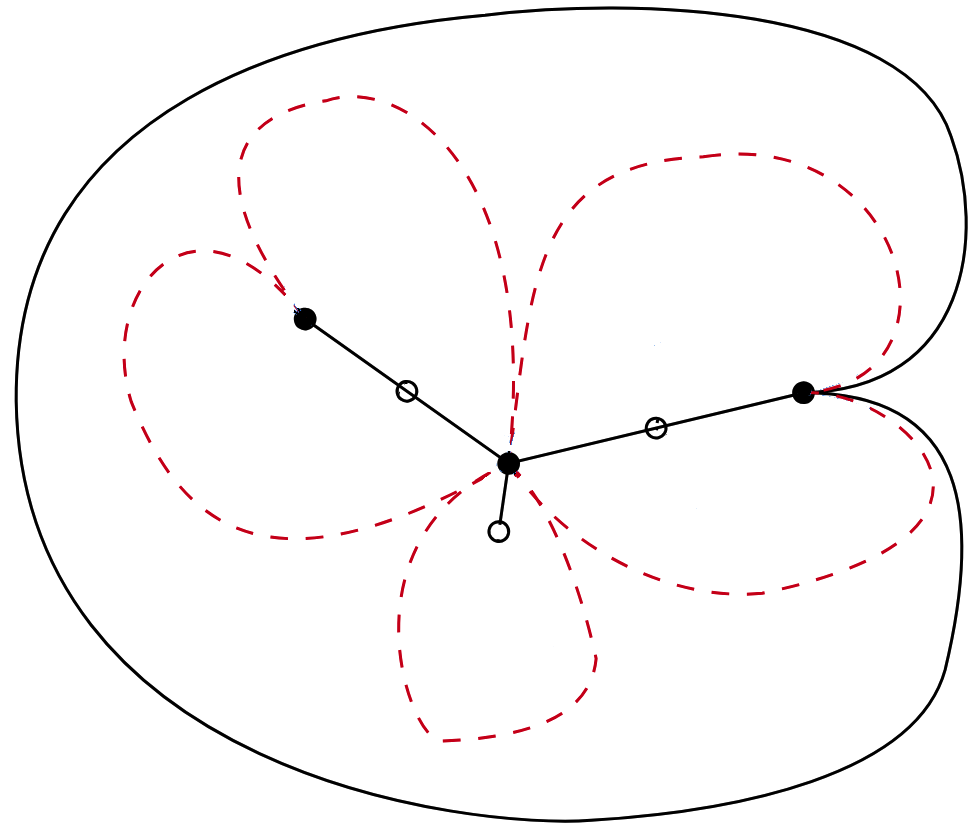}};
    \node[anchor=south west,inner sep=0] at (5.4,0) {\includegraphics[width=0.56\textwidth]{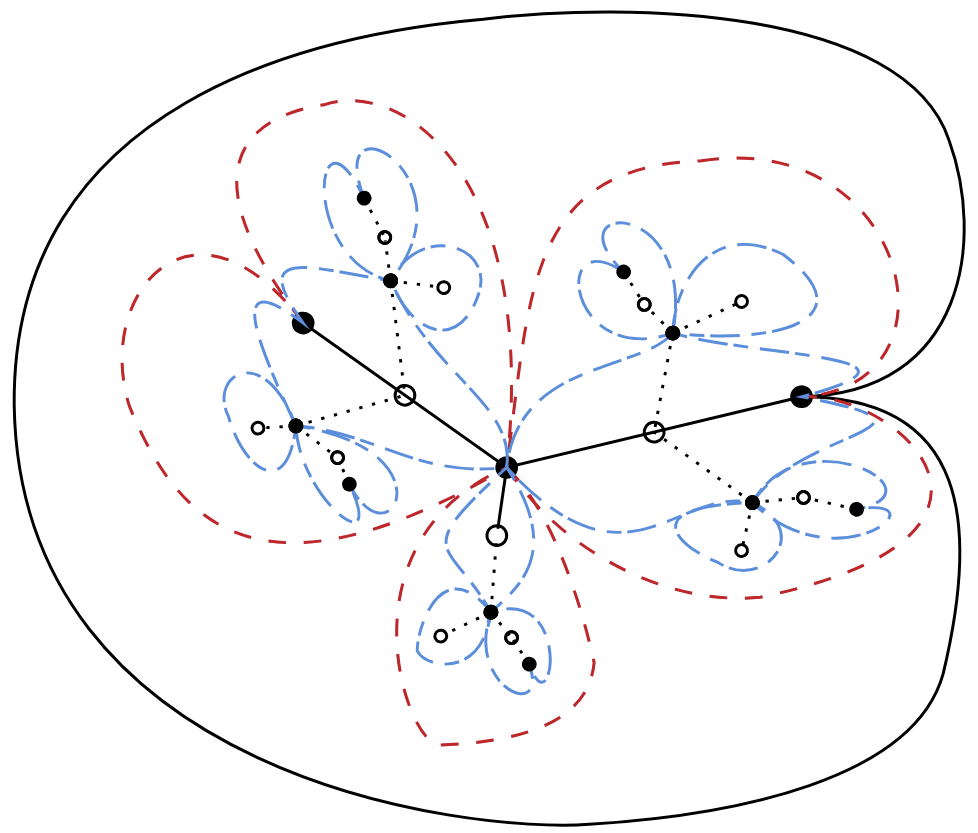}};
\end{tikzpicture}
\caption{The depth 0, 1, and 2 puzzle pieces for a Schwarz reflection map. The depth 2 puzzle pieces are determined under the assumption that the critical value lies in the depth 1 puzzle piece whose boundary contains the cusp.}
    \label{puzzle_fig}
\end{figure}

\begin{definition}\label{puzzle_dessin_depth_k_def}
	For a map $\sigma_a\in \cC(S_\cT)$, the \textit{puzzle pieces of depth} $k\geq 0$ are the closures of the components of $\Omega_a^k:=\sigma_a^{-k}(\Omega_a)$.
    
    The \textit{depth $k$ dessin} $\cT^k(a)$ is an augmentation of the contact graph for the depth $k$ puzzle pieces, defined to be the combinatorial, bi-colored, plane graph with a white vertex $w^k_i$ for every component of $\Omega_a^k$, a black vertex $b^k_j$ for every $k-$th preimage of the cusp $y_c$, and an edge connecting $b_j^k$ to $w_i^k$ if $b_j^k$ lies on the boundary of the component corresponding to $w_i^k$. (See Figure~\ref{puzzle_fig}.)
\end{definition}

We will call a puzzle piece \emph{valuable} if it contains the free critical value $y_f$, and \emph{critical} if it contains a free critical point. 
The map $\sigma_a$ induces a branched covering $\widehat{\sigma}_a:\cT^k(a)\to\cT^{k-1}(a)$; where the sub-graph of $\cT^k(a)$ contained in a depth $k$ critical puzzle piece is mapped as a unicritical branched covering onto the sub-graph of $\cT^{k-1}(a)$ in the valuable depth $(k-1)$ puzzle piece (where the white vertex corresponding to this depth $k$ critical puzzle piece is the branch point), while the sub-graph of $\cT^k(a)$ in a non-critical depth $k$ puzzle piece maps homeomorphically. Thus, the graph $\cT^{k}(a)$ can be regarded as the pullback of $\cT^{k-1}(a)$ under the induced branched covering map~$\widehat{\sigma}_a$.

It is easy to see from the tiling structure of the escaping set that any puzzle piece of depth $k$ is contained in exactly one puzzle piece of depth $k-1$.

\begin{remark}
    If $\Lambda(\sigma_a)$ is locally connected then there is a combinatorial copy of $\cT^1(a)=\cT(a)$ embedded in the the non-escaping set $K(\sigma_a)$. In the non-locally connected case, Proposition~\ref{partition_prop} gives us a copy of $\cT^1(a)$ (in $\C$) to which $K(\sigma_a)$ has a homotopic retraction rel the vertices. We may use this observation to give an alternative description of the depth $k$ dessins.
\end{remark}

    \begin{proposition}
        For $k\geq 1$, the graph  $\cT^k(a)$ is a tree. 
    \end{proposition}

    \begin{proof}
         The combinatorial embedding of $\cT^1(a)=\cT$ in $K(\sigma_a)$ as described in Proposition~\ref{partition_prop} is given by the contact graph of the depth $1$ puzzle pieces, augmented to include preimages of the cusp $y_c$. Further depth puzzle pieces are determined by pulling back these under $\sigma_a$, and hence the augmented contact graphs are pulled back under $\widehat{\sigma}_a$. By definition of depth $k$ dessins, the sub-graph of $\cT^k(a)$ in each depth $k$ puzzle piece is a star-tree with the corresponding white vertex as the base-point and black vertices as tips. The graph $\cT^k(a)$ is the union of these (star-tree) sub-graphs touching only at black vertices. Hence, there is a retraction from each depth $k$ puzzle piece to the part of $\cT^k(a)$ contained in it, and as the union of depth $k$ puzzle pieces is simply connected (it is the complement of the interior of the union of tiles of rank at most $k$, which is a topological disk; see Proposition~\ref{connected_critical}), it follows that $\cT^k(a)$ is connected and has no loops.
    \end{proof}

    Define the \textit{pre--cuspidal} lamination for a map $\sigma_a\in S_\cT$ to be the restriction of the rational lamination to those angles whose associated dynamical rays land at preimages of the cusp, that is the restriction to $\bigcup_{n\geq 0} \cR_d^{-n}(\{0\})$, where $\cR_d$ is the anti-Farey map described in \S~\ref{antiFarey_subsec}. The pre--cuspidal lamination is naturally stratified by depth, and we will say that the pre--cuspidal lamination of depth $k$ is the restriction of the rational lamination to the angles in $\bigcup_{n=0}^k \cR_d^{-n}(\{0\})$.
    
    Yet another description of the depth $k$ dessins is as the tree dual to the depth $k$ pre--cuspidal lamination. Indeed, each region cut out by the pre--cuspidal rays corresponds to a white vertex, and each landing point of the rays corresponds to a black vertex.

\begin{definition}
	We define the \textit{depth $k$ kneading sequence} $\{\vartheta^k_n\}$ to be the itinerary of the critical orbit $\{\sigma_a^{\circ n}(y_f)\}_{n\geq 0}$ in the depth $k$ puzzles.
\end{definition}

\begin{proposition}\label{dessins_from_kneading_seqn_prop}
	The depth $k$ dessin and the depth $k$ kneading sequence uniquely determine the depth $k+1$ dessin. That is, for two parameters $a_1, a_2\in\cC(S_\cT)$, if the depth $k$ dessins $\cT^k(a_1)$ and $\cT^k(a_2)$ are isomorphic as bi-colored planar embedded trees and $\vartheta_1^k(a_1)=\vartheta_1^k(a_2)$, then $\cT^{k+1}(a_1)$ is isomorphic to $\cT^{k+1}(a_2)$ as bi-colored planar embedded trees.
\end{proposition}

\begin{proof}
    Let us begin with the following elementary topological fact. Let $\mathcal Y$ be a planar embedded tree, and suppose that $\mathcal X_1,\mathcal X_2$ are both planar embedded trees such that there exist maps
    $$
    f_i\colon \mathcal X_i\to \mathcal Y,\quad i=1,2,
    $$
    which are both orientation reversing, unicritical (i.e., having exactly one critical point) branched covering maps of the same degree $d\geq 1$. Further assume that the critical values of $f_1$ and $f_2$ (when $d>1$) are the same vertex of $\mathcal Y$. Then  $\mathcal X_1$ and $\mathcal X_2$ are isomorphic as planar embedded trees.

    To prove this, note that the critical value $v$ partitions $\mathcal Y$ into $k\geq 1$ subtrees, each having $v$ as a tip point. We label these subtrees as $\mathcal Y_1,\mathcal Y_2,\cdots \mathcal Y_k$, in counterclockwise order. The trees $\mathcal X_1,\mathcal X_2$ will then be given by attaching $d$ copies of the mirror-images of these subtrees in the reverse order. Specifically, attach $\mathcal Y_k^\mathrm{op},\mathcal Y_{k-1}^\mathrm{op},\cdots,\mathcal Y_1^\mathrm{op}$ to a common vertex in a counterclockwise order, repeating this process $d$ times. This uniquely constructs the pullback under a unicritical branched covering. (In the case of $d=1$ we have $\mathcal{X}_1\cong \mathcal{X}_2\cong \mathcal{Y}^\mathrm{op}$.)

    Now to prove the proposition, the depth 1 pre--cuspidal points of $\cT^{k+1}(a_i)$ (i.e., the black points of $\cT$) partition it into a collection of subtrees, each of which is a unicritical branched cover of $\cT^k(a_i)$, and each degree depending only on the depth $1$ dessin $\cT$. By the above fact, the corresponding subtrees of $\cT^{k+1}(a_1)$ and $\cT^{k+1}(a_2)$ are isomorphic, and since the subtrees are joined together as described by the same underlying dessin $\cT$, it follows that $\cT^{k+1}(a_1)\cong \cT^{k+1}(a_2)$.
\end{proof} 
%    {\st{The tree $\cT^{k+1}(a_i)$ is given by pulling back $\cT^k(a_i)$ under the branched covering $\widehat{\sigma}_{a_i}$, where the branching data is given by $\vartheta_n^k(a_i)$, $i\in\{1,2\}$.}}

\begin{question} 
What are the admissible kneading sequences and dessins?
\end{question}

We a priori do not know what the allowable kneading sequences and depth $k$ dessins are. However, given a realization of a kneading sequence by \textit{some} parameter, we may realize these same combinatorics up to an arbitrary depth with a PCF parameter as the following proposition shows.

As discussed above, the map $\sigma_a$ induces a branched covering map $\widehat{\sigma}_a:\cT^k(a)\to \cT^{k-1}(a)$ for $k\geq 2$. One may also  construct natural `homotopic' embeddings $\cT^{k-1}(a)\hookrightarrow\cT^k(a)$, identifying the black vertices of $\cT^{k-1}(a)$ with the corresponding black vertices in $\cT^k(a)$ (since $y_c$ is a fixed point, the black vertices of $\cT^{k-1}(a)$ are also black vertices of~$\cT^k(a)$) and mapping white vertices to appropriate white vertices (see Figure~\ref{dessin_map_fig}). The composition of these two maps now gives rise to branched covering dynamics on (angled) trees, to which we may apply Poirier's theorem \cite{Poi13} (c.f. \S~\ref{dynamical_plane_sec}).
\begin{figure}[ht!]
\captionsetup{width=0.98\linewidth}	
\begin{tikzpicture}
    \node[anchor=south west,inner sep=0] at (0,1) {\includegraphics[width=0.4\textwidth]{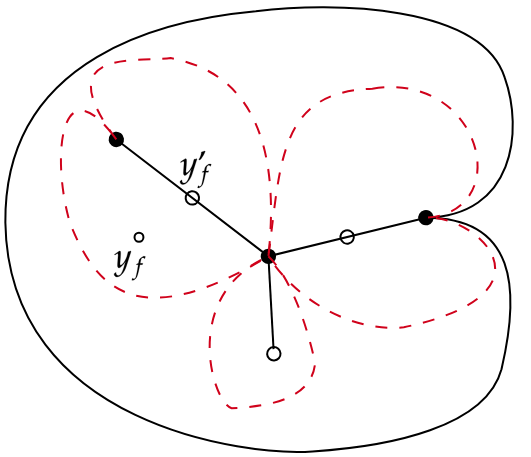}};
    \node[anchor=south west,inner sep=0] at (5.4,0) {\includegraphics[width=0.56\textwidth]{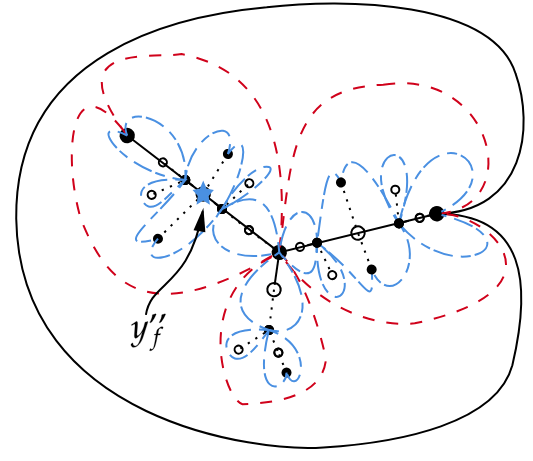}};
\end{tikzpicture}
    \caption{The depth 1 and depth 2 puzzle pieces and dessins for a given combinatorics. The branched covering map $\widehat{\sigma}_a$ sends $y_f''$ to $y_f'$ with local degree $2$, and respects the color of vertices. The natural embedding  $\cT^1\hookrightarrow\cT^2$ maps $y'_f$ to $y''_f$. The induced tree map $s:\cT^2\to\cT^2$ from Proposition~\ref{knead_pcf_prop} fixes the critical point~$y_f''$.}
    \label{dessin_map_fig}
\end{figure}

\begin{proposition}\label{knead_pcf_prop}
    Let $\sigma_a\in \cC(S_\cT)$, and let $s\colon \cT^k(a)\to \cT^{k-1}(a)\hookrightarrow \cT^k(a)$ be the composition of the branched covering $\sigma_a$ and the natural embedding.

    \begin{enumerate}[leftmargin=8mm]
        \item There exists a PCF map $\sigma_{a'}\in \cC(S_\cT)$ realizing the dynamics of $s$ on the augmented Hubbard tree $\cH(\sigma_{a'})$.
        \item For the depth $k$ dessins we have $\cT^{k}(a)=\cT^{k}(a')$, and we also have agreement of the kneading sequences $\vartheta_n^k(a)=\vartheta^k_n(a')$ for $n=0,1,\cdots,k-1$.
        \item If the depth $k$ kneading sequence $\{\vartheta^k_n(a)\}$ is periodic of period $p$, then $\sigma_{a'}$ can be chosen to be hyperbolic with the free critical value having period~$p$.
    \end{enumerate}
\end{proposition}

\begin{proof}
    (1) As noted above, the PCF map is constructed by using Poirier's realization theorem to realize the tree dynamics $s: \cT^k(a)\to\cT^k(a)$ as the augmented Hubbard tree dynamics of a PCF \textit{polynomial} and then applying \cite[Theorem C]{PartI} to obtian a PCF Schwarz reflection. To apply Poirier's result we must argue that the tree dynamic is \textit{expanding}, that is, any two adjacent Julia vertices are eventually non-adjacent. Note that all of the white vertices of $\cT^k(a)$ will eventually map under $s$ onto $y_f$ and all of the black vertices will eventually map onto $y_c$.
    There are two possible cases: in the first, the critical orbit of $\sigma_a$ remains in $\Omega^1_{1,a}$, the depth $1$ puzzle piece with the cusp $y_c$ on the boundary. In this case $y_f$ is fixed, and therefore all white vertices in $\cT^k$ are of Fatou type under $s$.
    In the second case the free critical value eventually leaves $\Omega^1_{1,a}$. But then any two adjacent vertices will eventually be mapped to $y_c$ and $y_f$, which are not adjacent.
    
    (2) This statement follows immediately by taking appropriate pullbacks.
    
    (3) The valuable white vertex of $\cT^k$ (i.e., the white vertex in the puzzle piece containing the free critical value $y_f$) is periodic in this setting, of period  $p$.
\end{proof}

Note that since the depth $1$ dessin is the original tree $\cT$, Proposition~\ref{dessins_from_kneading_seqn_prop} can be restated as follows: dessins of all depth can be recovered from the knowledge of kneading sequences of all depth. We will now show that the rational lamination of a non-renormalizable map in $\cC(S_\cT)$ can be recovered from the kneading sequences. Here, non-renormalizable means that $\sigma_a$ is neither renormalizable nor pinched renormalizable in the sense of Definition~\ref{renormalizability_def}.

\begin{proposition}\label{rat_lam_prop}
    Suppose that $\sigma_{a_1},\sigma_{a_2}$ are non-renormalizable maps with the same kneading sequences; i.e., $\vartheta_n^k(a_1)=\vartheta_n^k(a_2)$ for all $n,k$. Then $\sigma_{a_1}$ and $\sigma_{a_2}$ have the same rational lamination.
\end{proposition}

In fact we will show something stronger. As $\cT^k(a)$ is the dual tree to the depth $k$ pre--cuspidal lamination, if two parameters share the same kneading sequence for all depths, then they share the same pre--cuspidal lamination.

\begin{proposition}\label{rat_lam_density_prop}
    For a non-renormalizable map, the pre--cuspidal lamination is dense in the rational lamination.
\end{proposition}
\begin{proof}
Let $z\in \Lambda(\sigma)$ be a periodic point of period $q$. Note that neutral periodic points can only occur for parameters on the boundaries of hyperbolic components and such parameters are (pinched) renormalizable. Hence $z$ must be a repelling periodic point. Denote the angles of the (periodic) rays landing at $z$ by $\theta_1,\cdots, \theta_n$, and label them so that they are oriented counter-clockwise.
As $\sigma$ is non-renormalizalbe it follows that there is some $k$ sufficiently large so that $V^k_j(z)$ (where $V^k_j(z)$ is the depth $k$ puzzle piece containing $z$) does not contain any critical points of $\sigma$. In particular for all $k$ sufficiently large there is a well-defined inverse branch $\sigma^{-q}\colon V^k_j(z)\to V^{k+q}_{j'}(z)$ which is conformal and contracting towards the periodic point $z$. As $\Lambda(\sigma)$ is locally connected by Theorem~\ref{fin_renorm_geom}, we may also choose $k$ large enough so that $\left(\Lambda(\sigma)\setminus\{z\}\right)\cap V^k_j(z)$ has exactly $n$ connected components. Each such connected component will intersect $\partial V^k_j(z)$ in exactly one point, and such a point of intersection is a $k-$th preimage of the cusp.

There are $2n$ angles in $\cR_d^{-k}(\{0\})$ which cut out $\Lambda(\sigma)\cap V^k_j(z)$. We denote these angles counter-clockwise as $\phi_{k,1}^-,\phi_{k,1}^+,\cdots,\phi_{k,n}^{-},\phi_{k,n}^{+}$, labeled so that $\theta_i$ is immediately between $\phi_{k,i}^{-}$ and $\phi_{k,i}^{+}$, and hence $\phi_{k,i}^{+}$ has the same landing point as $\phi_{k,i+1}^{-}$ (where $i+1$ is taken mod $n$).  
    It follows that $\displaystyle\lim_{k\to \infty} \phi_{k,i}^\pm = \theta_i$, completing the proof.
\end{proof}

\begin{proof}[Proof of Proposition~\ref{rat_lam_prop}]
This follows from Proposition~\ref{rat_lam_density_prop} and the discussion preceding it.
\end{proof}

\section{Renormalization combinatorics and associated\\ Multibrot and Multicorn sets}\label{multibrot_multicorn_copies_sec}

In this section we will study the set of (possibly pinched) renormalizable parameters, organized by their combinatorics (see Definition~\ref{renormalizability_def} for the notion of renormalizability).

\begin{definition}
	We will say that two renormalizable maps $\sigma_a, \sigma_{a'}$ have the same \emph{first-level renormalization combinatorics} if for all $k$, $n$ we have that $\vartheta^k_n(a)=\vartheta^k_n(a')$ (in particular $\cT^k(a)=\cT^k(a')$ for all $k$).
\end{definition}

Recall that the degree $n$ Multibrot/Multicorn set $\cM_n^\pm$ is the connectedness locus of degree $n$ unicritical holomorphic/antiholomorphic polynomials (see \cite{EMS16,NS,MNS} for basic properties of these sets).

\begin{theorem}\label{multicorn_thm}
Let $\sigma_a$ be renormalizable of minimal renormalization period $p$, and let the renormalizable critical point $w_i$ have degree $d_i$. Further suppose that the orbit of the free critical value $y_f$ does not always lie in $\Omega_1^1(a)$, the preimage of $\Omega_a$ having $y_c$ on its boundary. Then there is a compact set $\mathcal M_a\subset \cC(S_\cT)$ consisting of all parameters with the same first-level renormalization combinatorics as $\sigma_a$ such that
    \begin{enumerate}[leftmargin=8mm]
        \item if $p$ is even then $\mathcal M_a$ is homeomorphic to a degree $d_i$ Multibrot set $\cM_{d_i}^+$, and
        \item if $p$ is odd then $\mathcal M_a$ is a combinatorial copy of a degree $d_i$ Multicorn~$\cM_{d_i}^-$. More precisely, there exists a dynamically natural bijection between the closure of the hyperbolic components in $\cM_a$ and the closure of the hyperbolic components of $\cM_{d_i}^-$ that is continuous at hyperbolic and quasiconformally rigid parameters.
    \end{enumerate}
\end{theorem}

\begin{remark}
    The even period case of the above theorem is almost classical at this point.
\end{remark}

\begin{theorem}\label{PinchedMulticorn}
	The parameters with constant depth--$1$ kneading sequence $\vartheta^1_n(a)=1$ (i.e., $\sigma_a^{\circ n}(y_f)\in\overline{\Omega^1_{1,a}}$), for all $n\geq 0$, form a set combinatorially equivalent to a parabolic Multicorn $\cM_{d_1}^{\mathrm{par},-}$.
\end{theorem}
Here, the parabolic Multicorn $\cM_{d_1}^{\mathrm{par},-}$ is the unicritical slice of $\cF_{d_1}$ (i.e., maps in $\cF_{d_1}$ whose filled Julia sets contain a unique critical point, see \S~\ref{para_anti_rat_subsec}), where $d_1$ is the valence of the free critical point $w_1$ (which is adjacent to $y_c$) in~$\cT$.

\begin{remark}
    Throughout this section we will consider only the first level of renormalization, since by the above theorems this reduces our situation to the classical unicritical cases. These have been studied in \cite{DH2,IK12} in the holomorphic case and \cite{IM21,IM22} in the anti-holomorphic case.
\end{remark}

\begin{proposition}\label{renorm_center_prop}
	For a given renormalization combinatorics of period $p$, there exists a unique parameter with the same combinatorics which has a periodic critical point of period~$p$.
\end{proposition}

\begin{proof}
    This follows immediately from Proposition~\ref{knead_pcf_prop}, after noting that renormalizable parameters have a dessin with a periodic kneading sequence.
    \end{proof}

\subsection{Non-pinched renormalizations}

We will begin our discussion with the case that the critical value does not remain in $\Omega^1_{1,a}$. Fix a renormalization combinatorics of renormalization period $p$ and let $a_0$ be the parameter given by Proposition~\ref{renorm_center_prop}. We denote $\mathcal{M}_{a_0}$ as those parameters with the same renormalization combinatorics.

For each $k\ge 0$ let $\Theta_k:=\bigcup_{n=0}^{k}\cR_d^{-n}(\{0\})$ be the set of
\emph{pre--cuspidal angles of depth $\le k$}, and let $\overline{\mathfrak{R}_\theta}$
denote the parameter ray $\mathfrak{R}_\theta$ together with its landing parameter
(Definition~\ref{para_ray_schwarz}, Propositions~\ref{preper_para_rays_prop},
\ref{zero_rays_prop}). We define the \emph{depth$-k$ pre--cuspidal island corresponding to the parameter $a_0$}, denoted $\mathcal{I}_k(a_0)$, to be the closure in $\overline{S_\cT}$ of the connected
component of
$$
S_\cT\ \setminus\ \bigcup_{\theta\in\Theta_k}\overline{\mathfrak{R}_\theta}
$$
containing $a_0$.
The islands $\mathcal{I}_k(a_0)$ are nested; i.e., $\mathcal{I}_{k+1}(a_0)\subset\mathcal{I}_k(a_0)$, since $\Theta_k\subset\Theta_{k+1}$.
We then set

$$\cM_{a_0}\ :=\ \Bigl(\bigcap_{k\geq 1}\mathcal{I}_k(a_0)\Bigr)\bigcap\cC(S_\cT).$$

\begin{proposition}\label{wake_equals_renorm_prop}
$\cM_{a_0}$ consists precisely of the parameters $a\in\cC(S_\cT)$ having the same
first-level renormalization combinatorics as $a_0$; that is,
$\vartheta^k_n(a)=\vartheta^k_n(a_0)$ for all $k,n$.
\end{proposition}

\begin{proof}
Fix $a\in\cC(S_\cT)$. By Lemma~\ref{wake_lem} (cf. the argument in the proof of Theorem~\ref{comb_model_thm}), $a$ lies in $\bigcap_k\mathcal{I}_k(a_0)$ if and only if $a$
cannot be separated from $a_0$ by any parameter ray pair at pre--cuspidal angles, which holds
if and only if $\sigma_a$ and $\sigma_{a_0}$ have the same pre--cuspidal lamination of every
depth. By the duality between the depth--$k$ pre--cuspidal lamination and the depth--$k$
dessin (see \S~\ref{renorm_sec}), this is equivalent to requiring $\cT^k(a)=\cT^k(a_0)$ for all $k$.

It remains to match the critical itinerary. By the standard correspondence between dynamical and parameter rays, $a$ and $a_0$ are unseparated by the depth--$k$ pre--cuspidal parameter rays exactly when $y_f$ occupies the same depth--$k$ puzzle piece for $\sigma_a$ and $\sigma_{a_0}$, i.e. $\vartheta^k_n(a)=\vartheta^k_n(a_0)$ for all $n$. Letting $k\to\infty$ proves the claim.
\end{proof}

By construction, $\cM_{a_0}$ is compact. Further, all parameters in $\cM_{a_0}$ have one  renormalizable critical point $w_i(a)$, with the address of $w_i(a)$ in $\cT$ depending only on the renormalization combinatorics. Let $U(a)$ be the component of $\sigma_a^{-p}(\Omega_a)$ containing the free critical value $y_f$. As the orbit of $y_f$ does not remain in the component $\Omega^1_{1,a}$ of $\sigma_a^{-1}(\Omega_a)$ adjacent to the cusp, $U(a)$ is compactly contained in $\Omega_a$, and $\sigma_a^{\circ p}\colon U(a)\to \Omega_a$ is a family of (anti-)holomorphic, unicritical, polynomial-like maps of degree $d_i=\mathrm{val}(w_i)$, where $\mathrm{val}(w_i)$ is the valence of $w_i$ in $\cT$.

Following \cite[\S 5]{IM21}, for each $a\in \mathcal{M}_{a_0}$ there exists a quasiconformal map $\rho_a\colon U(a)\to \C$ which is a hybrid conjugacy between the polynomial-like maps $\sigma_a^{\circ p}\vert_{U(a)}$ and a unicritical polynomial: if $p$ is even the polynomial will be holomorphic and of the form $g_{c_a}: z\mapsto z^{d_i}+c_a$, and if $p$ is odd the polynomial is anti-holomorphic and of the form $g_{c_a}: z\mapsto \overline z^{d_i} + c_a$. The parameter $c_a$ here is not unique; it is only determined up to a $d_i\pm 1$ ($+$ when $p$ is odd, $-$ when $p$ is even) order symmetry. To fix the polynomial to which we straighten, we will first canonically mark an access to a fixed point of the polynomial-like map.

For the parameter $a_0$ (produced in Proposition~\ref{renorm_center_prop}), the free critical point $w_i(a_0)$ is $p-$periodic. Consider the Fatou component $\mathcal{W}$ of $\sigma_{a_0}$ containing $w_i(a_0)$, so $\sigma_{a_0}^{\circ p}\vert_{\mathcal{W}}$ is conformally conjugate to $\overline{z}^p\vert_\D$ or $z^p\vert_\D$ (depending on the parity of $p$). There exists a unique fixed point of $\sigma_{a_0}^{\circ p}$ on $\partial\mathcal{W}$ that  separates the renormalizable critical point $w_i(a_0)$ from the cusp $y_c$. Now consider the two dynamical rays of $\sigma_{a_0}$ landing at this repelling periodic point that are consecutive in the cyclic order around their landing point, and such that the counter-clockwise sector enclosed by these rays (and their landing point) contains $w_i(a_0)$ but not $y_c$. 

Throughout $\cM_{a_0}$ the dynamical rays chosen above land at a common repelling periodic point, and separate the small filled Julia set of the associated (anti-)polynomial-like map from $y_c$. These rays define a preferred access to a distinguished periodic point of the (anti-)polynomial-like maps in question. To define the straightening map, we always normalize the hybrid conjugacy so that this access is sent to the dynamical $0$-ray of the straightened polynomial.
This gives rise to a well-defined \emph{straightening map} 
$$
\chi_{a_0}\colon \cM_{a_0}\to \mathcal M_{d_i}^*
$$
by applying a dynamical straightening to each parameter in $\cM_{a_0}$.

\begin{proposition}\label{surj_pcf_prop}
The image of the straightening map $\chi_{a_0}(\cM_{a_0})$ contains all postcritically finite hyperbolic parameters of $\mathcal M_{d_i}^*$. 
\end{proposition}
\begin{proof}
    This follows immediately from applying Lemma~\ref{tuning_trees_lem} to the parameter $\sigma_{a_0}$ constructed in Proposition~\ref{renorm_center_prop}. It is easy to see that tunings of a critically periodic parameter will have the same first-level renormalization combinatorics as it.
\end{proof}

\begin{proposition}\label{str_image_qc_closed_prop}
The image of the straightening map $\chi_{a_0}(\cM_{a_0})$ is closed under quasiconformal conjugations.
\end{proposition}

\begin{proof}
Let $g_c\in\chi_{a_0}(\cM_{a_0})$ and $g_{c'}\in\cM_{a_0}$ be quasiconformally conjugate to $g_c$ via some quasiconformal map $\phi$ that is conformal on the basin of infinity. Then, $\mu:=\overline{\partial}\phi/\partial\phi$ is an invariant Beltrami coefficient supported on filled Julia set of $g_c$. Let $g_c=\chi_{a_0}(\sigma_a)$. Pull $\mu$ back under the hybrid conjugacy $\rho_a$ to the small filled Julia set of the (anti-)polynomial-like restriction of $\sigma_a$, and spread $\rho_a^\ast(\mu)$ to the entire grand orbit of this small filled Julia set by pulling it back iteratively under $\sigma_a$. Finally, extend this Beltrami coefficient to $\widehat{\C}$ by setting it equal to zero elsewhere. This produces a $\sigma_a-$invariant Beltrami coefficient on $\widehat{\C}$.
Then we use the Measurable Riemann Mapping Theorem and Proposition~\ref{qc_closed_prop} to find another $\sigma_{a'}\in\cM_{a_0}$ which straightens to the map $g_{c'}$.
\end{proof}

\begin{proof}[Proof of Theorem \ref{multicorn_thm}]
\textbf{Injectivity.} 
The proof of \cite[Theorem~B]{IK12}, which uses a pullback argument (c.f. \cite[\S~38.5]{Lyu25}), applies mutatis mutandis to the current setting establishing injectivity of the straightening map $\chi_{a_0}$. 
\smallskip

\noindent\textbf{Almost surjectivity.} It is easy to extend Proposition~\ref{surj_pcf_prop} to show that the image of $\chi_{a_0}$ contains all hyperbolic components by using the uniformization of hyperbolic parameters from \S~\ref{hyp_comp_unif_subsec}. Now suppose that $g_{c_n}\in \mathcal M_{d_i}^*$ are hyperbolic parameters in the associated Multibrot/Multicorn set with $c_n\to c$, and let $a_n\in \cM_{a_0}$ be the preimages under $\chi_{a_0}$. After passing to a subsequence, we assume that $a_n\to a\in \cM_{a_0}$. Denote the associated hybrid conjugacies as $\rho_{a_n}$, extending them as necessary so that they all share a common domain, and let $\rho_a$ be the straightening map for $\sigma_a$. Note that the domains $U(a)$ and $\Omega_a$ all vary continuously, and so the modulus of the annulus $\Omega_a\setminus \overline{U(a)}$ is bounded below. This tells us that the hybrid conjugacies $\rho_{a_n}$ can be chosen to have uniformly bounded dilatation, so in particular, after passing to a subsequence, there is a quasiconformal limit $\rho_{a_n}\to \widetilde{\rho}$ (cf. \cite[\S II.7, Lemma, Page 313]{DH2}). It follows that $g_c=\widetilde{\rho} \circ \sigma_a \circ \widetilde{\rho}^{-1} = \widetilde \rho\circ \rho^{-1} \circ \chi_{a_0}(a) \circ \rho\circ \widetilde\rho^{-1}$ on appropriate domains. As the image of the straightening map is quasiconformally closed by Proposition~\ref{str_image_qc_closed_prop}, it follows that $g_c\in \chi_{a_0}(\cM_{a_0})$.
\smallskip

\noindent\textbf{Continuity properties of $\chi_{a_0}$.}
Note that hybrid conjugacies preserve the conformal class of the dynamics on the interior of filled Julia sets, and hence $\chi_{a_0}$ respects the uniformizations of the hyperbolic components of $\cM_{a_0}$ (see \S~\ref{hyp_comp_unif_subsec}) and of the hyperbolic components of $\cM_{d_i}^*$ (cf. \cite{Dou83}, \cite[Theorem~4.7]{BF14}, \cite{NS}). It follows that $\chi_{a_0}$ induces a homeomorphism between the hyperbolic parameters of $\cM_{a_0}$ and those of $\cM_{d_i}^*$.

Continuity of the straightening map $\chi_{a_0}$ at quasiconformally rigid parameters
follows from the fact that if a sequence $\{a'_n\}\subset \cM_{a_0}$ converges to $a'_\infty$, then all accumulation points of $\chi_{a_0}(a'_n)$ in $\cM_{d_i}^*$ are quasiconformally conjugate to $\chi_{a_0}(a'_\infty)$ (see \cite[\S II.7, Lemma, Page 313]{DH2}, cf. \cite[\S 6]{PartI}).
\smallskip

\noindent\textbf{Homeomorphism in the even period case.}
It is well-known that straightening maps from full families of unicritical (holomorphic) polynomial-like maps are homeomorphisms onto Multibrot sets of appropriate degree. We refer the reader to \cite{DH2}, \cite[Theorem~9.3]{IK12} for details.
\end{proof}

Let us now fix a white vertex $\widehat{v}_w$ of $\cT$ of valence greater than one.
The proof of Lemma~\ref{dyn_shabat_existence_lem} (or Theorem~\ref{poi_thm}) and the David surgery of Theorem~\ref{Thm_C_part_I} can be combined to prove the existence of a parameter $\widehat{a}\in\cC(S_\cT)$ such that in the dynamical plane of $\sigma_{\widehat{a}}$, the corresponding free critical point $w_i(\widehat{a})$ is a superattracting fixed point; i.e., $w_i(\widehat{a})=\sigma_{\widehat{a}}(w_i(\widehat{a}))=y_f$. In fact, the parameter $\widehat{a}$ is unique by the rigidity of postcritically finite parameters. Evidently, the map $\sigma_{\widehat{a}}$ is renormalizable at the critical point $w_i(\widehat{a})$ with (minimal) renormalization period $1$. Let $\widehat{\cM}\subset\cC(S_\cT)$ be the set of
all parameters with the same first-level renormalization combinatorics as~$\sigma_{\widehat{a}}$. Theorem~\ref{multicorn_thm} now implies that if $\widehat{v}_w$ is not adjacent to the root $v_b$ in the tree $\cT$, then $\widehat{\cM}$ is a combinatorial copy of the degree $\widehat{d}$ Multicorn $\cM_{\widehat{d}}^-$, where $\widehat{d}$ is the valence of the white vertex $\widehat{v}_w$ in $\cT$.
We record this fact in the following corollary.

\begin{corollary}\label{special_copies_cor}
For each white vertex $\widehat{v}_w$ of $\cT$ of valence greater than one, except for the vertex $v_w'$ (that is adjacent to the root $v_b$), the connectedness locus $\cC(S_\cT)$ contains a combinatorial copy $\widehat{\cM}$ of the degree $\widehat{d}$ Multicorn $\cM_{\widehat{d}}^-$, where $\widehat{d}$ is the valence of the white vertex $\widehat{v}_w$ in $\cT$. Further, the map $\overline{z}^{\widehat{d}}\in\cM_{\widehat{d}}^-$ corresponds to the unique parameter $\widehat{a}\in\cM$ such that in the dynamical plane of $\sigma_{\widehat{a}}$, the free critical point $w_i(\widehat{a})$ corresponding to $\widehat{v}_w$ is a superattracting fixed point; i.e., $w_i(\widehat{a})=y_f$.
\end{corollary}

\subsection{Pinched period 1 renormalizations}

In this subsection we consider the case considered by Theorem~\ref{PinchedMulticorn}. It is no longer true that $U(a)=\Omega^1_{1,a}$ is compactly contained in $\Omega_a$, and so the classical straightening theorem does not apply. However so long as the parameter $a$ does not lie in $\Gamma^\mathrm{hoc}$, the mapping $\sigma_a: \overline{U(a)}\to \overline{\Omega_a}$ can be restricted to a \textit{simple} pinched polynomial-like mapping, and the straightening theorem \cite[Theorem~4.8]{PartI} applies. Those parameters lying on $\Gamma^{\mathrm{hoc}}$ will be mapped bijectively to the parameters of $\cM_{d_1}^{\mathrm{par},-}$ with higher order parabolics.

We will refer to the set of parameters for which $\sigma^{\circ n}_a(y_f)\in \overline{\Omega^1_{1,a}}$, for all $n$, as $\cM_{\mathrm{cusp}}$, and denote by $K^1(\sigma_a)$ and $J^1(\sigma_a)$ the little filled Julia set and the little Julia set for the corresponding period $1$ pinched polynomial-like restrictions.

\begin{proposition}\label{meas_attr_prop}
Let $\sigma_a\in \cM_{\mathrm{cusp}}$. Almost every point $z\in K(\sigma_a)$ with respect to Lebesgue measure eventually lands in the little filled Julia set $K^1(\sigma_a)$.
\end{proposition}

\begin{proof}
By the main result of \cite{PartI} every such Schwarz reflection is quasiconformally conjugate to an anti-rational map. For a rational map, the orbit of almost every point on the Julia set converges to the $\omega-$limit sets of the critical points (see \cite{Lyu82}, \cite[Theorem 22.2]{Lyu25}). As quasiconformal maps preserve the Lebesgue measure class this property remains true for $\sigma_a$. Thus, the orbit of almost every point in $\Lambda(\sigma_a)$ converges to some point in the little Julia set $J^1(\sigma_a)$.

If $a\notin \Gamma^{\mathrm{hoc}}$ then no point converges to $y_c$ except for its iterated preimages. As the intersection $\Lambda(\sigma_a)\cap \partial U(a)$ consists exclusively of pre--cuspidal points, no points converge to them. Therefore Lebesgue--almost every point of $K(\sigma_a)$ eventually lands inside of $U(a)$ and remains within it under iteration (as their orbits converge to $K^1(\sigma_a)$), and thus almost every point of $K(\sigma_a)$ lies in some preimage of $K^1(\sigma_a)$.

On the other hand if $a\in \Gamma^{\mathrm{hoc}}$ then every critical point is either pre-periodic or lies in the preimage of a parabolic basin. It follows that the limit set $\Lambda(\sigma_a)$ has zero measure, and so almost every point of $K(\sigma_a)$ lies in the preimage of the parabolic basin contained in $K^1(\sigma_a)$.
\end{proof}

\begin{proposition}
    The renormalization locus $\cM_{\mathrm{cusp}}$ is compact.
\end{proposition}

\begin{proof}
    Let $a\notin \cM_{\mathrm{cusp}}$, so that there is some smallest $n\geq 0$ for which $\sigma_a^{\circ n}(y_f)\notin \overline{\Omega^1_{1,a}}$. Now note that the boundaries of the components of $\sigma_a^{-n}(\Omega_a)$ are real-analytic curves with coefficients depending real-analytically on $a$. It follows that $\sigma_{a'}^{\circ n}(y_f)\notin \overline{\Omega^1_{1,a'}}$ for $a'$ sufficiently close to $a$. Therefore $\cM_{\mathrm{cusp}}$ is closed, and hence compact by Lemma~\ref{para_space_bdd_lem}.
\end{proof}

\begin{proof}[Proof of Theorem~\ref{PinchedMulticorn}]
If $a\in \Gamma^{\mathrm{hoc}}$, then we will map $a$ to the Belyi anti-rational map $R\in \cM_{d_1}^{\mathrm{par},-}$ for which $\infty$ is a double parabolic fixed point, and such that the {\'E}calle height of the unique critical point of $R$ in $\cK(R)$ is equal to that of $y_f$ (see Lemma~\ref{hoc_arc_lem}.

If $a\notin \Gamma^{\mathrm{hoc}}$ then by Theorem~\ref{simp_restrict_thm}, the map $\sigma_a: \overline{U(a)}\to \overline{\Omega_a}$ admits a simple pinched polynomial-like restriction, and there exists a map $\rho_a\colon \overline{U(a)}\to \widehat{\C}$ which gives a hybrid conjugacy between the said simple pinched polynomial-like map and restrictions of maps in $\cM_{d_1}^{\mathrm{par},-}$. We will denote by $\chi\colon \cM_{\mathrm{cusp}}\to \cM_{d_1}^{\mathrm{par},-}$ the induced straightening map. 
Also recall that for each parameter $a\in \cC(S_\cT)$ there is the external uniformization $\psi_a\colon \mathcal{Q}\longrightarrow T^\infty(\sigma_a)$, conjugating the anti-Farey map $\cR_d$ to $\sigma_a$. 
\smallskip

\noindent\textbf{Injectivity.} 
Suppose that $\chi(\sigma_{a_1})=\chi(\sigma_{a_2})$. If $a_1\in\Gamma^{\mathrm{hoc}}$, then the definition of $\chi$ and Lemma~\ref{hoc_arc_lem} imply that $a_1=a_2$. Hence, we may assume that $a_1\notin\Gamma^{\mathrm{hoc}}$.

By the proofs of \cite[Theorem~4.8, Lemma~4.13]{PartI}, there exist, for each $i\in\{1,2\}$, a Jordan domain $V_{i}\subset U(a_i)$  and a global quasiconformal map $\beta_i$ such that
\begin{enumerate}[leftmargin=8mm]
    \item $\overline{V_i}\subset\overline{U(a_i)}$ and $\overline{\sigma_{a_i}(V_i)}\subset\overline{\Omega_{a_i}}$ are polygons touching only at the cusp $y_c$,
    \item the boundaries $\partial\sigma_{a_i}(V_i)$, $\partial V_{i}$ subtend positive angles at $y_c$, 
     \item $\beta_i$ is conformal on $\overline{\sigma_{a_i}(V_i)}$, and
    \item $\beta_i$ conjugates $\sigma_{a_i}:\overline{V_{i}}\to\overline{\sigma_{a_i}(V_{i})}$ 
to a unicritical pinched anti-polynomial-like map.
\end{enumerate} 
As $\chi(\sigma_{a_1})=\chi(\sigma_{a_2})$, there are global quasiconformal homeomorphisms $\rho_{1}, \rho_{2}$ that hybrid conjugate the above pinched anti-polynomial-like maps to the same member of $\cM_{d_1}^{\mathrm{par},-}$.
By choosing the angles of $\partial V_{i}$ at $y_c$ appropriately, we can assume that the quasiconformal map $\beta_2^{-1}\circ\rho_{2}^{-1}\circ\rho_{1}\circ\beta_1$, which carries $\partial V_1$ to $\partial V_2$, is asymptotically linear near $y_c$.

By \cite[Lemma~B.2]{LLMM2}, the homeomorphic extension of the conformal map $\psi_{a_2}\circ\psi_{a_1}^{-1}:\widehat{\C}\setminus\Omega_{a_1}\to\widehat{\C}\setminus\Omega_{a_2}$ is also asymptotically linear on $\partial\Omega_{a_1}$ near $y_c$. The above properties allow us to define a quasiconformal homeomorphism $\Phi_0$ of the sphere that agrees with $\beta_2^{-1}\circ\rho_{2}^{-1}\circ\rho_{1}\circ\beta_1$ on $\overline{V_1}$, and agrees with $\psi_{a_2}\circ\psi_{a_1}^{-1}$ outside $\Omega_1$. We note that $\bar \partial \Phi_0 = 0$ almost everywhere on $K^1(\sigma_{a_1})$.

Having the map $\Phi_0$ at our disposal, we can employ the pullback argument of \cite[Theorem~B]{IK12} to produce a sequence of $K-$quasiconformal maps $\{\Phi_n\}$ (specifically, $\Phi_n$ is a normalized lift of $\Phi_{n-1}$ under the branched coverings $\sigma_{a_1}, \sigma_{a_2}$) such that 
\begin{itemize}[leftmargin=8mm]
\item $\Phi_n$ is conformal on $\displaystyle\bigcup_{i=0}^n\sigma_{a_1}^{-i}\Big(K^1(\sigma_{a_1})\cup T^0(\sigma_{a_1})\Big)$, and

\item $\Phi_n=\Phi_{n-1}$ on $\displaystyle\bigcup_{i=0}^{n-1}\sigma_{a_1}^{-i}\Big(K^1(\sigma_{a_1})\cup T^0(\sigma_{a_1})\Big)$.
\end{itemize}
By compactness of $K-$quasiconformal maps, we obtain a subsequential limit $\Phi_\infty$ that is conformal on $\displaystyle T^\infty(\sigma_{a_1})\bigcup\left(\bigcup_{i=0}^\infty\sigma_{a_1}^{-i}\big(K^1(\sigma_{a_1})\big)\right)$ and conjugates $\sigma_{a_1}$ to $\sigma_{a_2}$. By Proposition~\ref{meas_attr_prop}, $K(\sigma_{a_1})\setminus \bigcup_{i=0}^\infty\sigma_{a_1}^{-i}\big(K^1(\sigma_{a_1})\big)$ has zero Lebesgue measure zero; i.e., the quasiconformal map $\Phi_\infty$ is conformal almost everywhere. By Weyl's lemma, $\Phi_\infty$ is a M{\"o}bius conjugacy between $\sigma_{a_1}$ and $\sigma_{a_2}$; and hence~$a_1=a_2$.
\smallskip

\noindent\textbf{Almost surjectivity and continuity properties.} The facts that the image of $\chi$ contains the closure of the hyperbolic components in $\cM_{d_1}^{\mathrm{par},-}$ and that $\chi$ is continuous at hyperbolic and quasiconformally rigid parameters follow from the same arguments as in Theorem~\ref{multicorn_thm} together with the results of \cite{PartI}.
\end{proof}

\section{Model of the connectedness locus $\cC(S_\cT)$}\label{conn_comb_model_sec}

\subsection{Combinatorial classes and parameter laminations}
Recall from \S~\ref{renorm_sec} that the rational lamination of $\sigma_a\in\cC(S_\cT)$ is an equivalence relation on $\cR_d-$pre-periodic angles in $\mathbb{S}^1\cong\R/\Z$ where two angles are identified if the corresponding dynamical rays of $\sigma_a$ land at the same point of $\Lambda(\sigma_a)$.

\begin{definition}
\noindent\begin{enumerate}[leftmargin=8mm]
\item Two parameters $a$ and $a'$ in $\cC(S_\mathcal T)$ are said to be {\em combinatorially equivalent} if they have the same rational lamination.
\item The combinatorial class $\text{Comb}(a)$ of $a\in \cC(S_\mathcal T)$ is defined as the set of all parameters in $\cC(S_\mathcal T)$ which are combinatorially equivalent to $a$.
\item A combinatorial class $\text{Comb}(a)$ is called {\em periodically repelling} if for every $a'\in \text{Comb}(a)$, each periodic orbit of $\sigma_{a'}$ is repelling.
\end{enumerate}
\end{definition}

\begin{proposition}
Two parameters $a$ and $a'$ in $\mathcal C(S_\mathcal T)$ are combinatorially equivalent if and only if $\chi(a),\chi(a')\in \mathfrak{L}_\cT$ are so.
\end{proposition}

\begin{proof}
This follows from the hybrid equivalence between $\sigma_a$ and $\chi(\sigma_a)$.
\end{proof}

The following result can be proved by adapting the arguments of \cite[Proposition~8.4]{Sch04} to the current setting and taking the results of \S~\ref{hyp_comp_bdry_subsec} into consideration.

\begin{proposition}\label{comb_class_prop}
Every combinatorial class $\text{Comb}(a)$ of $\mathcal C(S_\mathcal T)$ is of one of the following types.
\begin{itemize}[leftmargin=6mm]
\item $\text{Comb}(a)$ consists of an even period hyperbolic component that does not bifurcate from an odd period hyperbolic component, its root point, and the irrationally neutral parameters on its boundary.
\item  $\text{Comb}(a)$ consists of an even period hyperbolic component that bifurcates from an odd period hyperbolic component, the unique parabolic cusp and the irrationally neutral parameters on its boundary.
\item  $\text{Comb}(a)$ consists of an odd period hyperbolic component and the parabolic arcs on its boundary.
\item  $\text{Comb}(a)$ is periodically repelling.
\end{itemize}

\end{proposition}

We define an equivalence relation $\sim$ on $\mathbb S^2$ as follows:

\begin{enumerate}[leftmargin=8mm]
\item identify all points in the closure of each periodically repelling combinatorial class of $\mathcal C(S_\mathcal T)$,
\item identify all points in the closure of the non-bifurcating sub-arc of each parabolic arc of $\mathcal C(S_\mathcal T)$, and
\item identify all points in $\Gamma^\text{hoc}$.
\end{enumerate}

All equivalence classes are connected and non-separating, and therefore by Moore's theorem the image of $\mathbb S^2$ under this relation is again $\mathbb S^2$.

\begin{remark}
    Conjecturally, periodically repelling combinatorial classes consist of single points. If this is true, it would imply that hyperbolic maps are dense in the parameter space (a version of the Fatou conjecture).
\end{remark}

\begin{definition}
The abstract connectedness locus of $S_\cT$ is defined as the image of $\mathcal C(S_\mathcal T)$ under the above equivalence relation, and will be denoted as $\widetilde{\cC_\cT}$.
\end{definition}

Within $\cC_\cT$ we may consider those parameters for which the free critical value eventually lands at the cusp. These parameters are combinatorially rigid, and we will refer to both these parameters and their images in the combinatorial model as \textit{pre--cuspidal parameters}.

The following proposition provides a parameter space analog to Proposition~\ref{rat_lam_density_prop}, and can be proved by adapting standard arguments from polynomial dynamics.

\begin{proposition}\label{cusp_par_dense_prop}
    Pre-cuspidal parameters are dense among non-renormalizable maps.
\end{proposition}
\begin{proof}[Sketch of proof]
Let $\sigma_{a_0}\in\cC(S_\cT)$ be a non-renormalizable map. Push forward the rational lamination of $\sigma_{a_0}$ under the circle homeomorphism conjugating $\cR_d$ to $\overline{z}^d$ (and sending $1$ to $1$). This yields a formal $m_{-d}-$invariant rational lamination, and by an antiholomorphic analog of \cite[Theorem~1.1]{Kiw01}, it is realized by some degree $d$ anti-polynomial $p_0$. We also choose a \emph{critical portrait} $\Theta_0$ for $p_0$, and note that the rational lamination of $p_0$ can be recovered from the critical portrait (cf. \cite[\S 6]{kiwi1}). We can now choose a sequence $\Theta_n$ of \emph{admissible critical portraits} converging to the critical portrait of $p_0$, such that the angles corresponding to the black critical points remain the same, while the angles corresponding to the white critical points are iterated $m_{-d}-$preimages of the angle $0$ (cf. \cite[Definition~6.23, Proposition~6.24]{kiwi1}). By \cite[Theorem~6.22]{kiwi1}, there exist post-critically finite degree $d$ anti-polynomials $p_n$ realizing the critical portraits $\Theta_n$. Such choices of $\Theta_n$ ensure that $p_n$ is a Shabat polynomial with dessin $\cT$ such that the black critical value is fixed and the white critical points eventually map to the black critical value. The David surgery construction of Theorem~\ref{Thm_C_part_I}, applied to the maps $p_n$, yields pre--cuspidal Schwarz reflections $\sigma_{a_n}\in\cC(S_\cT)$. Since the critical portraits $\Theta_n$ generate the rational laminations of $p_n$ (see \cite[Proposition~6.19]{kiwi1}) and $\Theta_n\to\Theta_0$, it follows that the rational laminations of $\sigma_{a_n}$ converge to that $\sigma_{a_0}$. By \cite[Theorem 7.15]{LLM24}, each accumulation point of $\sigma_{a_n}$ in $\cC(S_\cT)$ must have the same rational lamination as $\sigma_{a_0}$. But by the combinatorial rigidity of non-renormalizable maps (Theorem~\ref{fin_renorm_geom}), $\sigma_{a_0}$ is the only map with this rational lamination and so we must have $\sigma_{a_n}\to \sigma_{a_0}$.
\end{proof}

Let us give another description of the combinatorial model for the connectedness locus. Consider the equivalence relation on $\mathbb{S}^1$ defined so that two angles which are pre-periodic under $\cR_d$ are identified if the corresponding \textit{parameter rays} land at the same parabolic/Misiurewicz parameter or accumulate on the same parabolic arc (see Propositions~\ref{preper_para_rays_prop}) and~\ref{per_para_rays_prop}), then take the closure of this equivalence relation on $\mathbb{S}^1\times \mathbb{S}^1$. Equivalence classes are clearly unlinked as parameter rays do not intersect, and hence one can extend the relation to a lamination $\lambda_\mathrm{par}\subset \overline{\D}$.

\begin{theorem}\label{comb_model_thm}
There is a homeomorphism
$$
\overline{\D}/\lambda_\mathrm{par} \cong \widetilde{\cC_\cT}.
$$
\end{theorem}
\begin{remark}
We refer the reader to \cite{Dou92}, \cite[\S 8]{Sch04} for related statements for the Multibrot sets.
\end{remark}

To construct the desired homeomorphism between $\overline{\D}/\lambda_\mathrm{par}$ and $\widetilde{\cC_\cT}$, it is important to have the following wake lemma, detailing how parameter rays interplay with the bifurcation of certain dynamical behavior.

\begin{lemma}[Wake Lemma (c.f. \cite{MNS})]\label{wake_lem}
    Let $a_0$ be a parameter such that $\sigma_{a_0}$ has a repelling periodic point $z_0$ of period $p$ and the dynamical rays at angles $\theta_1,\theta_2,\cdots, \theta_v$ land at $z_0$.
    \begin{enumerate}[leftmargin=8mm]
        \item Then there is a neighborhood $U\subset S_\cT$ of $a_0$ and a unique real-analytic function $z\colon U\to \C$ with $z(a_0) = z_0$ so that for every $a \in U$ the point $z(a)$ is a repelling periodic point of period $p$ for $\sigma_a$, and all the dynamical rays at angles $\theta_1,\cdots,\theta_v$ land at $z(a)$.
        \item Let $n_i$ be the period of $\theta_i$ for $i=1,\cdots,v$. Define $F$ to be the set of all parameters $a$ such that $\sigma_a$ has a parabolic cycle, and some dynamical ray of period $n_i$ lands on the parabolic cycle of $\sigma_a$. Let $W$ be an open, connected neighborhood of $a_0$ which is disjoint from $F$ and from all parameter rays at angles in $\bigcup_{j\geq 0}\cR^j_d(\{\theta_1,\cdots,\theta_v\})$. Then for every parameter $a\in W$, the dynamical rays at angles $\theta_1,\cdots,\theta_v$ land at a common point.
    \end{enumerate}
\end{lemma}

\begin{proof}
    The proof \cite[Lemma~2.4]{MNS} relies only on the local dynamics of repelling periodic points and the analytic parameter dependence of the B\"ottcher coordinate. The same holds true in our setting, and so the proof applies mutatis mutandis to the above.
\end{proof}

\begin{proof}[Sketch of Proof of Theorem~\ref{comb_model_thm}]
    We begin by noting that there is a well-defined map from equivalence classes of $\lambda_\mathrm{par}$ having endpoints at periodic and pre-periodic angles (under $\cR_d$) to corresponding equivalence classes in $\widetilde{\cC_\cT}$, where leaves (including trivial ones) are sent to the equivalence classes (given by Misiurewicz parameters or parts of parabolic arcs) that the corresponding rays accumulate on, as guaranteed by Propositions~\ref{preper_para_rays_prop},~\ref{per_para_rays_prop}. We also note that by Proposition~\ref{zero_rays_prop}, both the $0/1$ and $1/1$ rays land on $\Gamma^\mathrm{hoc}$, and so we send $e^{2\pi i\cdot 0}=e^{2\pi i\cdot 1}\in \overline{\D}$ to $\Gamma^\mathrm{hoc}/\sim\ \in\widetilde{\cC_\cT}$.

    By Lemma~\ref{wake_lem}, if two parameters of $\cC(S_\cT)$ cannot be separated by parameter rays at $\cR_d-$pre-periodic angles, then these parameters have the same rational lamination. This fact allows us to define a map $\pmb{\Phi}:\overline{\D}/\lambda_{\mathrm{par}}\longrightarrow\widetilde{\cC_\cT}$ which extends the association defined in the previous paragraph.

To prove continuity of $\pmb{\Phi}$, suppose that a sequence of equivalence classes $\ell_n$ of $\lambda_{\mathrm{par}}$, consisting of $\cR_d-$pre-periodic angles, converge to some equivalence class $\ell_\infty$. Let $a_n$ be the co-landing Misiurewicz/parabolic point of the parameter rays $\mathfrak{R}_\theta$, $\theta\in\ell_n$, or a parabolic parameter on the parabolic arc where the parameter rays $\mathfrak{R}_\theta$, $\theta\in\ell_n$, co-accumulate. Then, the rational lamination of $\sigma_{a_n}$ is generated by the equivalence class $\ell_n$. By the proof of \cite[Theorem 7.15]{LLM24}, if $a_\infty$ is an accumulation point of $a_n$, then the free critical value of $a_\infty$ either lies in the impression of the corresponding dynamical rays at angles in $\ell_\infty$, or it lies in a parabolic Fatou component such that the corresponding dynamical rays at angles in $\ell_\infty$ land at the characteristic parabolic point.
Further, the rational lamination of $a_\infty$ is generated by $\ell_\infty$. It follows that the parameter rays $\mathfrak{R}_\theta$, $\theta\in\ell_\infty$, accumulate on the same combinatorial class. If this combinatorial class contains hyperbolic parameters, then $a_\infty$ is either a rigid neutral parameter, or lies on the non-bifurcating sub-arc of a parabolic arc. Otherwise, $a_\infty$ lies in a periodically repelling combinatorial class (see Proposition~\ref{comb_class_prop}). In all cases, we have that the classes $[a_n]\to[a_\infty]$ in $\widetilde{\cC_\cT}$. This proves continuity of the map $\pmb{\Phi}$.

By construction, two distinct equivalence classes of $\lambda_{\mathrm{par}}$ consisting of $\cR_d-$pre-periodic angles are mapped by $\pmb{\Phi}$ to distinct points in $\widetilde{\cC_\cT}$. Moreover, the infinite gaps of $\lambda_{\mathrm{par}}$ whose boundaries only consist of leaves having $\cR_d-$periodic endpoints, are mapped homeomorphically to the corresponding hyperbolic combinatorial classes. All other equivalence classes of $\lambda_{\mathrm{par}}$ are finite-sided polygons, and they are mapped to distinct periodically repelling combinatorial classes. This proves injectivity of $\pmb{\Phi}$.

Surjectivity of $\pmb{\Phi}$ follows from the fact that every point on the boundary of $\cC(S_\cT)$ lies in the impression of some parameter ray, and we can approximate this ray by parameter rays at $\cR_d-$pre-periodic angles.
As $\overline{\D}/\lambda_{\mathrm{par}}$ is compact and $\widetilde{\cC_\cT}$ is Hausdorff, it follows that $\pmb{\Phi}$ is a homeomorphism.
\end{proof}

Since $\overline{\D}/\lambda_\mathrm{par}$ is the quotient of a compact, locally connected space by a closed equivalence relation, it is locally connected. We then immediately have the following corollary.

\begin{corollary}
    The combinatorial model $\widetilde{\cC_\cT}$ is locally connected.
\end{corollary}

The understanding of $\widetilde{\cC_\cT}$ as a pinched disk gives the following `depth 1' understanding of the combinatorial model (see Figure~\ref{per_para_ray_misi_fig}).

\begin{corollary}\label{dessin_in_conn_locus_thm}
    There is an embedding of $\cT$ into $\widetilde {\cC_\cT}$, which maps the white vertices of $\cT$ which are of valence at least $2$ to the combinatorial classes associated with period one hyperbolic components, the white vertices of valence $1$ to Misiurewicz parameter tips of the connectedness locus (the free critical value $y_f$ is a repelling fixed point for such parameters), and the black vertices to the combinatorial classes associated with pre--cuspidal parameters.
\end{corollary}
\begin{proof}
This follows as the dessin $\cT$ is the tree dual to the first level pre--cuspidal lamination.
\end{proof}

A more refined description of the lamination $\lambda_\mathrm{par}$ can be given using Proposition~\ref{cusp_par_dense_prop}. One can explicitly compute the subset of $\lambda_\mathrm{par}$ given by rays landing at the pre--cuspidal parameters, by pulling back the persistently landing pre--cuspidal rays. 

We claim that the closure of this countable collection gives the parameter rays that land at non-renormalizable parameters and the parameter rays that separate the combinatorial (respectively, homeomorphic) copies of the Multicorn (respectively, Multibrot) sets from the rest of $\cC(S_\cT)$. Indeed, the tips of Multibrot/Multicorn sets are either combinatorially rigid Misiurewicz parameters, or combinatorially rigid parabolic arcs. By the proof of Proposition~\ref{cusp_par_dense_prop}, such parameters/arcs can be approximated by pre--cuspidal parameters such that the angles of the dynamical rays landing at the free critical value of the pre--cuspidal parameters converge to those of the dynamical rays landing at the free critical value of the limiting Misiurewicz parameter (or in the case of a limiting parabolic parameter, to the angles of the dynamical rays landing at the characteristic parabolic point). It now follows from Propositions~\ref{preper_para_rays_prop}) and~\ref{per_para_rays_prop} that the angles of the parameter rays landing at the approximating pre--cuspidal parameters converge to those of the parameter rays landing at the limiting Misiurewicz parameter (or to the angles of the parameter rays accumulating on the limiting parabolic arc).

Thanks to the above description, one can finally ``fill in" the gaps of this `parameter pre--cuspidal lamination' with copies of tunings of the laminations for appropriate Multibrot and Multicorn sets to produce the lamination $\lambda_{\mathrm{par}}$.

\section{Some explicit examples}\label{examples_sec}

In this final section, we will describe the Shabat polynomials $\pmb{f}$ appearing in the definition of the family $S_{\mathcal{T}}$ for some explicit choices of the rooted tree $(\mathcal{T},O)$.

\subsection{$\mathcal{T}$ is a star-tree}\label{unicritical_subsec} Suppose that $\mathcal{T}$ is a star-tree; i.e., $\mathcal{T}$ has a unique vertex of valence greater than one (colored white). Further let $O$ be any black vertex of $\mathcal{T}$. In this case, each $R\in\mathfrak{L}_{\mathcal{T}}$ has exactly two fully ramified critical points (one in the completely invariant parabolic basin $\mathcal{B}(R)$ and the other in the filled Julia set $\mathcal{K}(R)$). The dessin d'enfants $\mathcal{T}^{\mathrm{aug}}$ of the Shabat polynomial $\pmb{f}$ is given by attaching an edge to $\mathcal{T}$ at $O$.
\begin{figure}[ht!]
\captionsetup{width=0.98\linewidth}
\begin{tikzpicture}
\node[anchor=south west,inner sep=0] at (0.5,0) {\includegraphics[width=0.21\textwidth]{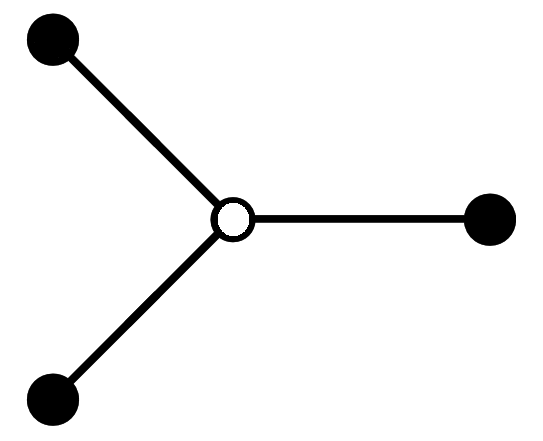}};
\node[anchor=south west,inner sep=0] at (5.5,0) {\includegraphics[width=0.32\textwidth]{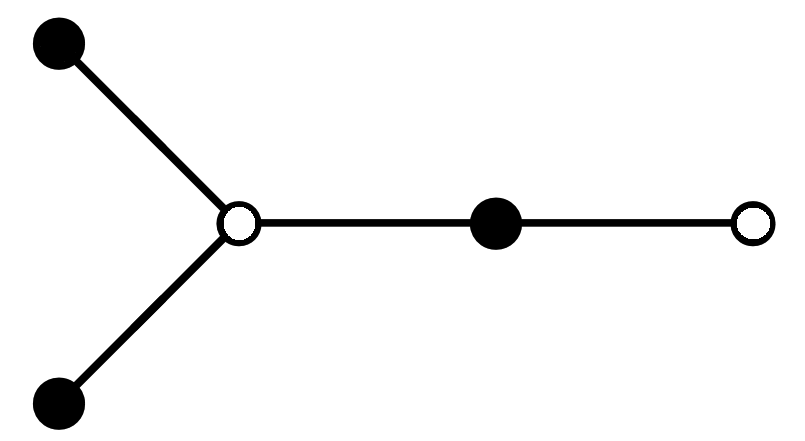}};
\node at (2.94,1.42) {$O$};
\node at (6.3,1.16) {$v_w'$};
\node at (8,1.4) {$v_b$};
\node at (9.32,1.4) {$v_w$};
\node at (2,0.28) {$\mathcal{T}$};
\node at (8,0.28) {$\mathcal{T}^{\mathrm{aug}}$};
\end{tikzpicture}
\caption{A $3-$star tree $\mathcal{T}$ and the corresponding augmented dessin $\mathcal{T}^{\mathrm{aug}}$ are displayed. $\mathcal{T}^{\mathrm{aug}}$ is realized by the Shabat polynomial $\pmb{f}(z)=z^3(4-z)$.}
\label{star_subtree_fig}
\end{figure}
See Figure~\ref{star_subtree_fig} for a $3-$star tree and an explicit formula for the corresponding Shabat polynomial $\pmb{f}$.

\subsection{$\mathcal{T}$ is a chain-tree}\label{chebychev_subsubsec} 
If $\mathcal{T}$ is a chain-tree (i.e., a tree with no branch point), then so is $\mathcal{T}^{\mathrm{aug}}$. Since chain-trees precisely correspond to equivalence classes of Chebyshev polynomials, it follows that $\pmb{f}$ can be chosen to be a Chebyshev polynomial in this case.

\subsection{Different roots for the same dessin $\mathcal{T}$}\label{root_matters_subsec}
Let $\mathcal{T}$ be the tree depicted on the top of Figure~\ref{root_matters_fig}. The two choices of the root point $O$ depicted in that figure yield two different rooted plane bicolored trees $(\mathcal{T},O)$, and hence the associated families of anti-rational maps $\mathfrak{L}_{\mathcal{T}}$ are dynamically different.
The corresponding augmented dessins $\mathcal{T}^{\mathrm{aug}}$ are shown at the bottom of the figure. We note that while these two augmented dessins are isomorphic as abstract trees, they are not homeomorphic as plane trees (i.e., there is no orientation-preserving homeomorphism of the sphere between these trees). In fact, the augmented dessins are realized by nonequivalent Shabat polynomials $\pmb{f}(z)=z^4(z^2-1)$ and $\pmb{f}(z)=z^4(z^2-2z+\frac{25}{9})$.

\begin{figure}[ht!]
\captionsetup{width=0.98\linewidth}
\begin{tikzpicture}
\node[anchor=south west,inner sep=0] at (0.5,0) {\includegraphics[width=0.8\textwidth]{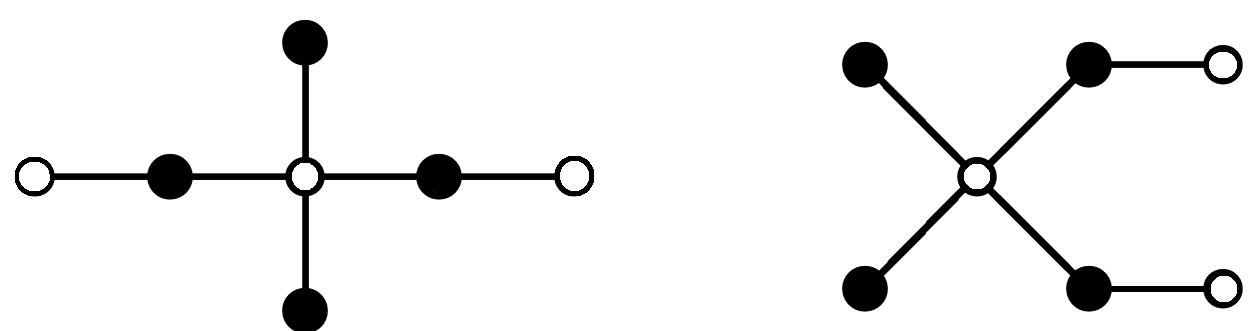}};
\node[anchor=south west,inner sep=0] at (4,2.8) {\includegraphics[width=0.28\textwidth]{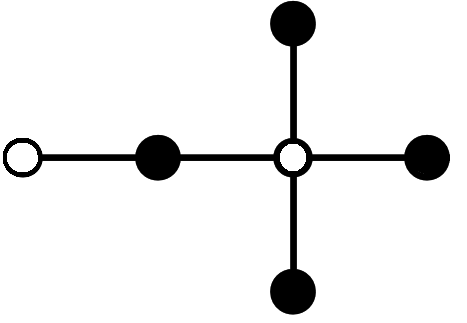}};
\node at (2.66,1.66) {\begin{large}$v_w'$\end{large}};
\node at (4,1.66) {\begin{large}$v_b$\end{large}};
\node at (5.08,1.66) {\begin{large}$v_w$\end{large}};
\node at (8.9,1.32) {\begin{large}$v_w'$\end{large}};
\node at (9.32,0.75) {\begin{large}$v_b$\end{large}};
\node at (10.36,0.75) {\begin{large}$v_w$\end{large}};
\node at (7.28,4.48) {$O$};
\node at (5.8,5.08) {$O$};
\node at (4.2,5.08) {$\mathcal{T}$};
\node at (2,0.28) {$\mathcal{T}^{\mathrm{aug}}$};
\node at (8.4,0.28) {$\mathcal{T}^{\mathrm{aug}}$};
\end{tikzpicture}
\caption{Top: The dessin $\mathcal{T}$ with two choices for the root point $O$ are shown. Bottom: The corresponding augmented dessins $\mathcal{T}^{\mathrm{aug}}$ and the vertices $v_b, v_w, v_w'$ are displayed.}
\label{root_matters_fig}
\end{figure}

\subsection{Different dessins leading to the same augmented dessin}\label{same_aug_tree_subsec} 

Consider the two rooted dessins $\mathcal{T}$ shown on the top of Figure~\ref{same_aug_tree_fig}. They correspond to two dynamically different families of Belyi anti-rational maps $\mathfrak{L}_{\mathcal{T}}$. The associated augmented dessins $\mathcal{T}^{\mathrm{aug}}$ are isomorphic (as plane bicolored trees), and are realized by Shabat polynomial $\pmb{f}(z)=z^3(z^2-2z+\alpha)^2$ (where $\alpha=\frac{1}{7}(34\pm 6\sqrt{21})$). However, the corresponding families of Schwarz reflections $S_\cT$ (which are dynamically different) are obtained from different choices of the black vertex $v_b$ on $\mathcal{T}^{\mathrm{aug}}$.

\begin{figure}[ht!]
\captionsetup{width=0.98\linewidth}
\begin{tikzpicture}
\node[anchor=south west,inner sep=0] at (-3.,4) {\includegraphics[width=0.4\textwidth]{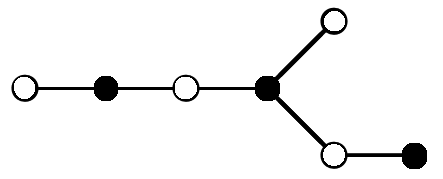}};
\node[anchor=south west,inner sep=0] at (3,4) {\includegraphics[width=0.41\textwidth]{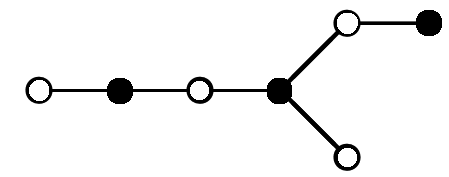}};
\node[anchor=south west,inner sep=0] at (0,0) {\includegraphics[width=0.4\textwidth]{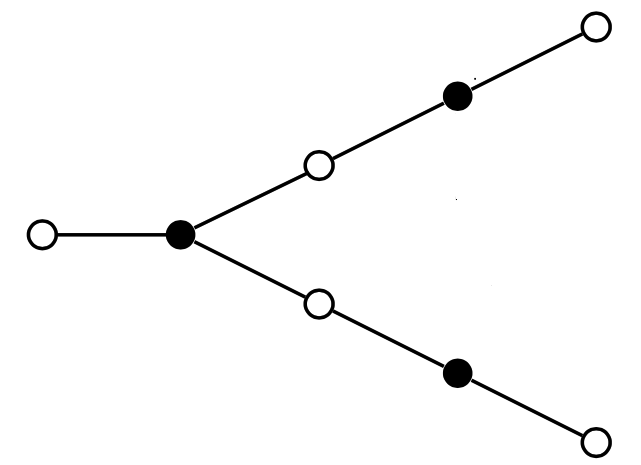}};
\node at (3.7,2.56) {\textcolor{red}{$v_b$}};
\node at (4.7,3.1) {\textcolor{red}{$v_w$}};
\node at (2.5,2.8) {\textcolor{red}{$v_w'$}};
\node at (3.7,1.06) {\textcolor{blue}{$v_b$}};
\node at (4.7,0.5) {\textcolor{blue}{$v_w$}};
\node at (2.5,0.84) {\textcolor{blue}{$v_w'$}};
\node at (1.75,4.6) {$O$};
\node at (8,5.4) {$O$};
\node at (2,0.28) {$\mathcal{T}^{\mathrm{aug}}$};
\node at (-1.66,4.28) {$\mathcal{T}$};
\node at (4.6,4.28) {$\mathcal{T}$};
\end{tikzpicture}
\caption{The rooted dessins $\mathcal{T}$ depicted on the top give rise to the same augmented dessin $\mathcal{T}^{\mathrm{aug}}$, but the corresponding black vertices $v_b$ are different.}
\label{same_aug_tree_fig}
\end{figure}

\end{document}